\documentclass[a4paper, 12pt, reqno]{amsart}

\pdfoutput=1

\usepackage[margin=2.5cm, headsep=1cm, footskip=1cm]{geometry}
\usepackage[hang, flushmargin]{footmisc}
\usepackage[english]{isodate}
\usepackage{amsmath, amsthm, amssymb, mathtools}
\usepackage{enumitem}
\usepackage{xcolor}
\usepackage{tikz, tikzsymbols, tikz-cd}
\usepackage{hyperref}
\usepackage[capitalise, nameinlink, noabbrev, nosort]{cleveref}

\hypersetup{pdfauthor={Becky Armstrong, Gilles G. de Castro, Lisa Orloff Clark, Kristin Courtney, Ying-Fen Lin, Kathryn McCormick, Jacqui Ramagge, Aidan Sims, and Benjamin Steinberg}, pdftitle={Reconstruction of twisted Steinberg algebras}, colorlinks=true, linkcolor=blue, linkbordercolor=blue, citecolor=red, citebordercolor=red, urlcolor=indigo, linktocpage=true}

\definecolor{indigo}{HTML}{492DA5}

\allowdisplaybreaks[2]
\setenumerate{listparindent=\parindent}

\providecommand{\noopsort}[1]{}

\makeatletter\g@addto@macro\bfseries{\boldmath}\makeatother

\makeatletter
\let\origsection\section
\renewcommand\section{\@ifstar{\starsection}{\nostarsection}}
\newcommand\sectionspace{\vspace{0.5ex}}
\newcommand\nostarsection[1]{\sectionspace\origsection{#1}\sectionspace}
\newcommand\starsection[1]{\sectionspace\origsection*{#1}\sectionspace}
\makeatother

\crefname{page}{page}{pages}

\setlist[enumerate]{font=\normalfont}
\crefname{enumi}{}{}
\crefname{enumii}{}{}

\numberwithin{equation}{section}
\crefname{equation}{equation}{equations}

\crefname{condition}{condition}{conditions}
\creflabelformat{condition}{#2(#1)#3}

\newtheorem{theorem}{Theorem}[section]

\crefname{thm}{Theorem}{Theorems}

\newtheorem{lemma}[theorem]{Lemma}
\crefname{lemma}{Lemma}{Lemmas}

\newtheorem{prop}[theorem]{Proposition}
\crefname{prop}{Proposition}{Propositions}

\newtheorem{cor}[theorem]{Corollary}
\crefname{cor}{Corollary}{Corollaries}

\theoremstyle{definition}

\newtheorem{definition}[theorem]{Definition}
\crefname{definition}{Definition}{Definitions}

\theoremstyle{remark}

\newtheorem{remark}[theorem]{Remark}
\crefname{remark}{Remark}{Remarks}

\newtheorem{remarks}[theorem]{Remarks}
\crefname{remarks}{Remarks}{Remarks}

\newtheorem{example}[theorem]{Example}
\crefname{example}{Example}{Examples}

\crefname{ADP}{}{}
\creflabelformat{ADP}{#2(ADP)#3}

\crefname{ACP}{}{}
\creflabelformat{ACP}{#2(ACP)#3}

\crefname{AQP}{}{}
\creflabelformat{AQP}{#2(AQP)#3}

\newcommand{\Cstar}{C*}
\newcommand{\C}{\mathbb{C}}
\newcommand{\N}{\mathbb{N}}
\newcommand{\R}{\mathbb{R}}
\newcommand{\T}{\mathbb{T}}
\newcommand{\Z}{\mathbb{Z}}
\newcommand{\BB}{\mathcal{B}}
\newcommand{\FF}{\mathcal{F}}
\newcommand{\GG}{\mathcal{G}}
\newcommand{\II}{\mathcal{I}}
\newcommand{\JJ}{\mathcal{J}}
\renewcommand{\SS}{\mathcal{S}}
\newcommand{\UU}{\mathcal{U}}
\newcommand{\VV}{\mathcal{V}}

\newcommand{\fp}{\mathfrak{p}}
\newcommand{\fq}{\mathfrak{q}}
\newcommand{\Go}{{G^{(0)}}}
\newcommand{\Gc}{{G^{(2)}}}
\newcommand{\Sigmao}{{\Sigma^{(0)}}}
\newcommand{\Sigmac}{{\Sigma^{(2)}}}
\newcommand{\vecspan}{\operatorname{span}}
\newcommand{\supp}{\operatorname{supp}}
\newcommand{\id}{\operatorname{id}}
\newcommand{\Iso}{\operatorname{Iso}}
\newcommand{\Spec}{\operatorname{Spec}}

\newcommand{\restr}[1]{\vert_{#1}}
\renewcommand{\r}{\boldsymbol{r}}
\newcommand{\s}{\boldsymbol{s}}

\begin{document}

\date{\today}
\title[Reconstruction of twisted Steinberg algebras]{Reconstruction of twisted Steinberg algebras}

\author[Armstrong]{Becky Armstrong}
\author[de Castro]{Gilles G. de Castro}
\author[Clark]{Lisa Orloff Clark}
\author[Courtney]{Kristin Courtney}
\author[Lin]{Ying-Fen Lin}
\author[McCormick]{Kathryn McCormick}
\author[Ramagge]{Jacqui Ramagge}
\author[Sims]{Aidan Sims}
\author[Steinberg]{Benjamin Steinberg}
\renewcommand{\shortauthors}{B.\ Armstrong et al.}

\address[B.\ Armstrong and K.\ Courtney]{Mathematical Institute, WWU M\"unster, Einsteinstr.\ 62, 48149 M\"unster, GERMANY}
\email{\href{mailto:becky.armstrong@uni-muenster.de}{becky.armstrong}, \href{mailto:kcourtne@uni-muenster.de}{kcourtne@uni-muenster.de}}
\address[L.O.\ Clark]{School of Mathematics and Statistics, Victoria University of Wellington, PO Box 600, Wellington 6140, NEW ZEALAND}
\email{\href{mailto:lisa.clark@vuw.ac.nz}{lisa.clark@vuw.ac.nz}}
\address[G.G.\ de Castro]{Departamento de Matem\'atica, Universidade Federal de Santa Catarina, Florian\'opolis, 88040-900, Brazil}
\email{\href{mailto:gilles.castro@ufsc.br}{gilles.castro@ufsc.br}}
\address[Y.-F.\ Lin]{Mathematical Sciences Research Centre, Queen's University Belfast, Belfast, BT7 1NN, UNITED KINGDOM}
\email{\href{mailto:y.lin@qub.ac.uk}{y.lin@qub.ac.uk}}
\address[K.\ McCormick]{Department of Mathematics and Statistics, California State University, Long Beach, CA, UNITED STATES}
\email{\href{mailto:kathryn.mccormick@csulb.edu}{kathryn.mccormick@csulb.edu}}
\address[J.\ Ramagge]{Faculty of Science, Durham University, Durham, DH1 3LE, UNITED KINGDOM}
\email{\href{mailto:jacqui.ramagge@durham.ac.uk}{jacqui.ramagge@durham.ac.uk}}
\address[A.\ Sims]{School of Mathematics and Applied Statistics, University of Wollongong, Northfields Avenue, Wollongong 2522, AUSTRALIA}
\email{\href{mailto:asims@uow.edu.au}{asims@uow.edu.au}}
\address[B.\ Steinberg]{Department of Mathematics, City College of New York, Convent Avenue at 138th Street, New York, New York 10031, USA.}
\email{\href{mailto:bsteinberg@ccny.cuny.edu}{bsteinberg@ccny.cuny.edu}}

\thanks{This research was supported by: the project-oriented workshop ``Women in Operator Algebras'' (18w5168) in November 2018, which was funded and hosted by the Banff International Research Station; the Australian Research Council Discovery Project DP200100155; the Sydney Mathematical Research Institute Domestic Visitor Program; the Marsden Fund of the Royal Society of New Zealand (grant number 18-VUW-056); CAPES-PrInt grant number 88887.368595/2019-00; the Deutsche Forschungsgemeinschaft (DFG, German Research Foundation) under Germany's Excellence Strategy -- EXC 2044 -- 390685587, Mathematics M\"unster -- Dynamics -- Geometry -- Structure; the Deutsche Forschungsgemeinschaft (DFG, German Research Foundation) – Project-ID 427320536 – SFB 1442; ERC Advanced Grant 834267 - AMAREC; PSC-CUNY; and the Fulbright Commission. We thank the anonymous referees for their careful reading and helpful suggestions.}

\subjclass[2020]{16S99 (primary), 22A22 (secondary)} \keywords{Steinberg algebra, groupoid, cohomology}

\begin{abstract}
We show how to recover a discrete twist over an ample Hausdorff groupoid from a pair consisting of an algebra and what we call a \emph{quasi-Cartan subalgebra}. We identify precisely which twists arise in this way (namely, those that satisfy the \emph{local bisection hypothesis}), and we prove that the assignment of twisted Steinberg algebras to such twists and our construction of a twist from a quasi-Cartan pair are mutually inverse. We identify the algebraic pairs that correspond to effective groupoids and to principal groupoids. We also indicate the scope of our results by identifying large classes of twists for which the local bisection hypothesis holds automatically.
\end{abstract}

\maketitle
\vspace{-3.75ex}
\tableofcontents
\vspace{-4.5ex}

\section{Introduction}

In this paper we provide a concrete representation theorem for a large class of abstract algebras by establishing a correspondence between algebraic twists over ample Hausdorff groupoids and a corresponding class of pairs $(A,B)$ of abstract algebras.

When trying to understand an algebra, it is often helpful to describe it in terms of a concrete representation. For example: Stone duality for commutative algebras generated by idempotents; the duality between commutative von Neumann algebras and hyperstonean spaces; Gelfand duality for commutative \Cstar-algebras; and the Gelfand--Naimark theorem for noncommutative \Cstar-algebras.

A more recent example comes from Leavitt path algebras, whose advent \cite{AAP2005JA, AMP2007ART} has sparked substantial activity around interactions between abstract algebra and \Cstar-algebras over the last decade or so. It was discovered early on that the complex Leavitt path algebra of a graph embeds in the graph \Cstar-algebra, providing a concrete representation of any given Leavitt path algebra by bounded operators on a Hilbert space. But a clear understanding of the striking structural similarities between Leavitt path algebras over general rings and the corresponding graph \Cstar-algebras was only achieved through the development of convolution algebras of functions on groupoids \cite{Steinberg2010, CFST2014}, now called Steinberg algebras, as a unifying framework. Both Leavitt path algebras and graph \Cstar-algebras can be realised as algebras of functions on the underlying graph groupoid \cite{CFST2014, KPRR1997}. Moreover, the groupoid can be recovered from either the graph \Cstar-algebra together with its abelian subalgebra generated by the range projections of its generating partial isometries \cite{BCW2017ETDS}, or from the Leavitt path algebra (over a very broad class of rings) and its corresponding commutative subalgebra \cite{ABHS2017FM, BCaH2017JPAA, Carlsen2018, CR2018CCM, Steinberg2019}. This led to deep connections between abstract algebra and symbolic dynamics \cite{CR2018CCM}, and these have been significantly extended through recent reconstruction results that show how increasingly broad classes of groupoids can be reconstructed both from their \Cstar-algebras \cite{Renault2008, CRST2021} and from their Steinberg algebras \cite{ABHS2017FM, Steinberg2019}. Analogously to the case for groupoid \Cstar-algebras, there is a disintegration theorem~\cite{Steinberg2014} for realising modules over Steinberg algebras as sheaves over the groupoid, and Morita equivalence of Steinberg algebras can often be seen at the groupoid level~\cite{CS2015,Steinberg2014}. The upshot is that there are significant advantages of being able to recognise a given algebra as the Steinberg algebra of a groupoid.

In a series of papers \cite{FM1975BAMS, FM1977TAMS1, FM1977TAMS2} Feldman and Moore showed that every Borel $2$-cocycle on a Borel equivalence relation gives rise to a von Neumann algebra and a so-called Cartan subalgebra (roughly, an abelian subalgebra whose normaliser generates the whole algebra); and that the equivalence relation and $2$-cocycle can be recovered (up to cohomology) from the pair of algebras. A corresponding theorem for \Cstar-algebras was developed by Renault and by Kumjian first for principal groupoids \cite{Kumjian1986} and then for effective groupoids \cite{Renault2008}, using Kumjian's notion of a \emph{twist} over an \'etale groupoid in place of a $2$-cocycle. It was these results that Matsumoto and Matui subsequently used \cite{MM2014KJM} to prove that flow equivalence of shift spaces is characterised by diagonal-preserving isomorphism of the associated Cuntz--Krieger algebras, leading to the surge of activity around reconstruction of groupoids that are not even effective from the \Cstar-algebras discussed above.

Renault's results beg an extension to abstract algebras, following the program of ``algebraisation'' of operator-theoretic ideas going back to von Neumann and Kaplansky (see \cite{Kaplansky1968}) described in \cite{CH2020}: in the words of Berberian \cite{Berberian1972} ``if all the functional analysis is stripped away\,\textellipsis what remains should stand firmly as a substantial piece of algebra, completely accessible through algebraic avenues.'' The groundwork for such an extension was laid recently in \cite{ACCLMR2022}, where a notion of an algebraic twist over an ample groupoid, suitable for constructing twisted Steinberg algebras, was developed, and the structure of the resulting twisted Steinberg algebras was analysed.

In this paper we provide such an extension. We identify a class of commutative subalgebras, called \emph{algebraic quasi-Cartan subalgebras}, in abstract algebras over indecomposable commutative rings. Our main theorems establish a correspondence between algebraic quasi-Cartan pairs and twists over ample Hausdorff \'etale groupoids that satisfy a \emph{local bisection hypothesis} (analogous to that of \cite{Steinberg2019}) that generalises the no-nontrivial-units condition for twisted group rings. Indeed, we show that the local bisection hypothesis follows from the no-nontrivial-units condition for sufficiently many of the twisted group rings over the interior of the isotropy of the groupoid. This guarantees that our results apply to many important examples, as we detail in the final section of the paper---carrying our results for algebras beyond the current state of the art for \Cstar-algebras. We also pick out from amongst all algebraic quasi-Cartan pairs the algebraic Cartan pairs that correspond to the effective groupoids of Renault's setting, and the algebraic diagonal pairs that correspond to the principal groupoids of Kumjian's setting.

This opens interesting new avenues of study for algebras admitting quasi-Cartan subalgebras, and provides a wealth of invariants for algebraic quasi-Cartan pairs: for example, the topological full-group of the underlying groupoid $G$, with intriguing links to the Thompson groups \cite{Nekrashevych2004, MM2017GGD}; the homology groups of $G$, which have promising connections to computations of $K$-theory in the \Cstar-algebraic setting \cite{Matui2012PLMS, Matui2015JRAM, Matui2016AM}; and the twist itself, which behaves like a second cohomology class for the orbit space of the groupoid that can be used as a classifying invariant for Fell algebras \cite{aHKS2011}.

This paper resulted from a merging of two research teams. Each team independently proved similar results and when we realised what had happened we decided to join forces. We are glad we did---we think the end result is the better for it.

\subsection*{Related work}

In the final stages of preparation of this manuscript we became aware of Bice's independent work \cite{Bice2021} on ringoid bundles, and in particular on Steinberg rings and Steinberg bundles. We thank Bice bringing his work to our attention, and for sharing his preprint.

Bice's results deal, at their most general, with topological categories appropriately fibred over \'etale groupoids. His main results are about \emph{Steinberg bundles} which, loosely speaking, bear the same relationship to our discrete $R$-twists as Fell bundles bear to twists in the \Cstar-algebraic setting; in particular, in Bice's setting, the fibres of $q^{-1}(G)$ can vary, and need not be commutative. Bice shows that every Steinberg bundle determines a \emph{Steinberg ring}: a ring $A$ with distinguished subsemigroups $Z \subseteq S \subseteq A$ (corresponding to $I(B) \subseteq N(B) \subseteq A$), and a suitable notion of a conditional expectation $\Phi\colon A \to S$ such that $Z \subseteq \Phi(A)$, satisfying an additional support condition. He demonstrates that every Steinberg ring determines a corresponding Steinberg bundle (via an ultrafilter construction) whose Steinberg ring is the original $(A, S, Z, \Phi)$. However, the same pair $(A, Z)$ could admit multiple Steinberg-ring structures, yielding different Steinberg bundles.

Our set-up is less general, but requires less algebraic data to describe the algebraic objects of study: the quasi-Cartan pair $(A,B)$ is all the information required, and then $N(B)$ and $P$ can be recovered. We then obtain a complete and concrete correspondence between quasi-Cartan pairs and $R$-twists satisfying the local bisection hypothesis.

\subsection*{Outline of the paper}

The paper is organised as follows. In \cref{sec:prelims}, we introduce preliminaries on commutative $R$-algebras without torsion (this is a standing assumption on the commutative subalgebras $B$ in our quasi-Cartan pairs), on groupoids and discrete $R$-twists, on twisted Steinberg algebras, and on inverse semigroups of normalisers of commutative subalgebras.

In \cref{sec:algebraic qcps}, we introduce our three notions of algebraic pairs (see \cref{def:ACP}), and we prove in \cref{lem:D=>C} that every algebraic diagonal pair is an algebraic Cartan pair. We also show that for a given quasi-Cartan pair $(A,B)$, there is only one conditional expectation from $A$ to $B$ that satisfies the key algebraic condition of being implemented by idempotents (see \cref{def:conditional expectation}). To finish the section, we show that if $(A,B)$ is an algebraic Cartan pair, then there is a unique conditional expectation from $A$ to $B$, and we show in \cref{lem:C=>Q} that every algebraic Cartan pair is an algebraic quasi-Cartan pair.

In \cref{sec:prototypes}, we study the pairs of algebras arising from discrete $R$-twists, and demonstrate that such a pair is an algebraic quasi-Cartan pair if and only if the twist satisfies the local bisection hypothesis. We also provide tools for checking the local bisection hypothesis, and we finish the section by showing in \cref{prop:effectiveACPprincipalADP} that twists over principal groupoids yield algebraic diagonal pairs, and twists over effective groupoids yield algebraic Cartan pairs.

In \cref{sec:build twist}, we show how to construct a twist from a pair of algebras. We have made an effort to be clear about which hypotheses are required for each part of the construction. In \cref{sec:Sigma def}, we begin with a pair $(A,B)$ such that $B$ is without torsion and the idempotents of $B$ are a set of local units for $A$, and we describe an associated groupoid $\Sigma$ of ultrafilters of normalisers of $B$ and some of its key properties. In \cref{sec:G def} we additionally assume that $B$ is the $R$-linear span of its idempotent elements. With this hypothesis, we construct a quotient $G$ of $\Sigma$ by the action of $R^\times$, and show that $G$ is an ample groupoid. In \cref{sec:twist def}, we show that under the same hypotheses we obtain a discrete $R$-twist $\Go \times R^\times \hookrightarrow \Sigma \twoheadrightarrow G$ over $G$. We also show that if $(A,B)$ is an algebraic quasi-Cartan pair, then $G$ (and hence also $\Sigma$) is Hausdorff---this is all contained in \cref{thm:the twist,prop:Hausdorff}.

In \cref{sec:main iso}, we prove our first main result. We start with an algebraic quasi-Cartan pair $(A,B)$, and consider the twist $\Sigma$ over $G$ constructed in the preceding section. Our main theorem, \cref{thm:main}, shows that there is an explicit isomorphism of $A$ onto the twisted Steinberg algebra $A_R(G;\Sigma)$ that restricts to the standard isomorphism of $B$ onto $A_R(\Go; q^{-1}(\Go)) \cong A_R(\Go)$ arising from Stone duality as in \cite[Th\'eor\`eme~1]{Keimel1970}.

In \cref{sec:info from iso}, we explore the relationship between properties of a twist $\Sigma$ over $G$ satisfying the local bisection hypothesis and properties of the associated Steinberg algebra. Specifically, we prove that if $(A,B)$ is an algebraic quasi-Cartan pair, then it is a Cartan pair if and only if the associated groupoid $G$ is effective (so that our terminology is consistent with Renault's in \cite{Renault2008}) and that it is a diagonal pair if and only if $G$ is principal (so that our terminology is consistent with Kumjian's in
\cite{Kumjian1986}).

We see in \cref{sec:main iso} that passing from an algebraic quasi-Cartan pair to the associated twist and then to the associated twisted Steinberg algebra and its diagonal subalgebra recovers the original pair. In \cref{sec:recover twist} we consider the dual approach where we start with a twist, pass to the associated pair of algebras, and then construct the associated sequence of groupoids. We prove in \cref{prop:embedding} that given any twist $\Sigma$ over an ample Hausdorff groupoid $G$, if the algebraic pair $(A,B)$ consists of a Steinberg algebra and a diagonal subalgebra, then the groupoid $\Sigma'$ of ultrafilters of normalisers of $B$ admits a continuous open embedding of $\Sigma$ that restricts to a homeomorphism of unit spaces. Then in \cref{thm:reconstructing.the.twist}, we show that this embedding is an isomorphism if and only if $(A,B)$ is an algebraic quasi-Cartan pair (equivalently, if and only if $\Sigma \to G$ satisfies the local bisection hypothesis). We deduce in \cref{cor:equiv.twist} that given twists $\Sigma_1 \to G_1$ and $\Sigma_2 \to G_2$ such that $\Sigma_1 \to G_1$ satisfies the local bisection hypothesis, the twists are isomorphic (which implies in particular that $\Sigma_2 \to G_2$ also satisfies the local bisection hypothesis) if and only if the associated algebraic pairs are isomorphic. We then deduce in \cref{cor:equiv.twist.effective} that if $G_1$ is effective, then the twists are isomorphic (which implies in particular that $G_2$ is effective) if and only if there is an isomorphism $A_R(G_1; \Sigma_1) \mapsto A_R(G_2; \Sigma_2)$ of twisted Steinberg algebras that carries the diagonal subalgebra of $A_R(G_1; \Sigma_1)$ into the diagonal subalgebra of $A_R(G_2; \Sigma_2)$.

Finally, in \cref{sec:examples} we demonstrate the applicability of our results by showing that there are substantial classes of twists that satisfy the local bisection hypothesis, and therefore correspond to algebraic quasi-Cartan pairs. We consider groups $H$ with the \emph{unique product property}; that is, that products $AB$ of finite subsets of $H$ always contain an element with a unique factorisation of the form $ab$ with $a \in A$ and $b \in B$. (Examples include right-orderable groups, and in particular torsion-free abelian groups.) By adapting the approach used in \cite{Steinberg2019} for untwisted group rings, we show that if $R$ is reduced and indecomposable and $H$ is a group with the unique product property, then every twisted group $R$-algebra of $H$ has no nontrivial units. Combining this with our results in \cref{sec:prototypes}, we deduce that for reduced and indecomposable $R$, all $R$-twists over groupoids $G$ such that the fibres of the interior of the isotropy of $G$ have the unique product property satisfy the local bisection hypothesis (the reduced assumption can be dropped if the groupoid is effective). In particular, this includes all twists over ample Deaconu--Renault groupoids; we give a specific application to twisted Kumjian--Pask algebras.

\section{Preliminaries} \label{sec:prelims}

\subsection{Commutative \texorpdfstring{$R$}{R}-algebras without torsion}

Throughout this article, $R$ denotes a commutative ring with identity, $R^\times$ is the group of units of $R$, and $A$ is an $R$-algebra. We always assume that $B \subseteq A$ is a commutative subalgebra with idempotents $I(B)$. We also ask that $B$ is \emph{without torsion} with respect to $R$ in the sense implicitly intended in \cite{Keimel1970}; specifically,
\begin{gather} \label[condition]{cond:torsion free}
\text{if } e \in I(B) {\setminus} \{0\} \text{ and } t \in R \text{ satisfy } te=0, \text{ then } t = 0. \tag{WT}
\end{gather}
\Cref{cond:torsion free} holds automatically if $R$ is a field (and also if $R$ is an integral domain and $B$ is a torsion-free $R$-module, as explained below). Observe that if $B$ has at least one nonzero idempotent, then \cref{cond:torsion free} also implies that $R$ is \emph{indecomposable}\footnote{For us, indecomposable implicitly implies unital.} in the sense that its only idempotents are $0$ and $1$: if $t \in R$ satisfies $t^2=t$, then $t^2e=te$ for any $e \in I(B) {\setminus} \{0\}$, and then $t(1-t)e=0$, forcing $t=0$ or $t=1$ (since $(1-t)e\in I(B)$). Since we will be working exclusively with algebras $B$ that are generated by their idempotent elements, it follows that $R$ is generally indecomposable. This is also the condition on $R$ used in \cite{Steinberg2019}. The condition of a ring $R$ being indecomposable is a natural one in commutative ring theory. For instance, recall that the Zariski (or prime) spectrum $\Spec(R)$ of $R$ is the set of prime ideals of $R$ with the topology whose basic open subsets are of the form $D(r) = \{\fp \in \Spec(R) : r \notin \fp\}$ with $r\in R$. It is well known that $\Spec(R)$ is a connected space if and only if $R$ is indecomposable; see \cite[Chapter~1, Exercise~22]{AM1969}.

We will frequently use without comment the following observation. If $B$ is spanned as an $R$-module by $I(B)$, then each element $b\in B$ can be written as $b=\sum_{i=1}^n t_ie_i$ with $e_1,\dotsc, e_n\in I(B)$ mutually orthogonal and $t_1,\dotsc, t_n\in R$. Indeed, if $b=\sum_{e\in F}t_ee$ with $F\subseteq I(B)$ finite, then $F$ generates a finite Boolean algebra, where we recall that the join of commuting idempotents $e$ and $f$ is $e\vee f=e+f-ef$, their meet is $e\wedge f=ef$, and their relative complement is $e\setminus f=e-ef$. We can then take $e_1,\dotsc, e_n$ to be the atoms (minimal idempotents) of this Boolean algebra.

An $R$-module $M$ is called \emph{torsion-free} if $rm=0$ implies that either $m=0$ or $r$ is a zero divisor, for $r\in R$ and $m\in M$. Notice that if $I(B)$ generates $B$ as an $R$-algebra, then \cref{cond:torsion free} implies that $B$ is a torsion-free $R$-module. Indeed, if $0\ne b\in B$ and $tb=0$, we can write $b=\sum_{i=1}^n t_ie_i$ with $0\ne t_i\in R$ and $e_i\in I(B)$ with $e_1,\dotsc,e_n$ mutually orthogonal. Then $0=tbe_1=tt_1e_1$, and so $tt_1=0$ by \labelcref{cond:torsion free}, whence $t$ is a zero divisor. In particular, if $R$ is an integral domain and $B$ is generated by $I(B)$, then $B$ is torsion-free if and only if it satisfies \cref{cond:torsion free}.

\subsection{Groupoids}

We use $G$ to mean a locally compact topological groupoid with Hausdorff unit space $\Go$, composable pairs $\Gc \subseteq G \times G$, and range and source maps $\r,\s\colon G \to \Go$. We evaluate composition of groupoid elements from right to left, which means that $\gamma \gamma^{-1} = \r(\gamma)$ and $\gamma^{-1} \gamma = \s(\gamma)$, for all $\gamma \in G$.

For each $x \in \Go$, we define
\[
G_x \coloneqq \s^{-1}(x), \quad G^x \coloneqq \r^{-1}(x), \quad \text{and } \quad G_x^x \coloneqq G_x \cap G^x.
\]
For any two subsets $U$ and $V$ of a groupoid $G$, we define
\[
UV \coloneqq \{\alpha\beta : (\alpha,\beta) \in (U \times V) \cap \Gc \} \ \text{ and } \
U^{-1} \coloneqq \{\alpha^{-1} : \alpha \in U \}.
\]

We call a subset $B$ of $G$ a \emph{bisection} if there exists an open subset $U$ of $G$ such that $B \subseteq U$, and $\r\restr{U}$ and $\s\restr{U}$ are homeomorphisms onto open subsets of $\Go$. We say that $G$ is \emph{\'etale} if $\r$ (or, equivalently, $\s$) is a local homeomorphism. If $G$ is \'etale, then $\Go$ is open, and both $G_x$ and $G^x$ are discrete in the subspace topology for any $x \in \Go$. We recall that $G$ is \'etale if and only if $G$ has a basis of open bisections. We say that $G$ is \emph{ample} if it has a basis of \emph{compact} open bisections. If $G$ is \'etale, then $G$ is ample if and only if its unit space $\Go$ is totally disconnected (see \cite[Proposition~4.1]{Exel2010}). It is well known that an \'etale groupoid is Hausdorff if and only if its unit space is closed (see \cite[Proposition~3.10]{EP2016}, or, more generally, \cite[Lemma~8.3.2]{Sims2020}).

If $B$ and $D$ are compact open bisections of an ample groupoid, then $B^{-1}$ and $BD$ are also compact open bisections. The collection of compact open bisections is an inverse semigroup under these operations (see \cite[Proposition~2.2.4]{Paterson1999}).

The \emph{isotropy} of a groupoid $G$ is the set
\[
\Iso(G) \coloneqq \{ \gamma \in G : \r(\gamma) = \s(\gamma) \} = \bigcup_{x \in \Go} G_x^x.
\]
We say that $G$ is \emph{principal} if $\Iso(G) = \Go$, and that $G$ is \emph{effective} if the topological interior of $\Iso(G)$ is equal to $\Go$. So every principal \'etale groupoid is effective.

\begin{remark}
In full generality, effectiveness should not be confused with several related notions. There is the notion of a groupoid being topologically principal, which means that points with trivial isotropy (that is, $G_x^x=\{x\}$) are dense in $\Go$, and the weaker notion of a groupoid being topologically free, which means that $\Iso(G) \setminus \Go$ has empty interior. Both effective and topologically principal \'etale groupoids are topologically free (see \cite[Page~283]{KM2021}). For second-countable \'etale groupoids, the notions of topologically principal and topologically free coincide (see \cite[Corollary~2.26]{KM2021}), whereas for Hausdorff \'etale groupoids, being effective is equivalent to being topologically free (see \cite[Page~284]{KM2021}). Hence for second-countable Hausdorff \'etale groupoids, all three notions agree (see also \cite[Proposition~3.6]{Renault2008}). However, there are examples of ample groupoids that are effective but not topologically principal (see \cite[Example~6.4]{BCFS2014}, or~\cite[Section~6.2]{StSz21} for examples that are minimal, in the sense that $r(G_x)$ is dense in $\Go$ for every $x \in \Go$), and examples of second-countable non-Hausdorff groupoids that are \'etale and topologically principal but not effective (see \cite[Section~5.1]{CEPSS2019}). The construction in~\cite[Proposition~6.11]{StSz21}, but with the finite alphabet $X$ replaced by an infinite one, yields a topologically free minimal ample groupoid that is neither effective (by~\cite[Proposition~6.1]{StSz21}) nor topologically principal.
\end{remark}

For the majority of this paper, the groupoids we work with are ample Hausdorff groupoids, but we will make these assumptions explicit throughout. (Note that we use the term ``Hausdorff \'etale groupoid'' to refer to a topological groupoid that is locally compact as well as Hausdorff and \'etale.)

\subsection{Discrete twists}

We next recall \cite[Definition~4.1]{ACCLMR2022}. Let $G$ be an \'etale groupoid, let $R$ be a commutative unital ring, and let $T$ be a subgroup of $R^\times$. A \emph{discrete twist} over $G$ is a sequence
\begin{equation} \label{eqn:exact.seq}
\displaystyle \Go \times T \overset{i} \hookrightarrow \Sigma \overset{q} \twoheadrightarrow G,
\end{equation}
where the groupoid $\Go \times T$ is regarded as a trivial group bundle with fibres $T$, $\Sigma$ is a groupoid with $\Sigmao = i\big(\Go \times \{1\}\big)$, and $i$ and $q$ are continuous groupoid homomorphisms that restrict to homeomorphisms of unit spaces, such that the following conditions hold:
\begin{enumerate}[label=(DT\arabic*)]
\item \label{item:exact} The sequence is exact, in the sense that $i(\{x\} \times T) = q^{-1}(x)$ for every $x \in \Go$, $i$ is injective, and $q$ is surjective.
\item \label{item:P_alpha} The groupoid $\Sigma$ is a locally trivial $G$-bundle, in the sense that for each $\alpha \in G$, there is an open bisection $B_\alpha$ of $G$ containing $\alpha$, and a continuous map $P_\alpha\colon B_\alpha \to \Sigma$ such that
\begin{enumerate}[label=(\roman*)]
\item \label{item:section} $q \circ P_\alpha = \id_{B_\alpha}$;
\item \label{item:P_alpha homeo} the map $(\beta, t) \mapsto i(\r(\beta), t) \,
P_\alpha(\beta)$ is a homeomorphism from $B_\alpha \times T$ to $q^{-1}(B_\alpha)$.
\end{enumerate}
\item \label{item:centrality} The image of $i$ is \emph{central} in $\Sigma$, in the sense that $i(\r(\sigma), t) \, \sigma = \sigma \, i(\s(\sigma), t)$ for all $\sigma \in \Sigma$ and $t \in T$.
\end{enumerate}

We denote a discrete twist over $G$ by $(\Sigma,i,q)$, or by $\Sigma \to G$, or simply by $\Sigma$. We identify $\Sigmao$ with $\Go$ via $i$ (or via $q\restr{\Sigmao})$. Note that \cite[Definition~4.1]{ACCLMR2022} requires $G$ and $\Sigma$ to be Hausdorff, although we show in \cref{cor:coveringspace} that $\Sigma$ is Hausdorff whenever $G$ is Hausdorff.

Recall from \cite[Page~14]{ACCLMR2022} that there is a continuous free action of $T$ on $\Sigma$ given by
\[
t \cdot \sigma = i(\r(\sigma),t) \sigma \, \text{ for all } t \in T \text{ and } \sigma \in \Sigma.
\]

In this paper, we will only be interested in the setting where $T = R^\times$. We will sometimes emphasise the dependence on $R^\times$ by calling $\Sigma$ a \emph{discrete $R$-twist over $G$}, or sometimes just an \emph{$R$-twist over $G$}.

We begin by observing that if $\Sigma$ is a discrete twist, then it, too, is an \'etale groupoid, and the map $q$ is a quotient map. In fact, we have the following characterisations of discrete twists, some of which are easier to check than others.

\begin{prop} \label{prop:check.locally.trivial}
Consider an exact sequence of groupoids as in \labelcref{eqn:exact.seq} with $G$ \'etale but not necessarily Hausdorff. Suppose that conditions \cref{item:exact,item:centrality} in the definition of a discrete twist are satisfied. If $q$ is a local homeomorphism, then $\Sigma$ is \'etale. The following are equivalent.
\begin{enumerate}[label=(\arabic*)]
\item \label{item1:check.locally.trivial}
The sequence is a discrete twist.
\item \label{item2:check.locally.trivial} The map $i$ is open and $q$ is a covering map.
\item \label{item3:check.locally.trivial} The map $i$ is open and the map $q$ is a local homeomorphism.
\item \label{item4:check.locally.trivial} The maps $i$ and $q$ are open, and $\Sigma$ is \'etale.
\end{enumerate}
\end{prop}

\begin{proof}
If $q$ is a local homeomorphism, then since the source map for $\Sigma$ is the composition of $q$ with the source map for $G$, and since local homeomorphisms are closed under composition, $\Sigma$ is \'etale.

\mbox{\cref{item1:check.locally.trivial}$\implies$\cref{item2:check.locally.trivial}}. Condition~\cref{item:P_alpha} says that $\Sigma$ is a locally trivial fibre bundle with discrete fibres $T$, and hence $q$ is a covering map (and so a local homeomorphism). Thus $\Sigma$ is \'etale by the previous assertion. Now \cite[Lemma~4.3(b)]{ACCLMR2022} (which formally assumes that $\Sigma$ and $G$ are Hausdorff but only uses that they are \'etale and that conditions \cref{item:exact}--\cref{item:centrality} are satisfied) implies that $i$ is open.

\mbox{\cref{item2:check.locally.trivial}$\implies$\cref{item3:check.locally.trivial}} is immediate: every covering map is a local homeomorphism.

\mbox{\cref{item3:check.locally.trivial}$\implies$\cref{item4:check.locally.trivial}} is also immediate: if $q$ is a local homeomorphism then it is open, and $\Sigma$ is \'etale by the first assertion of the proposition.

\mbox{\cref{item4:check.locally.trivial}$\implies$\cref{item1:check.locally.trivial}}. Suppose that $i$ and $q$ are open and $\Sigma$ is \'etale. We show that \cref{item:P_alpha} is satisfied; that is, that $\Sigma$ is locally trivial. Fix $\alpha\in G$ and $\sigma\in \Sigma$ with $q(\sigma)=\alpha$, and choose an open bisection $V\subseteq \Sigma$ with $\sigma\in V$ using that $\Sigma$ is \'etale. Note that $q\restr{V}$ is injective, because if $q(\beta) = q(\beta')$ with $\beta,\beta'\in V$, then $\s(\beta) = \s(\beta')$, and so $\beta=\beta'$ since $V$ is a bisection. Put $B_{\alpha}\coloneqq q(V)$ and $P_{\alpha}\coloneqq (q\restr{V})^{-1}$. Then $P_\alpha$ is a homeomorphism since $q$ is open by hypothesis. Moreover, $B_\alpha$ is an open bisection, because $V$ is an open bisection and $q$ is an open map that respects $\r$ and $\s$. It remains to show that the map $\phi_{P_\alpha}\colon (\beta,t) \mapsto i(\r(\alpha),t) \, P_\alpha(\beta)$ is a homeomorphism from $B_{\alpha}\times T$ to $q^{-1}(B_{\alpha})$. Since $i$, $\r$, $P_{\alpha}$, and the multiplication in an \'etale groupoid are all continuous and open, it follows that $\phi_{P_\alpha}$ is continuous and open. We now show that it is bijective. If $i(\r(\beta),t) \, P_\alpha(\beta) = i(\r(\gamma),t') \, P_{\alpha}(\gamma)$, then applying $q$ to both sides shows that $\beta = \gamma$. We then obtain that $i(\r(\beta),t)=i(\r(\beta),t')$, and so $t=t'$ by the injectivity of $i$. Finally, to see that $\phi_{P_\alpha}$ is surjective, fix $\beta\in q^{-1}(B_{\alpha})$. Then $q(\beta)=q(P_{\alpha}(q(\beta)))$, and so $\beta P_{\alpha}(q(\beta))^{-1}\in q^{-1}(\Go)$. Thus $\beta = i(\r(\beta),t)P_{\alpha}(q(\beta))$ for some $t \in T$. This completes the proof.
\end{proof}

The only time we apply \cref{prop:check.locally.trivial} to non-Hausdorff groupoids is in the proof of \cref{thm:the twist}, where we only use the implication \mbox{\cref{item4:check.locally.trivial}$\implies$\cref{item1:check.locally.trivial}}, which does not rely on~\cite[Lemma~4.3(b)]{ACCLMR2022} (stated only for Hausdorff groupoids).

\begin{cor} \label{cor:coveringspace}
Let $(\Sigma,i,q)$ be a discrete twist over a Hausdorff \'etale groupoid $G$. Then $\Sigma$ is a Hausdorff \'etale groupoid, and $q$ is a local homeomorphism. In particular, $G$ has the quotient topology. Moreover, if $G$ is ample, then $\Sigma$ is ample.
\end{cor}

\begin{proof}
The map $q$ is a covering map, and hence a local homeomorphism, by \cref{prop:check.locally.trivial}. Thus $\Sigma$ is Hausdorff, because any covering space of a Hausdorff space is Hausdorff. If $G$ is ample, then it follows from condition~\cref{item:P_alpha} that $\Sigma$ is also ample.
\end{proof}

In \cite[Lemma~4.3(c)]{ACCLMR2022} it is shown that if $(\Sigma,i,q)$ is a discrete twist, then the sets $B_\alpha$ and maps $P_\alpha$ of condition~\cref{item:P_alpha} can be chosen such that $P_\alpha(B_\alpha \cap \Go) \subseteq \Sigmao$. A continuous map $P_\alpha\colon B_\alpha \to \Sigma$ defined as in condition~\cref{item:P_alpha} that satisfies this additional condition is called a \emph{(continuous) local section}. If $G$ is ample, then the open bisections from condition~\cref{item:P_alpha} can be chosen to be compact.

If $(\Sigma,i,q)$ is a discrete twist, then by definition, every element of $G$ has an open bisection neighbourhood that admits a local trivialisation. To finish this section, we observe that the reduction of $\Sigma$ to any countable union of compact open subsets of $G$ is topologically trivial. The proof of the following result is essentially the proof of \cite[Theorem~4.10]{ACCLMR2022}.

\begin{lemma} \label{lem:local triviality}
Let $G$ be an ample Hausdorff groupoid, and let $(\Sigma,i,q)$ be a discrete $R$-twist over $G$. Suppose that $U \subseteq G$ is a countable union of compact open subsets. Then there exists a continuous section $S\colon U \to \Sigma$ that \emph{trivialises} $\Sigma$ over $U$, in the sense that $(\gamma,t) \mapsto t \cdot S(\gamma)$ is a homeomorphism of $U\times R^\times$ onto $q^{-1}(U)$.
\end{lemma}

In fact, in \cref{lem:local triviality} it is enough for $U {\setminus} \Go$ to be paracompact rather than $U$ being a countable union of compact open sets, since it is well known that a paracompact locally compact space can be partitioned into clopen $\sigma$-compact subspaces.

In \cite{ACCLMR2022}, the authors studied discrete twists arising from continuous $2$-cocycles (see \cite[Example~4.5]{ACCLMR2022}), and showed that every discrete $R$-twist $(\Sigma,i,q)$ admitting a continuous global section for $q$ is induced by a continuous $2$-cocycle (see \cite[Proposition~4.8]{ACCLMR2022}). A continuous $2$-cocycle on a Hausdorff \'etale groupoid $G$ is a continuous map $\sigma\colon \Gc \to R^\times$ that satisfies the \emph{$2$-cocycle identity}: $\sigma(\alpha,\beta) \, \sigma(\alpha\beta,\gamma) = \sigma(\alpha, \beta\gamma) \, \sigma(\beta, \gamma)$ for all composable $\alpha, \beta, \gamma \in G$; and is \emph{normalised}, in the sense that $\sigma(\r(\gamma),\gamma) = \sigma(\gamma, \s(\gamma)) = 1$ for all $\gamma \in G$.

\begin{example}
We show how to construct, for a fairly broad class of rings $R$, a discrete $R$-twist over an ample groupoid that does not admit a continuous global section, and hence is not equivalent to a twist coming from a continuous $2$-cocycle. The corresponding twisted Steinberg algebra and diagonal subalgebra comprise an algebraic quasi-Cartan pair to which our theory applies, so the full generality of twists, rather than just $2$-cocycles, is necessary.

Let $R$ be a commutative unital ring. Suppose that $A \subseteq R^\times$ is a subgroup and that $t \in R^\times$ is an element such that $(i,a) \mapsto t^i a$ is an isomorphism of $\Z/2\Z \times A$ onto $R^\times$. An obvious example is $R = \Z$ with $t = -1$ and $A = \{1\}$; but more generally, any ordered ring (for example, a subring of an ordered field like $\R$) provides an example, as does any finite field of order congruent to $3$ modulo $4$.

We begin by showing that there is a totally disconnected locally compact Hausdorff space $X$ that supports a nontrivial principal $(\Z/2\Z)$-bundle.

Example~5.3 of \cite{Wiegand1969} describes a totally disconnected locally compact Hausdorff space $X$ with a covering $\UU$ by compact open sets such that the first \v{C}ech cohomology group of $X$ with respect to this covering and with coefficients in some sheaf of abelian groups does not vanish. In the footnote to the example, the author points out that one can, in fact, prove that the first \v{C}ech cohomology group $H^1(\UU;\Z/2\Z)$ with coefficients in the constant sheaf with value $\Z/2\Z$ does not vanish (with respect to this covering or any covering by compact open sets); some further details are provided in~\cite{Sawin}. By the theory of principal bundles, it follows that there exists a nontrivial principal $(\Z/2\Z)$-bundle $p\colon B\to X$ over $X$. Note that $p$ is, in fact, a double cover and hence a local homeomorphism. Since $\langle t\rangle \cong \Z/2\Z$, we can and do view $B$ as a principal $\langle t\rangle$-bundle.

Let $G$ be the trivial group bundle $G = X\times \Z$. Endow $\Z$ and $R^\times$ with the discrete topology. Let $\Sigma$ denote the topological sum (disjoint union)
\[
\Sigma \coloneqq (X\times 2\Z\times R^\times) \coprod (B\times (1+2\Z)\times A).
\]
Define $\Sigmao \coloneqq X\times \{0\}\times \{1\}$. Define $s,r\colon \Sigma \to \Sigmao$ by $s(x,j,r) = r(x,j,r) = (x,0,1)$ for $(x,j,r) \in X\times 2\Z\times R^\times$ and $s(b,k,a) = r(b,k,a) = (p(b), 0, 1)$ for $(b,k,a) \in B\times (1+2\Z)\times A$. These are local homeomorphisms because $p$ is a covering map. Using that $p$ is a double cover, so that each fibre of $p\colon B \to X$ has two elements, we can define a multiplication $\Sigmac \to \Sigma$ as follows: for $x \in X$, $b,b' \in p^{-1}(x)$, $j,j' \in 2\Z$, $k,k' \in 1 + 2\Z$, $r,r' \in R^\times$, $a,a' \in A$, and $i \in \{0,1\}$,
\begin{align*}
(x,j,r)(x,j',r') &\coloneqq (x, j+j', rr'), \\
(x,j,t^ia)(b,k,a') = (b, k, a')(x, j, t^ia) &\coloneqq (t^ib, j+k,aa'),\text{ and} \\
(b,k,a)(b',k',a') &\coloneqq (p(b),k+k', t^{1-\delta_{b,b'}}aa').
\end{align*}
Routine checks show that this is a continuous, associative, commutative multiplication. Hence $\Sigma$ is an ample \'etale groupoid---indeed, it is an abelian group bundle.

Define $i\colon X \times R^{\times} = \Go \times R^\times \to \Sigma$ by $i(x,r) = (x,0,r)$. Then $i$ is an injective open groupoid homomorphism onto a clopen subgroupoid of $\Sigma$. We have $\Sigmao = i(\Go \times \{1\})$, and $i(X \times R^{\times})$ is central because $\Sigma$ is abelian. Define $q\colon \Sigma\to X\times \Z$ by $q(x,j,r) = (x,j)$ for $(x,j,r) \in X\times 2\Z\times R^\times$, and $q(b,k,a) = (p(b),k)$ for $(b,k,a) \in B\times (1+2\Z)\times A$. Then $q$ is a surjective groupoid homomorphism and a covering map, and $q^{-1}(x,0) = \{x\}\times \{0\}\times R^\times = i(\{(x,0)\} \times R^\times)$ for each $x \in X$. Also, $q$ restricts to the natural homeomorphism $(x,0,1) \mapsto (x,0)$ from $\Sigmao$ to $\Go$. Using \cref{prop:check.locally.trivial}, we see that $\Go \times T \overset{i} \hookrightarrow \Sigma \overset{q} \twoheadrightarrow G$ is a discrete twist over $G$.

We claim that $q$ does not admit a continuous global section. To see this, suppose for contradiction that $f\colon G\to \Sigma$ is a continuous section for $q$ such that $f(\Go)\subseteq \Sigmao$. Then $f(X\times \{1\})\subseteq B\times \{1\}\times A$. In particular, $x \mapsto f(x,1,0)$ is a continuous map from $X$ to $B \times (1+2\Z)\times A$, and post-composing it with the projection of $B \times (1+2\Z)\times A$ onto $B$ yields a continuous global section for $p\colon B\to X$, which is a contradiction. Hence $(\Sigma,i,q)$ does not come from a continuous $2$-cocycle.
\end{example}

\subsection{Twisted Steinberg algebras}

In \cite{ACCLMR2022}, the authors defined the \emph{twisted Steinberg algebra} arising from a discrete $R$-twist $(\Sigma,i,q)$ over an ample Hausdorff groupoid $G$. This was accomplished in the setting where $\Sigma$ is \emph{topologically trivial} (that is, $\Sigma$ admits a continuous global section for $q$, and is thus induced by a continuous $2$-cocycle). The original definition also allowed for twists over groupoids with fibres a subgroup $T \le R^\times$. We include the definition which we use below.

Given a topological space $X$ and a ring $R$, we write $C(X,R)$ for the $R$-module of locally constant maps from $X$ to $R$. For $f \in C(X,R)$, we define
\[
\supp(f) \coloneqq \{x \in X : f(x) \ne 0\},
\]
which is a clopen set. We also define
\[
C_c(X,R) \coloneqq \{f \in C(X,R) : \supp(f) \text{ is compact}\},
\]
which is an $R$-submodule of $C(X,R)$.

\begin{definition} \label{def:A_R(G;Sigma)}
Let $G$ be an ample Hausdorff groupoid, and let $(\Sigma,i,q)$ be a discrete $R$-twist over $G$. We say that $f \in C(\Sigma,R)$ is \emph{$R^\times$-contravariant} if $f(t \cdot \sigma) = t^{-1}f(\sigma)$ for all $t \in R^\times$ and $\sigma \in \Sigma$,\footnote{In \cite[Definition~4.17]{ACCLMR2022}, the authors consider \emph{$R^\times$-equivariant} functions $f \in C(\Sigma,R)$, which satisfy $f(t \cdot \sigma) = t f(\sigma)$ for all $t \in R^\times$ and $\sigma \in \Sigma$. We remark that the results in \cite[Section~4.3]{ACCLMR2022} go through with \cref{def:A_R(G;Sigma)} as well. See \cite[Remark~4.24]{ACCLMR2022}. Our use of the term ``contravariant'' is not related to the notion of a ``contravariant functor''.} and we define
\[
A_R(G;\Sigma) \coloneqq \{ f \in C(\Sigma,R)\,:\, \text{$f$ is $R^\times$-contravariant and $q(\supp(f))$ is compact} \}.
\]
\end{definition}

As a notational convenience, given a discrete $R$-twist $(\Sigma,i,q)$ over $G$, for $f \in A_R(G;\Sigma)$, we define
\[
\supp_G(f) \coloneqq q(\supp(f)) \subseteq G.
\]
We then have $\supp(f) = q^{-1}(\supp_G(f))$ because of the $R^\times$-contravariance of $f$. Since $q$ is a quotient map and $\supp(f)$ is clopen, $\supp_G(f)$ is also clopen. Thus the condition that $q(\supp(f))$ is compact and the condition that $\overline{q(\supp(f))}$ is compact (as in \cite[Definition~4.17]{ACCLMR2022}) are equivalent.

In \textcolor{red}{\cite{ACCLMR2022}}, the authors were mainly concerned with topologically trivial twists, and in particular proved that $A_R(G;\Sigma)$ has a well-defined multiplication under this assumption. Here we show that $A_R(G;\Sigma)$ is an $R$-algebra even when the twist is not topologically trivial.

We now establish that elements of $A_R(G;\Sigma)$ can be expressed in terms of their restriction to any ``slice'' of their support with respect to the action of $R^\times$ on $\Sigma$.

\begin{lemma} \label{lem:unique extensions}
Let $G$ be an ample Hausdorff groupoid and let $(\Sigma,i,q)$ be a discrete twist over $G$. Suppose that $X$ is an open subset of $\Sigma$ such that $q\restr{X}$ is injective. Let $f\colon X \to R$ be a locally constant compactly supported function. Then there is a unique element $\tilde{f}$ of $A_R(G;\Sigma)$ such that $\supp(\tilde{f}) \subseteq R^\times \cdot X$ and $\tilde{f}\restr{X} = f$.
\end{lemma}

\begin{proof}
We have $q(t \cdot \sigma) = q(\sigma)$ for all $t \in R^\times$ and $\sigma \in \Sigma$, and so since $q\restr{X}$ is injective, $X$ intersects each $R^\times$-orbit at no more than one point. So there is a well-defined function $f'\colon R^\times \cdot \supp(f) \to R$ satisfying $f'(t \cdot \sigma) = t^{-1}f(\sigma)$ for all $\sigma \in \supp(f)$ and $t \in R^\times$. Since $f$ is locally constant and $\sigma \mapsto t \cdot \sigma$ is a homeomorphism for each fixed $t \in R^\times$, the function $f'$ is also locally constant. It follows from $q$ being an open map, $q(\supp(f))$ being compact, and $G$ being Hausdorff that $q(\supp(f))$ is clopen, and hence $R^\times \cdot \supp(f)$ is also clopen. Therefore, the function $\tilde{f}$ given by $\tilde{f}\restr{R^\times \cdot \supp(f)} = f'$ and $\tilde{f}\restr{\Sigma \setminus R^\times \cdot \supp(f)}= 0$ is locally constant. Since $q(\supp(\tilde{f})) = q(\supp(f))$, we have that $q(\supp(\tilde{f}))$ is compact, and $\tilde{f}$ is $R^\times$-contravariant by definition. The uniqueness is clear: if $g \in A_R(G;\Sigma)$ agrees with $f$ on $X$ then it agrees with $f'$ on $R^\times \cdot X$, and so if in addition $\supp(g) \subseteq R^\times \cdot X$, then $g = \tilde{f}$.
\end{proof}

Suppose that $X \subseteq \Sigma$ is a compact open bisection. Then in particular, $q\restr{X}$ is injective because $q$ respects $\r$ and $\s$ and restricts to a homeomorphism of unit spaces. We will denote by $\tilde{1}_X$ the element of $A_R(G;\Sigma)$ obtained by applying \cref{lem:unique extensions} to the constant function $1_X\colon X \to R$ that maps every element of $X$ to $1$.

\begin{prop} \label{prop:f is a finite sum of lifts of indicators}
Let $G$ be an ample Hausdorff groupoid, and let $(\Sigma,i,q)$ be a discrete $R$-twist. Then for each $f \in A_R(G;\Sigma)$, there exist a finite set $\FF$ of compact open bisections of $\Sigma$ with mutually disjoint images in $G$, and elements $r_U$ of $R$, for each $U \in \FF$, such that $f = \sum_{U \in \FF} r_U \tilde{1}_U$.
\end{prop}

\begin{proof}
We apply \cref{lem:local triviality} to find a section $S\colon \supp_G(f) \to \supp(f)$ that trivialises $\Sigma$ over $\supp_G(f)$. Then $X \coloneqq S(\supp_G(f)) \cong \supp_G(f) \times \{1\}$ is compact and open, and $f\restr{X}$ belongs to the (untwisted) Steinberg algebra $A_R(\Sigma)$. Consequently, by \cite[Proposition~4.14]{Steinberg2010}, we can write $f\restr{X} = \sum_{U \in \FF} r_U 1_U$ as a finite $R$-linear combination of characteristic functions of mutually disjoint compact open bisections. Since $q$ is injective on $X$, the images of the elements of $\FF$ under $q$ are also mutually disjoint.

By \cref{lem:unique extensions}, each $\tilde{1}_U$ belongs to $A_R(G;\Sigma)$, and so $\sum_{U \in \FF} r_U \tilde{1}_U \in A_R(G;\Sigma)$ because $A_R(G;\Sigma)$ is an $R$-module. Now, since the functions $f$ and $\sum_{U \in \FF} r_U \tilde{1}_U$ agree on $X$ and both belong to $A_R(G;\Sigma)$ with support contained in $\supp(f) = R^\times\cdot X$, \cref{lem:unique extensions} implies that $f = \sum_{U \in \FF} r_U \tilde{1}_U$.
\end{proof}

\begin{prop} \label{prop:multiplication in twisted Steinberg algebras}
Let $G$ be an ample Hausdorff groupoid, and let $(\Sigma,i,q)$ be a discrete $R$-twist over $G$. Given $f, g \in A_R(G;\Sigma)$, let $S\colon \supp_G(f) \to \supp(f)$ be any section (not necessarily continuous) for $q$ on $\supp_G(f)$. The formula
\[
(f * g)(\sigma) \coloneqq \sum_{\alpha \in G^{\r(\sigma)} \cap \supp_G(f)} f(S(\alpha)) \, g(S(\alpha)^{-1}\sigma)
\]
does not depend on the choice of $S$. This formula defines an associative multiplication on $A_R(G;\Sigma)$, making it into an $R$-algebra.
\end{prop}

\begin{proof}
It is straightforward to check that $A_R(G;\Sigma)$ is an $R$-module. Fix $f, g \in A_R(G;\Sigma)$. For each $\sigma \in \Sigma$, define $F^{\r(\sigma)} \coloneqq G^{\r(\sigma)} \cap \supp_G(f)$. Since each $F^{\r(\sigma)}$ is compact and discrete, and hence finite, the convolution formula makes sense.

Suppose that $S$ and $S'$ are two sections for $q$ on $\supp_G(f)$. Then there is a unique (not necessarily continuous) function $\alpha \mapsto t_\alpha$ from $\supp_G(f)$ to $R^\times$ such that $S(\alpha) = t_\alpha \cdot S'(\alpha)$ for all $\alpha\in \supp_G(f)$. Fix $\sigma \in \Sigma$. Using the centrality of $i(\Go \times R^\times)$ and the $R^\times$-contravariance of $f$ and $g$ for the second equality, we obtain
\begin{align*}
\sum_{\alpha \in F^{\r(\sigma)}} f(S(\alpha)) \, g(S(\alpha)^{-1}\sigma)
&= \sum_{\alpha \in F^{\r(\sigma)}} f(t_\alpha \cdot S'(\alpha)) \, g((t_\alpha \cdot S'(\alpha))^{-1}\sigma) \\
&= \sum_{\alpha \in F^{\r(\sigma)}} t_\alpha^{-1} f(S'(\alpha)) \, t_\alpha \, g(S'(\alpha)^{-1}\sigma) \\
&= \sum_{\alpha \in F^{\r(\sigma)}} f(S'(\alpha)) \, g(S'(\alpha)^{-1}\sigma),
\end{align*}
and thus the convolution formula does not depend on $S$.

By \cref{prop:f is a finite sum of lifts of indicators}, any $f\in A_R(G;\Sigma)$ can be written in the form $f = \sum_{U \in \FF} r_U \tilde{1}_U$, where $\FF$ is a finite collection of compact open bisections of $\Sigma$. As the convolution product is clearly $R$-bilinear by definition, to show both that $A_R(G;\Sigma)$ is closed under convolution and that the convolution product is associative, it suffices to show that $\tilde{1}_U * \tilde{1}_V = \tilde{1}_{UV}$ for compact open bisections $U$ and $V$ of $\Sigma$ (as the compact open bisections form a semigroup). Note that $q\restr{U}$ and $q\restr{V}$ are injective. The definition of the convolution shows that $\supp(\tilde{1}_U * \tilde{1}_V) \subseteq (R^\times \cdot U)(R^\times \cdot V) = R^\times \cdot UV$. For $(\tau,\rho) \in (U \times V) \cap \Sigmac$, choosing a section $\zeta\colon G \to \Sigma$ such that $\zeta(q(\tau)) = \tau$ gives that if $\sigma = \tau\rho$, then $(\tilde{1}_U * \tilde{1}_V)(\sigma) = \sum_{\alpha \in G^{\r(\sigma)}} \tilde{1}_U(\zeta(\alpha)) \tilde{1}_V(\zeta(\alpha)^{-1}\sigma) = \tilde{1}_U(\tau)\tilde{1}_V(\rho) = 1$, and so $\tilde{1}_U * \tilde{1}_V$ agrees with $\tilde{1}_{UV}$ on $UV$. Thus $\tilde{1}_U * \tilde{1}_V = \tilde{1}_{UV}$ by \cref{lem:unique extensions}.
\end{proof}

Note that when the intended meaning is clear, we sometimes write $fg$ rather than $f * g$ for the convolution product of two functions $f, g \in A_R(G;\Sigma)$.

The following corollary is a straightforward consequence of \cref{prop:f is a finite sum of lifts of indicators} and the proof of \cref{prop:multiplication in twisted Steinberg algebras}.

\begin{cor} \label{cor:normalisermult}
Let $\FF$ and $\GG$ be finite collections of compact open bisections of $\Sigma$ with mutually disjoint images in $G$, and suppose that $f, g \in A_R(G;\Sigma)$ satisfy $f = \sum_{U \in \FF} r_U \tilde{1}_U$ and $g = \sum_{V \in \GG} t_V \tilde{1}_V$. Then
\begin{equation} \label{eqn:convolution of sums}
f*g = \sum_{U \in \FF, V \in \GG} r_Ut_V \tilde{1}_{UV}.
\end{equation}
In particular, each $\sigma \in \Sigma$ is contained in a compact open bisection $X$ of $\Sigma$, and
\[
(\tilde{1}_X * f)(\r(\sigma')) = f\big((\sigma')^{-1}\big) = (f * \tilde{1}_X)(\s(\sigma')) \ \text{ for all } \sigma' \in X.
\]
\end{cor}

\begin{proof}
\Cref{eqn:convolution of sums} follows from the formula $\tilde{1}_U\ast \tilde{1}_V=\tilde{1}_{UV}$ established in the proof of \cref{prop:multiplication in twisted Steinberg algebras}, and from the bilinearity of multiplication.
The second claim follows easily from \cref{eqn:convolution of sums}.
\end{proof}

An important feature of the twisted Steinberg algebra $A_R(G;\Sigma)$ is that, even though it does not contain the algebra $A_R(\Sigmao)$ of locally constant functions on the unit space of $\Sigma$ (because such functions are not $R^\times$-contravariant), it does contain an algebra canonically isomorphic to $A_R(\Sigmao)$, and the actions of elements of this algebra on $A_R(G;\Sigma)$ by multiplication are exactly as one would expect.

\begin{prop} \label{prop:Steinberg diagonal}
Let $G$ be an ample Hausdorff groupoid and let $(\Sigma,i,q)$ be a discrete $R$-twist over $G$. The map $f \mapsto \tilde{f}$ obtained from \cref{lem:unique extensions} applied to elements $f \in C_c(\Sigmao,R) \cong A_R(\Go)$ defines an injective $R$-algebra homomorphism of $C_c(\Sigmao,R)$ into $A_R(G;\Sigma)$ whose image is precisely $A_R(\Go; q^{-1}(\Go))$. For $f \in C_c(\Sigmao,R)$ and $a \in A_R(G;\Sigma)$, we have $(\tilde{f}*a)(\sigma) = f(\r(\sigma))a(\sigma)$ and $(a*\tilde{f})(\sigma) = a(\sigma)f(\s(\sigma))$ for all $\sigma \in \Sigma$.
\end{prop}

\begin{proof}
That $f \mapsto \tilde{f}$ is an injective $R$-module map that has image contained in $A_R(\Go; q^{-1}(\Go))$ is straightforward. For the reverse containment, observe that if $g \in A_R(\Go; q^{-1}(\Go))$, then $f \coloneqq g\restr{\Sigmao} \in C_c(\Sigmao,R)$, and since $\tilde{f}$ and $g$ are elements of $A_R(\Go; q^{-1}(\Go))$ that agree on $\Sigmao$, they are equal by \cref{lem:unique extensions}.

That this map is multiplicative and hence is an $R$-algebra homomorphism is a particular case of the final assertion (where $a \in C_c(\Sigmao,R)$), so we now just have to prove the final assertion. Fix $f \in C_c(\Sigmao,R)$ and $a \in A_R(G;\Sigma)$. Let $S\colon q(\supp(f)) \to \supp(f)$ be the identity map. Then $S$ is a section from $\supp_G(\tilde{f})$ to $\supp(\tilde{f})$, and so \cref{prop:multiplication in twisted Steinberg algebras} shows that for all $\sigma \in \Sigma$,
\[
(\tilde{f}* a)(\sigma) = \sum_{\beta \in G^{\r(\sigma)} \cap \supp_G(\tilde{f})}
\tilde{f}(S(\beta)) \, a(S(\beta)^{-1}\sigma).
\]
The only term in the sum is $\beta = \r(\sigma)$,
which yields $S(\beta) = \r(\sigma)$, and so the sum collapses to $(\tilde{f}* a)(\sigma) = \tilde{f}(\r(\sigma)) a(\sigma) = f(\r(\sigma)) a(\sigma)$ since $\tilde{f}$ extends $f$ by definition.

For the second equality, fix $\sigma \in \supp(a)$ and choose a section $S\colon \supp_G(a) \to \supp(a)$ such that $S(q(\sigma)) = \sigma$. Then
\[
(a* \tilde{f})(\sigma) = \sum_{\beta \in
G^{\r(\sigma)} \cap \supp_G(a)} a(S(\beta)) \tilde{f}(S(\beta)^{-1}\sigma).
\]
The only nonzero terms in the sum are those corresponding to $\beta \in G^{\r(\sigma)} \cap \supp_G(a)$ such that
\[
S(\beta)^{-1}\sigma \in \supp(\tilde{f}) \subseteq i(\Go \times R^\times).
\]
Since $\supp_G(\tilde{f}) \cap G^{\r(\sigma)} = \{\r(\sigma)\}$, the only possible nonzero summand is such that $q(S(\beta)^{-1}\sigma) = \r(\sigma)$; that is, $\beta = q(\sigma)$. Since $\sigma \in \supp(a)$, we have $S(\beta) = \sigma$ by our choice of $S$, and so the sum collapses to $(a* \tilde{f})(\sigma) = a(\sigma) \tilde{f}(\sigma^{-1}\sigma) = a(\sigma) f(\s(\sigma))$. If $\sigma \in \Sigma\setminus\supp(a)$, then both $(a* \tilde{f})(\sigma)$ and $a(\sigma)f(\s(\sigma))$ are zero.
\end{proof}

When $R$ is indecomposable, the idempotents of $A_R(\Go)$ are the characteristic functions of compact open subsets of $\Go$. Given that indecomposability of $R$ is implied by \cref{cond:torsion free} when $I(B)$ is nontrivial, the $R$-algebra homomorphism from \cref{prop:Steinberg diagonal} gives a better perspective on the elements of $I(B)$.

\begin{cor} \label{cor:lattice iso}
Let $G$ be an ample Hausdorff groupoid, let $R$ be an indecomposable commutative ring, and let $(\Sigma,i,q)$ be a discrete $R$-twist over $G$. Then the map $U \mapsto \tilde{1}_U$ gives a lattice isomorphism between compact open subsets of $\Sigmao$ and $I(B)$.
\end{cor}

\subsection{Normalisers and inverse semigroups} \label{sec:invsemigroup}

An inverse semigroup is a semigroup $\SS$ such that, for each $s \in \SS$, there is a unique element $s^\dagger$ (called the \emph{inverse} of $s$) satisfying $ss^\dagger s=s$ and $s^\dagger ss^\dagger=s^\dagger$. We have $(st)^\dagger=t^\dagger s^\dagger$ and $(s^\dagger)^\dagger = s$. The idempotents of an inverse semigroup form a commutative subsemigroup, and are precisely the elements of the form $s^\dagger s$. Details can be found in~\cite{Lawson1998}.

\begin{definition} \label{def:normaliser}
Let $R$ be a commutative unital ring, let $A$ be an $R$-algebra, and let $B \subseteq A$ be a commutative $R$-subalgebra. Suppose that the set $I(B)$ of idempotents of $B$ is a set of \emph{local units for $A$}: that is, for any $\{a_1,\dotsc,a_n\}\subseteq A$, there exists $e\in I(B)$ with $ea_i=a_i=a_ie$ for all $i \in \{1,\dotsc,n\}$. As in \cite[Definition~3.3]{ABHS2017FM}, define the \emph{normaliser} of $B$ to be the set
\[
N(B) \coloneqq \{ n \in A : \text{there exists } k \in A \text{ with } knk = k, nkn = n, \text{and } kBn\cup nBk \subseteq B \}.
\]
\end{definition}

We use the notation $N_A(B)$ for $N(B)$ if we need to specify the ambient algebra $A$. We begin by establishing that $N(B)$ is an inverse semigroup.

\begin{lemma} \label{lem:dagger}
For each $n \in N(B)$, there exists a unique $k \in N(B)$ such that
\begin{equation} \label{eqn:dagger}
knk = k, \quad nkn = n, \quad \text{and } \quad kBn \cup nBk \subseteq B.
\end{equation}
\end{lemma}

\begin{proof}
Fix $n \in N(B)$. If $k$ satisfies~\cref{eqn:dagger} for $n$, then $k$ is itself in $N(B)$ because $n$ satisfies~\cref{eqn:dagger} for $k$. Suppose that $k_1$ and $k_2$ both satisfy~\cref{eqn:dagger} for $n$. We claim that since $I(B)$ forms a set of local units for $A$, we have
\[
k_1n, k_2n, nk_1, nk_2 \in B.
\]
Indeed, there exists $e \in I(B)$ with $k_1,k_2,n \in eAe$, and so $k_in = k_ien \in B$ and $nk_i=nek_i\in B$ for $i \in \{1,2\}$. Using this and that $B$ is commutative for the third and sixth equalities, we obtain
\[
k_1 = k_1nk_1
= k_1nk_2nk_1
= k_1nk_1nk_2
= k_1nk_2
= k_1nk_2nk_2
= k_2nk_1nk_2
= k_2nk_2
=k_2.\qedhere
\]
\end{proof}

For each $n \in N(B)$, we write $n^\dagger$ for the unique element $k \in N(B)$ satisfying \cref{eqn:dagger}. So $a \in A$ belongs to $N(B)$ if and only if $a^\dagger$ exists. Notice that the above proof shows that $aa^\dagger, a^\dagger a\in I(B)$ for all $a\in N(B)$. Also note that $I(B) \subseteq N(B)$ since $e^\dagger = e$ for each $e \in I(B)$.

\begin{lemma}
The normaliser $N(B)$ is an inverse semigroup with inverse $n\mapsto n^\dagger$ and set of idempotents $I(B)$.
\end{lemma}

\begin{proof}
To see that $N(B)$ is closed under multiplication, fix $n,m \in N(B)$. Then $nmBm^\dagger n^\dagger\subseteq nBn^\dagger\subseteq B$, and similarly, $m^\dagger n^\dagger Bnm\subseteq B$. As observed above, $mm^\dagger, n^\dagger n \in I(B)$, and so $mm^\dagger$ and $n^\dagger n$ commute. Hence
\[
(nm)(m^\dagger n^\dagger)(nm) = nn^\dagger nmm^\dagger m = nm.
\]
Similarly, $m^\dagger n^\dagger (nm)m^\dagger n^\dagger= m^\dagger n^\dagger$, and so $nm \in N(B)$. Thus $N(B)$ is a semigroup. \Cref{lem:dagger} therefore implies that $N(B)$ is an inverse semigroup with inversion given by $n\mapsto n^\dagger$. The final statement follows because if $e \in N(B)$ is an idempotent, then $e = e^\dagger e\in I(B)$, as we saw earlier.
\end{proof}

Recall that there is a partial order on any inverse semigroup given by $s \le t$\label{def:partial order} if and only if $s = ss^\dagger t$ (or, equivalently, if and only if $s = ts^\dagger s$). The partial order is preserved by multiplication and inversion, and the product of two idempotents is their meet. Furthermore, $s \le t$ if and only if there is an idempotent $e$ such that $s=te$, and this is equivalent to the existence of an idempotent $f$ such that $s = ft$. See \cite{Lawson1998} for details. Note that $\le$ reduces to the usual order relation on commuting idempotents; that is, $e \le f$ if and only if $ef = e$. In \cref{lem:scalarn,lem:properties of le}, we present various additional properties of $\le$ that we use throughout the paper. The first of these results is straightforward, so we omit the proof.

\begin{lemma} \label{lem:scalarn}
Suppose that $A$ is an $R$-algebra, and that $B \subseteq A$ is a commutative subalgebra such that $I(B)$ forms a set of local units for $A$. Fix $t \in R^\times$.
\begin{enumerate}[label=(\alph*)]
\item \label{item1:scalarn} If $n \in N(B)$, then $t n \in N(B)$ with $(t n)^\dagger =
t^{-1}n^\dagger$.
\item \label{item2:scalarn} If $n, m \in N(B)$ with $n \le m$, then $t n \le t m$.
\end{enumerate}
\end{lemma}

\begin{lemma} \label{lem:properties of le}
Suppose that $A$ is an $R$-algebra, and that $B \subseteq A$ is a commutative subalgebra satisfying \cref{cond:torsion free} such that $I(B)$ forms a set of local units for $A$. Then the following hold.
\begin{enumerate}[label=(\alph*)]
\item \label{item1:properties of le} If $n \in N(B)$ and $e \in I(B)$ satisfy $n\le e$, then $n\in I(B)$.
\item \label{item2:properties of le} For $f,e\in I(B)$ and $t\in R^\times$, if $0 \ne f \le t e$, then $t=1$.
\item \label{item3:properties of le} For $n,m \in N(B)$, we have $nn^\dagger m = n \iff mn^\dagger n = n$.
\item \label{item4:properties of le} For all $n \in N(B)$ and $e \in I(B)$, we have $en, ne \le n$.
\item \label{item5:properties of le} If $e\in I(B)$ and $n\in N(B)$, then $n-ne, n-en \in N(B)$, with $(n-ne)^\dagger = n^\dagger - en^\dagger$ and $(n-en)^\dagger = n^\dagger - n^\dagger e$, and $n-ne, n-en \le n$.
\end{enumerate}
\end{lemma}

\begin{proof}
Part~\cref{item1:properties of le} is~\cite[Proposition~1.4.7(5)]{Lawson1998}, and parts~\cref{item3:properties of le,item4:properties of le} are contained in~\cite[Lemma~1.4.6]{Lawson1998}. For part~\cref{item2:properties of le}, suppose that $0 \ne f \le te$. Then $tfe = f \ne 0$ implies that $fe \ne 0$ and
\[
t^2fe = (tfe)^2 = f^2 = f = tfe.
\]
Multiplying both sides by $t^{-1}$ then gives $(1-t)fe=0$. Hence $t=1$, as $B$ satisfies~\labelcref{cond:torsion free}.

Part~\cref{item5:properties of le} follows from \cref{item4:properties of le} once we observe that $n^\dagger ne\le n^\dagger n$ and hence $n^\dagger n-n^\dagger ne\in I(B) \subseteq N(B)$. Thus $n-ne=n(n^\dagger n-n^\dagger ne) \le n$, and $n - ne$ is an element of $N(B)$ with inverse $(n^\dagger n-n^\dagger ne)n^\dagger = n^\dagger- en^\dagger$. A similar argument proves the claim about $n-en$.
\end{proof}

We define \emph{free} normalisers by analogy (with suitable modifications) with the definition given by Kumjian in the setting of \Cstar-algebras.

\begin{definition}[{cf.~\cite[Definition~1]{Kumjian1986}}]
Let $R$ be a commutative unital ring and let $A$ be an $R$-algebra, with a commutative subalgebra $B\subseteq A$ whose idempotents form a set of local units for $A$. We say that $n \in N(B)$ is a \emph{free normaliser} if either $n \in B$ or $(n^\dagger n) (nn^\dagger ) = 0$ (this latter condition is equivalent to $n^2=0$, as $n^2=n(n^\dagger nnn^\dagger) n$).
\end{definition}

\section{Algebraic quasi-Cartan pairs} \label{sec:algebraic qcps}

In this section we introduce our main objects of study---algebraic quasi-Cartan pairs. Our main theorem shows that these are precisely the pairs of algebras that arise from discrete $R$-twists satisfying an appropriate local bisection hypothesis. We also define two additional types of pairs of algebras: algebraic diagonal pairs, and algebraic Cartan pairs. We prove later in this section that every diagonal pair is an algebraic Cartan pair and every algebraic Cartan pair is an algebraic quasi-Cartan pair. The diagonal pairs correspond to twists over principal groupoids, while the Cartan pairs correspond to twists over effective groupoids.

Before defining these pairs of algebras, we need to introduce an appropriate notion of a conditional expectation from an algebra onto a commutative subalgebra.

\begin{definition} \label{def:conditional expectation}
Let $R$ be an indecomposable commutative ring, let $A$ be an $R$-algebra, and let $B\subseteq A$ be a commutative subalgebra such that $I(B)$ is a set of local units for $A$. A map $P\colon A \to B$ is called a
\emph{conditional expectation} if
\begin{enumerate}[label=(\roman*)]
\item $P$ is $R$-linear;
\item $P\restr{B}=\id_B$; and
\item \label{item:B-linear} $P(bab')=bP(a)b'$ for $a\in A$ and $b, b'\in B$.
\end{enumerate}
The conditional expectation $P\colon A \to B$ is \emph{faithful} if, for every $a \in A {\setminus} \{0\}$, there exists $n \in N(B)$ such that $P(n a) \ne 0$. The conditional expectation $P\colon A \to B$ is \emph{implemented by idempotents} if, for every $n \in N(B)$, there exists $e \in I(B)$ such that $P(n) = ne = en$.
\end{definition}

Note that since $B$ contains local units for $A$, condition~\cref{item:B-linear} of \cref{def:conditional expectation} implies that $P$ is both left and right $B$-linear.

\begin{remark}
Watatani~\cite{Watatani1990} defines a conditional expectation $P\colon A \to B$ to be \emph{nondegenerate} if for every nonzero $a \in A$ there exists $a' \in A$ such that $P(a'a) \ne 0$. Since we will assume that $A$ is spanned by $N(B)$, the notion of a nondegenerate conditional expectation is equivalent to our definition of a faithful conditional expectation in our setting.
\end{remark}

We now define algebraic diagonal pairs, Cartan pairs, and quasi-Cartan pairs.

\begin{definition}
\label{def:ACP} Let $R$ be a unital ring, let $A$ be an $R$-algebra, and let $B\subseteq A$ be a commutative subalgebra satisfying~\cref{cond:torsion free} and with the following properties.
\begin{enumerate}[label=(\roman*)]
\item \label{item:ACP.local units} The set $I(B)$ forms a set of local units for $A$.
\item \label{item:ACP.IB spans} $B = \vecspan{(I(B))}$.
\item \label{item:ACP.NB spans} $A = \vecspan{(N(B))}$.
\item \label{item:ACP.FCE} There exists a faithful conditional expectation $P\colon A \to B$.
\end{enumerate}
Then we say that the pair $(A,B)$ is:
\begin{enumerate}
\item[(ADP)] \label[ADP]{item:ADP} an \emph{algebraic diagonal pair} if $A$ is spanned by the free normalisers of $B$;
\item[(ACP)] \label[ACP]{item:ACP} an \emph{algebraic Cartan pair} if $B$ is a maximal commutative subalgebra of $A$; and
\item[(AQP)] \label[AQP]{item:AQP.expectation} an \emph{algebraic quasi-Cartan pair} if there is a faithful conditional expectation $P\colon A \to B$ that is implemented by idempotents.
\end{enumerate}
\end{definition}

We will see that every algebraic diagonal pair is an algebraic Cartan pair (\cref{lem:D=>C}); that every algebraic Cartan pair is an algebraic quasi-Cartan pair, and that for such pairs there is only one conditional expectation from $A$ to $B$ (\cref{lem:C=>Q}); and that for any algebraic quasi-Cartan pair $(A,B)$, there is only one conditional expectation from $A$ to $B$ that is implemented by idempotents (\cref{prop:canonical CE}).

\begin{remarks} \leavevmode
\begin{enumerate}[label=(\arabic*)]
\item To construct the inverse semigroup $N(B)$ in \cref{def:normaliser}, we required $B$ to be commutative and to satisfy property~\cref{item:ACP.local units} of \cref{def:ACP}. In fact, $B$ being abelian and property~\cref{item:ACP.local units} are almost enough to build our discrete twist $\Sigma$ over $G$ in the next two sections. However, we need properties~\cref{item:ACP.IB spans} and \cref{item:AQP.expectation} of \cref{def:ACP} in \cref{prop:Hausdorff} to show that $G$ is Hausdorff.
\item In the definition of an algebraic quasi-Cartan pair, note that if $n \in N(B)$ satisfies $P(n) = 0$, then $e = 0$ satisfies $P(n) = en = ne$. So we could equivalently require just that whenever $P(n) \ne 0$ there exists $e \in I(B)$ such that $P(n) = en = ne$.
\item Significant work has been done in the direction of ascertaining existence and uniqueness of Cartan subalgebras of von Neumann algebras (for example \cite{SS1999, HR2011, HV2013}) and \Cstar-algebras \cite{LR2019}. So the corresponding questions of existence and uniqueness of algebraic quasi-Cartan subalgebras in a given $R$-algebra are natural and interesting questions, though they are not addressed in this paper.
\end{enumerate}
\end{remarks}

The simplest example of an algebraic diagonal pair is the following. Let $R$ be an indecomposable commutative ring, let $A=M_n(R)$ be the ring of $n\times n$ matrices over $R$, and let $B$ be the subalgebra of diagonal matrices. Then $B$ is spanned by the diagonal matrix units $E_{ii}$, which are idempotent, and $A$ is spanned by the elementary matrix units $E_{ij}$, which are free normalisers. The faithful conditional expectation is given by making all non-diagonal entries zero. This example motivates the terminology.

A more complicated, but illustrative, class of examples is the class of Leavitt path algebras associated to directed graphs. Recall that if $E$ is a directed graph and $L_R(E)$ is its Leavitt path algebra over a ring $R$, then it admits a commutative subalgebra $D_R(E)$---the subalgebra generated by the idempotents of the form $s_\mu s_{\mu^*}$ for finite paths $\mu$. \Cref{eg:KP-algs} shows that $(A_R(E), D_R(E))$ is an algebraic quasi-Cartan pair. \Cref{prop:effective} together with known structure theory for the groupoid of a directed graph show that $(A_R(E), D_R(E))$ is an algebraic Cartan pair if and only if every cycle in $E$ has an entrance, and is an algebraic diagonal pair if and only if $E$ contains no cycles.

\begin{lemma} \label{lem:D=>C}
Every algebraic diagonal pair is an algebraic Cartan pair.
\end{lemma}

\begin{proof}
Suppose that $(A,B)$ is an algebraic diagonal pair and that $a \in A$ commutes with every element of $B$; we must show that $a \in B$. Since $(A,B)$ is an algebraic diagonal pair, we can express $a$ as a finite linear combination of free normalisers. Every free normaliser $n$ either belongs to $B$ or satisfies $n^\dagger n nn^\dagger = 0$; so there exist $b \in B$, a finite set $F \subseteq N(B)$ such that $n^\dagger n nn^\dagger = 0$ for all $n \in F$, and coefficients $\alpha_n \in R$ for each $n \in F$ such that
\[
a = b + \sum_{n \in F} \alpha_n n.
\]
Thus it suffices to show that $a' \coloneqq \sum_{n \in F} \alpha_n n$ is zero. Since $a$ and $b$ are in the commutant (centraliser) of $B$, their difference $a' = a - b$ is also in the commutant of $B$. Let $M$ be the set of minimal nonzero idempotents in the Boolean algebra $X$ generated by $\{n^\dagger n, nn^\dagger : n \in F\}$. Then each $e \in M$ commutes with $a'$. Since $\sum_{e\in M}e$ is the maximum element of $X$, we have $n\sum_{e\in M}e=nn^\dagger n\sum_{e\in M}e=nn^\dagger n=n$. Hence $a' = \sum_{e \in M} a' e = \sum_{e \in M} ea'e = \sum_{n \in F} \alpha_n \sum_{e \in M} e n e$. So it suffices to show that each $ene = 0$. Fix $n \in F$ and $e \in M$. If $n^\dagger n e = 0$, then $e n e = e nn^\dagger n e = 0$, so we may suppose that $n^\dagger n e \ne 0$. Since $e$ is a minimal nonzero idempotent in $X$, we deduce that $e n^\dagger n = e$, and hence $enn^\dagger = e n^\dagger n nn^\dagger = 0$, since $n \in F$. Thus $e n e = e nn^\dagger n e = 0$, as required.
\end{proof}

We next show that every algebraic Cartan pair is an algebraic quasi-Cartan pair, so that we have \labelcref{item:ADP} implies \labelcref{item:ACP} implies \labelcref{item:AQP.expectation}.

\begin{lemma} \label{lem:C=>Q}
Suppose that $(A,B)$ is an algebraic Cartan pair. Then $(A,B)$ is an algebraic quasi-Cartan pair, and every conditional expectation $P\colon A \to B$ is implemented by idempotents.
\end{lemma}

\begin{proof}
Suppose that $(A,B)$ is an algebraic Cartan pair and that $P\colon A \to B$ is a conditional expectation. To see that $P$ is implemented by idempotents, suppose that $n \in N(B)$ satisfies $P(n) \ne 0$. Write $P(n) = \sum_{i=1}^k t_i e_i$, where the $e_i$ are mutually orthogonal elements of $I(B)$ and each $t_i \in R {\setminus} \{0\}$. Let
\[
e(n) \coloneqq \sum_{i=1}^k e_i.
\]
Then $e(n) \in I(B)$. We show that $e(n)n = ne(n) = P(n)$. First, we claim that if $f\le e(n)$ with $f\in I(B)$, then $n^\dagger fn=f$. To prove the claim, notice that
\[
P(nf) = P(n)f = \sum_{i=1}^k t_ie_if,
\]
and since $nfn^\dagger \in nBn^\dagger \subseteq B$,
\begin{equation} \label{eqn:P(nf) sum}
P(nf) = P(nn^\dagger nf)= P(nfn^\dagger n) = nfn^\dagger P(n) = P(n)nfn^\dagger = \sum_{i=1}^k t_ie_infn^\dagger.
\end{equation}
Since $P(nf)f=P(nf)$, we deduce after multiplying both sides of \cref{eqn:P(nf) sum} on the right by $f$ that $\sum_{i=1}^k t_ie_if = \sum_{i=1}^k t_ie_infn^\dagger f$. For any $j \in \{1,\dotsc, k\}$, multiplying both sides by $e_j$ on the left yields $t_je_jf = t_je_jnfn^\dagger f$, and so each $t_j(e_jf-e_jnfn^\dagger f)=0$. Since $t_j \ne 0$ and $e_jf-e_jnfn^\dagger f=e_jf-e_jfnfn^\dagger$ is an idempotent, we deduce from \cref{cond:torsion free} that each $e_jf=e_jnfn^\dagger f$. Summing over all $j \in \{1, \dotsc, k\}$ and using that $f\le e(n)$ yields
\[
f=e(n)f=\sum_{i=1}^ke_if=\sum_{i=1}^ke_infn^\dagger f = e(n)nfn^\dagger f\le nfn^\dagger.
\]
Therefore, $n^\dagger fn\le n^\dagger nfn^\dagger n\le f$. To prove the reverse inequality, note that
\[
P(nf) = P(n)f=fP(n) = P(fn) = P(nn^\dagger fn) = P(n)n^\dagger fn = \sum_{i=1}^k t_ie_in^\dagger fn.
\]
Therefore, for any $j \in \{1, \dotsc, k\}$, we have $t_je_jf= t_je_jn^\dagger fn$, forcing
\[
t_j(e_jf-e_jn^\dagger fn)=0.
\]
Hence $e_jf=e_jn^\dagger fn$ by \labelcref{cond:torsion free}, as $n^\dagger fn\le f$ implies that $e_jf-e_jn^\dagger fn$ is an idempotent. Once again summing over $j$, and using that $f\le e(n)$, yields $f=e(n)n^\dagger fn\le n^\dagger fn$, and so $f=n^\dagger f n$ for all $f\le e(n)$, proving the claim.

It follows that if $f\le e(n)$ in $I(B)$, then
\begin{equation} \label{eqn:commutation}
nf = nn^\dagger f n = f nn^\dagger n = fn.
\end{equation}

Now fix $g \in I(B)$. Applying~\cref{eqn:commutation} to $f = ge(n) \le e(n)$ at the first step, and to $f = e(n)$ at the final step, we obtain
\[
g (e(n)n) = n g e(n) = ne(n) g = (e(n)n) g.
\]
Since $B$ is spanned by $I(B)$, we deduce that $e(n)n$ commutes with all elements of $B$. Since $(A,B)$ is an algebraic Cartan pair, $B$ is a maximal commutative subalgebra, and so $e(n)n \in B$. We then have $e(n)n = P(e(n)n) = e(n)P(n) = P(n)$ by the definition of $e(n)$. As we already observed, \cref{eqn:commutation} gives $e(n)n = ne(n)$, and so $(A,B)$ is an algebraic quasi-Cartan pair.
\end{proof}

In \cref{sec:main iso}, we will define an isomorphism of $A$ onto the twisted Steinberg algebra of an associated twist in terms of a conditional expectation $P\colon A \to B$ that is implemented by idempotents. So \emph{a priori} the isomorphism would depend upon a choice of such a conditional expectation. We finish this section by showing that in fact such a conditional expectation is uniquely determined by the algebraic structure of $(A,B)$.

\begin{prop} \label{prop:canonical CE}
Suppose that $(A,B)$ is an algebraic quasi-Cartan pair. For each normaliser $n \in N(B)$, the set $\{e \in I(B) : e \le n^\dagger n \text{ and } ne \in B\}$ contains a maximum element $e(n)$. If $P$ is a conditional expectation that is implemented by idempotents, then $P(n) = n e(n)=e(n)n$. In particular, there is a unique conditional expectation $P\colon A \to B$ that is implemented by idempotents. In general there can be many conditional expectations from $A$ to $B$ that are not implemented by idempotents.
\end{prop}

\begin{proof}
Fix $n\in N(B)$ and suppose that $P\colon A \to B$ is a conditional expectation that is implemented by idempotents. Write $P(n) = \sum_{i=1}^k t_i e_i$, where the $e_i$ are mutually orthogonal idempotents and each $t_i \in R {\setminus}\{0\}$. Let $e(n) \coloneqq \sum_{i=1}^k e_i \in I(B)$. Then we have $P(n)e(n) = P(n)$. We claim that $e(n)$ is the smallest such element of $I(B)$. To see this, suppose that $f \in I(B)$ satisfies $P(n)f=P(n)$. Then, for each $j \in \{1, \dotsc, k\}$, we have
\[
t_jfe_j = P(n)fe_j = P(n)e_j = t_je_j,
\]
and hence $t_j(e_j-fe_j)=0$, whence $e_j=fe_j$ by \labelcref{cond:torsion free}. Thus $e(n)=fe(n) \le f$, proving the claim.
Since $P$ is implemented by idempotents, there exists $f \in I(B)$ such that $fn = nf = P(n)$. Thus
\[
n e(n) = nfe(n) = P(n)e(n) = P(n) = e(n)P(n) = e(n)fn = fe(n)n = e(n)n.
\]

Notice that $e(n)$ is an element of the set $\{e \in I(B) : e \le n^\dagger n\text{ and }ne \in B\}$: we have shown that $ne(n) = P(n) \in B$, and since $P(n)(n^\dagger n) = P(n)$, we have $e(n) \le n^\dagger n$ by our previous argument showing that $e(n)$ is the smallest such element. It remains to show that $e(n)$ is a maximum. Fix $g \in \{e \in I(B) : e \le n^\dagger n \text{ and } ne \in B\}$. We show that $g \le e(n)$. Using that $ng, g \in B$, we compute $ng=P(ng)=P(n)g=ne(n)g=nge(n)$, and so $g=n^\dagger ng=n^\dagger nge(n)=ge(n) \le e(n)$.

To see that in general there can be multiple conditional expectations from $A$ to $B$, consider the complex group algebra $A = \C\Z$ and commutative subalgebra $B = \C1_A$ of scalar multiples of the identity. Then $A$ embeds unitally in the algebra $C(\T)$ of continuous functions on the complex unit circle as Laurent polynomials, and every probability measure $m$ on $\T$ with full support determines a faithful conditional expectation $f \mapsto \int_\T f(z)\, \mathrm{d} m(z)$ from $A$ to $B$.
\end{proof}

\begin{cor} \label{cor:unique CE}
Suppose that $(A,B)$ is an algebraic Cartan pair. Then there is a unique conditional expectation from $A$ to $B$.
\end{cor}

\begin{proof}
\Cref{lem:C=>Q} implies that every conditional expectation from $A$ to $B$ is implemented by idempotents, and \cref{prop:canonical CE} implies that there is a unique such conditional expectation.
\end{proof}

We conclude this section with an example showing that \labelcref{item:AQP.expectation} is a generalisation of the much studied no-nontrivial-units condition on twisted group rings. Recall that if $R$ is a commutative unital ring, $G$ is a group with identity $e$, and $c\colon G\times G\to R^\times$ is a normalised $2$-cocycle, then the \emph{twisted group ring} $R(G,c)$ is the $R$-algebra of finitely supported functions $f\colon G\to R$ with multiplication given by
\[
(f * g)(\beta) = \sum_{\alpha\in G} c(\alpha,\alpha^{-1}\beta) \, f(\alpha) \, g(\alpha^{-1}\beta).
\]
It is a unital ring with identity the point-mass function $\delta_e$. Note that twisted group rings are precisely the twisted Steinberg algebras of twists over discrete groups (viewed as ample groupoids); see~\cref{lem:unit cocycle}.

\begin{example} \label{exam:no.non.triv.units}
Let $G$ be a discrete group with identity $e$, let $R$ be an indecomposable commutative ring, and let $c\colon G\times G\to R^\times$ be a normalised $2$-cocycle. Let $A \coloneqq R(G,c)$ be the corresponding twisted group ring, and let $B \coloneqq R\delta_e \cong R$. Then $I(B) = \{0,\delta_e\}$ spans $B$ (which trivially satisfies \labelcref{cond:torsion free}), and $N(B)=A^\times \cup\{0\}$, which clearly contains $\{\delta_g : g \in G\}$ and hence spans $A$. We claim that $(A,B)$ is an algebraic quasi-Cartan pair if and only if every unit of $A$ is of the form $t\delta_g$ for some $t \in R^\times$ and $g \in G$; that is, $A$ only has trivial units. We discuss twisted group rings with only trivial units---and the connection to the famous Kaplansky unit conjecture, recently resolved by Gardam \cite{Gardam2021}---in more detail in \cref{sec:examples}.

In this setup, a conditional expectation is an $R$-linear map $P\colon A\to R\delta_e$ that fixes $R\delta_e$. We claim that the only possible conditional expectation that is implemented by idempotents is given on the basis by fixing $\delta_e$ and annihilating $\delta_g$ if $g\ne e$. Indeed, $\delta_g$ is a normaliser and hence $P(\delta_g)=f\delta_g$ for some $f\in I(B)=\{0,\delta_e\}$. But since $P$ has image $B$, we must choose $f=0$ unless $g=e$, in which case we must choose $f=\delta_e$. This formula defines a conditional expectation $P$. It is faithful because if $a \in A {\setminus} \{0\}$ has a nonzero coefficient of $\delta_g$, then $P(\delta_{g^{-1}}a)\ne 0$.

It remains to verify that $P$ is implemented by idempotents if and only if $A$ only has trivial units. If $A$ has only trivial units and $u\in N(A)=A^\times\cup\{0\}$, then $P(u)=u=u\delta_e=\delta_eu$ if $u\in R^\times\delta_e$, or $P(u)=0=u0=0u$, otherwise, and so $P$ is implemented by idempotents. Conversely, if $P$ is implemented by idempotents and $u\in A^\times$, then we can find $g\in G$ with $e$ in the support of $u\delta_{g^{-1}}$. Then $P(u\delta_{g^{-1}})=r\delta_e$ for some $r\ne 0$. On the other hand, $P(u\delta_{g^{-1}}) = u\delta_{g^{-1}} f$ with $f=0$ or $f=\delta_e$. We conclude that $r\delta_e=P(u\delta_{g^{-1}}) = u\delta_{g^{-1}}$, and so $r\delta_g = u \in A^\times$ and $r \in R^\times$.
\end{example}

\section{Building an algebraic quasi-Cartan pair from a twist} \label{sec:prototypes}

In this section we show that a discrete twist over an ample Hausdorff groupoid $G$ gives rise to an algebraic quasi-Cartan pair if and only if it satisfies an appropriate analogue of the local bisection hypothesis of \cite{Steinberg2019}. We then also follow the arguments of \cite{Steinberg2019} to see that an algebraic twist satisfies the local bisection hypothesis if and only if the sub-twist corresponding to the interior of the isotropy does so, and that a sufficient condition for this is that there is a dense set of units $x$ for which the twisted group algebra corresponding to the fibre of the interior of the isotropy over $x$ has no nontrivial units. Finally, we demonstrate that if the underlying groupoid is effective, then the twist gives rise to an algebraic Cartan pair, and if it is principal, then we obtain an algebraic diagonal pair.

\begin{definition}
Let $R$ be a commutative unital ring. We say that a discrete $R$-twist $(\Sigma,i,q)$ over an ample Hausdorff groupoid $G$ \emph{satisfies the local bisection hypothesis} if for every normaliser $n$ of $A_R(\Go; q^{-1}(\Go))$ in $A_R(G;\Sigma)$, the set $\supp_G(n)$ is a bisection of $G$.
\end{definition}

Our interest in twists satisfying the local bisection hypothesis is that they give rise to algebraic quasi-Cartan pairs. The following lemma can be seen as a generalisation of \cref{exam:no.non.triv.units}.

\begin{lemma} \label{lem:LBH<->AQP}
Let $R$ be an indecomposable commutative ring and let $(\Sigma,i,q)$ be a discrete $R$-twist over an ample Hausdorff groupoid $G$. Let $A = A_R(G;\Sigma)$, and let $B = A_R(\Go; q^{-1}(\Go)) \subseteq A$. The pair $(A,B)$ satisfies conditions \cref{item:ACP.local units}--\cref{item:ACP.FCE} of \cref{def:ACP} with faithful conditional expectation $P\colon A \to B$ given by restriction of functions from $\Sigma$ to $q^{-1}(\Go)$. If $(A,B)$ is an algebraic quasi-Cartan pair, then $P$ is implemented by idempotents. The pair $(A,B)$ is an algebraic quasi-Cartan pair if and only if $(\Sigma,i,q)$ satisfies the local bisection hypothesis.
\end{lemma}

\begin{proof}
Since $R$ is indecomposable, the idempotents of $C_c(\Go,R)$ are precisely the characteristic functions of compact open subsets of $\Go$, and hence $B$ satisfies \cref{cond:torsion free}, because $B \cong C_c(\Go,R)$ by \cref{prop:Steinberg diagonal}. A routine argument using \cref{cor:normalisermult} shows that for each compact open bisection $U$ of $\Sigma$, $\tilde{1}_U$ is a normaliser of $B$ in $A$ with inverse $\tilde{1}_U^\dagger = \tilde{1}_{U^{-1}}$.

Properties \cref{item:ACP.local units,item:ACP.IB spans} of \cref{def:ACP} follow immediately from \cref{prop:Steinberg diagonal}, and property~\cref{item:ACP.NB spans} follows from \cref{prop:f is a finite sum of lifts of indicators}. For \cref{item:ACP.FCE}, it is straightforward to see that $P$ is a conditional expectation, and for faithfulness, note that if $f(\sigma) \ne 0$ for some $\sigma \in \Sigma$, then for any compact open bisection $X \subseteq \Sigma$ containing $\sigma^{-1}$, \cref{cor:normalisermult} implies that $P(\tilde{1}_X f)(\s(\sigma)) = (\tilde{1}_X f)(\r(\sigma^{-1})) = f(\sigma) \ne 0$.

Suppose that $(A,B)$ is an algebraic quasi-Cartan pair. Then by \cref{prop:canonical CE}, there is a unique faithful conditional expectation $P'\colon A \to B$ that is implemented by idempotents. We claim that $P = P'$. Since the functions $\tilde{1}_V$ with $V\subseteq \Sigma$ a compact open bisection span $A$ as an $R$-module by \cref{prop:f is a finite sum of lifts of indicators}, it suffices by the linearity of $P$ and $P'$ to prove that for each such $V$, $P'(\tilde{1}_V) = \tilde{1}_V\restr{q^{-1}(\Go)} = \tilde{1}_W$, where $W \coloneqq V \cap q^{-1}(\Go)$, which is a compact open bisection contained in $V$. Recall from \cref{cor:lattice iso} that $U \mapsto \tilde{1}_U$ is a lattice isomorphism between compact open subsets of $\Sigmao$ and $I(B)$. By \cref{prop:canonical CE} (and several applications of \cref{cor:normalisermult}), there is a maximum element $\tilde{1}_U$ of $I(B)$ satisfying $\tilde{1}_U \le \tilde{1}_V^\dagger \tilde{1}_V = \tilde{1}_{\s(V)}$ such that $\tilde{1}_V\tilde{1}_U\in B$, and $P'(\tilde{1}_V) = \tilde{1}_V\tilde{1}_U=\tilde{1}_{VU}$. It follows that $VU \subseteq W$ by the definition of $B$ (and since $U \subseteq \Sigmao$). On the other hand, $\s(W) \subseteq \s(V)$ and $\tilde{1}_V\tilde{1}_{\s(W)}=\tilde{1}_W\in B$, and so it follows by maximality that $\s(W) \subseteq U$. Thus $W = V\s(W) \subseteq VU$, and so $W=VU$. We conclude that $P'(\tilde{1}_V) = \tilde{1}_W = P(\tilde{1}_V)$, as required.

Now, we show that $(\Sigma,i,q)$ satisfies the local bisection hypothesis. Suppose that $n$ is a normaliser and that $\sigma \in \supp(n)$. We show that if $\tau \in \Sigma_{\s(\sigma)} \setminus (R^\times \cdot \sigma)$, then $n(\tau) = 0$; that the same holds if $\tau \in \Sigma^{\r(\sigma)} \setminus (R^\times \cdot \sigma)$ follows by a similar argument. By choosing a continuous local section from $G$ to $\Sigma$, we can find a compact open bisection $U \subseteq \Sigma$ containing $\sigma$ such that $(r,\epsilon) \mapsto r \cdot\epsilon$ is a homeomorphism of $R^\times \times U$ onto $q^{-1}(q(U))$. Then we have $P(n\tilde{1}_{U^{-1}}) = (n\tilde{1}_{U^{-1}})\restr{q^{-1}(\Go)}$, and in particular, by \cref{cor:normalisermult},
\[
P(n\tilde{1}_{U^{-1}})(\r(\sigma)) = n(\sigma)\tilde{1}_{U^{-1}}(\sigma^{-1}) = n(\sigma) \ne 0.
\]
Since $P(n\tilde{1}_{U^{-1}}) = n\tilde{1}_{U^{-1}}\tilde{1}_V$ for some compact open subset $V \subseteq \Go$, we deduce that $\r(\sigma) \in V$. Using again that $P(n\tilde{1}_{U^{-1}}) = (n\tilde{1}_{U^{-1}})\restr{q^{-1}(\Go)}$, we see that
\[
n(\tau) = (n\tilde{1}_{U^{-1}})(\tau\sigma^{-1}) = (n\tilde{1}_{U^{-1}} \tilde{1}_V)(\tau\sigma^{-1}) = P(n\tilde{1}_{U^{-1}})(\tau\sigma^{-1}).
\]
Since $\tau \in \Sigma_{\s(\sigma)} \setminus (R^\times \cdot \sigma)$, we have $\tau\sigma^{-1}\notin q^{-1}(\Go)$, and so we deduce that $n(\tau) = 0$.

Finally, suppose that $(\Sigma,i,q)$ satisfies the local bisection hypothesis. We must show that $(A,B)$ is an algebraic quasi-Cartan pair. Fix a normaliser $n \in N(B)$. Then $\supp_G(n)$ is a bisection. Let $U \coloneqq \supp(n) \cap \Sigmao$, and let $e \coloneqq \tilde{1}_U \in I(B)$. Then $ne = n\restr{\s^{-1}(U)}$. If $n(\sigma) \ne 0$ and $\sigma \notin R^\times \cdot U = \supp(n)\cap q^{-1}(\Go)$, then $\s(\sigma) \notin U$ because $\supp_G(n)$ is a bisection. It follows by a routine calculation that $ne = n\restr{\s^{-1}(U)} = P(n)$, and similarly $en = P(n)$,
as required.
\end{proof}

Since twists satisfying the local bisection hypothesis give examples of algebraic quasi-Cartan pairs, we are interested in identifying conditions that guarantee the local bisection hypothesis. In the next few results, we follow the analysis used in \cite{Steinberg2019} to give a sufficient condition in terms of twisted group algebras associated to the fibres of the interior of the isotropy in the twist. We begin by describing these algebras. This is standard, and our proof follows ideas from \cite[Section 4.3]{ACCLMR2022}.

Given a discrete $R$-twist $(\Sigma,i,q)$ over a groupoid $G$, let $\II$ denote the interior of the isotropy in $G$, and let $\JJ \coloneqq q^{-1}(\II)$, which is the interior of the isotropy in $\Sigma$. Notice that the image of $i$ is contained in $\JJ$ and we can restrict $q$ to $\JJ$ and obtain a map $q\restr{\JJ}\colon \JJ \to \II$. It is straightforward to check that $(\JJ, i, q\restr{\JJ})$ is again a twist and that $A_R(\II;\JJ)$ can be seen as a subalgebra of $A_R(G;\Sigma)$ by extending an element of $A_R(\II;\JJ)$ to be zero outside $\JJ$. Similarly, if we fix $x\in\Go$, then $\II_x=\II^x=\II_x^x$, and $\JJ_x=\JJ^x=\JJ_x^x$, and we can build a discrete $R$-twist $(\JJ_x, i_x, q_x)$ over $\II_x$, where $i_x=i\restr{\{x\}\times R^\times}$ and $q_x=q\restr{\JJ_x}$.

\begin{lemma} \label{lem:unit cocycle}
Let $R$ be a commutative unital ring and let $(\Sigma,i,q)$ be a discrete $R$-twist over an ample Hausdorff groupoid $G$. Let $\II$ denote the interior of the isotropy in $G$ and let $\JJ \coloneqq q^{-1}(\II)$, which is the interior of the isotropy in $\Sigma$. Fix $x \in \Go$, and let $\zeta\colon \II_x \to \JJ_x$ be a section for $q\restr{\JJ_x}$ such that $\zeta(x) = x$. Define $c_x\colon \II_x \times \II_x \to R^\times$ by $\zeta(\alpha)\zeta(\beta)= c_x(\alpha, \beta) \cdot\zeta(\alpha\beta)$. Then $c_x$ is a normalised $2$-cocycle on $\II_x$, and there is an isomorphism $\rho_\zeta\colon A_R(\II_x;\JJ_x) \to R(\II_x,c_x)$ of $R$-algebras such that $\rho_\zeta(f)(\beta) = f(\zeta(\beta))$ for all $f\in A_R(\II_x;\JJ_x)$ and $\beta \in \II_x$.
\end{lemma}

\begin{proof}
It is routine to verify that each $c_x$ is a normalised $2$-cocycle using associativity of multiplication in $\Sigma$ and that $\zeta$ preserves the identity. Certainly there is an $R$-linear map $\rho_\zeta$ satisfying the given formula. To see that $\rho_\zeta$ is injective, suppose that $\rho_\zeta(f) = 0$. Then for each $\sigma \in \JJ_x$, we have $\sigma = t \cdot \zeta(q(\sigma))$ for some $t \in R^\times$, and then $f(\sigma) = f(t\cdot\zeta(q(\sigma))) = t^{-1}f(\zeta(q(\sigma))) = t^{-1} \rho_\zeta(f)(q(\sigma)) = 0$. To see that $\rho_\zeta$ is surjective, fix $h \in R(\II_x,c_x)$. For each $\sigma \in \JJ_x$, there is a unique $t_\sigma \in R^\times$ such that $\sigma = t_\sigma \cdot \zeta(q(\sigma))$ and $t_{\sigma} = 1$ for $\sigma \in \zeta(\II_x)$. Define $f\colon \JJ_x \to R$ by $f(\sigma) = t_\sigma^{-1} h(q(\sigma))$ for all $\sigma \in \JJ_x$. Then $f \in A_R(\II_x;\JJ_x)$, and $\rho_{\zeta}(f)=h$.

It remains to check that $\rho_\zeta$ is multiplicative. Fix $f,g \in A_R(\II_x;\JJ_x)$ and $\beta \in \II_x$. Then
\[
\rho_\zeta(f * g)(\beta) = (f * g)(\zeta(\beta)).
\]
By the definition of convolution in $A_R(\II_x;\JJ_x)$, the convolution product $(f*g)(\zeta(\beta))$ can be computed by choosing any section from $\II_x$ to $\JJ_x$, so we obtain
\[
\rho_\zeta(f * g)(\beta) = \sum_{\alpha \in \II^x} f(\zeta(\alpha)) \, g(\zeta(\alpha)^{-1}\zeta(\beta)).
\]
For each $\alpha \in \II^x$, we have $\zeta(\alpha)\zeta(\alpha^{-1}) = c_x(\alpha,\alpha^{-1})\zeta(\alpha\alpha^{-1}) = c_x(\alpha, \alpha^{-1}) \cdot x$, and rearranging this shows that $\zeta(\alpha)^{-1} = c_x(\alpha, \alpha^{-1})^{-1}\zeta(\alpha^{-1})$. Therefore, since $\zeta(\alpha^{-1}) \zeta(\beta) = c_x(\alpha^{-1},\beta) \cdot \zeta(\alpha^{-1}\beta)$ for all $\alpha \in \II^x$, we have
\begin{align*}
\rho_\zeta(f * g)(\beta)
&= \sum_{\alpha \in \II^x} c_x(\alpha, \alpha^{-1}) \, f(\zeta(\alpha)) \, g(\zeta(\alpha^{-1})\zeta(\beta)) \\
&= \sum_{\alpha \in \II^x} c_x(\alpha, \alpha^{-1}) \, c_x(\alpha^{-1}, \beta)^{-1} \, f(\zeta(\alpha)) \, g(\zeta(\alpha^{-1}\beta)).
\end{align*}
The $2$-cocycle identity and that $c_x$ is normalised give
\[
c_x(\alpha^{-1},\beta)c_x(\alpha, \alpha^{-1}\beta) = c_x(\alpha, \alpha^{-1})c_x(\alpha\alpha^{-1},\beta) = c_x(\alpha,\alpha^{-1}).
\]
Thus, rearranging gives $c_x(\alpha, \alpha^{-1}) c_x(\alpha^{-1},\beta)^{-1} =
c_x(\alpha,\alpha^{-1}\beta)$. Hence
\begin{align*}
\rho_\zeta(f * g)(\beta)
&= \sum_{\alpha \in \II^x} c_x(\alpha, \alpha^{-1}\beta) f(\zeta(\alpha)) g(\zeta(\alpha^{-1}\beta)) \\
&= \sum_{\alpha \in \II^x} c_x(\alpha, \alpha^{-1}\beta) \rho_\zeta(f)(\alpha) \rho_\zeta(g)(\alpha^{-1}\beta) \\
&= (\rho_\zeta(f) * \rho_\zeta(g))(\beta).\qedhere
\end{align*}
\end{proof}

Changing the section $\zeta$ in \cref{lem:unit cocycle} results in a cohomologous normalised $2$-cocycle and hence it does not change the algebra (up to isomorphism), nor does it affect whether or not the twisted group ring has only trivial units. In fact, a twisted group ring is graded by the underlying group, and the property of having only trivial units means that all the units are homogeneous elements; the isomorphism of algebras corresponding to replacing a $2$-cocycle by a cohomologous one does not change the grading.

\begin{remark} \label{rmk:Steinberg4.5}
Let $R$ be an indecomposable commutative ring and let $(\Sigma,i,q)$ be a discrete $R$-twist over an ample Hausdorff groupoid $G$. The argument of \cite[Proposition~4.5]{Steinberg2019} shows that if $n$ is a normaliser of $A_R(\Go; q^{-1}(\Go))$ in $A_R(G;\Sigma)$, then $n^\dagger n \in A_R(\Go; q^{-1}(\Go))$ is equal to $\tilde{1}_{\s(\supp(n))}$ and $n n^\dagger = \tilde{1}_{\r(\supp(n))}$.
\end{remark}

\begin{prop} \label{prop:trivial.units.good}
Let $R$ be an indecomposable commutative ring and let $(\Sigma,i,q)$ be a discrete $R$-twist over an ample Hausdorff groupoid $G$. Let $\II$ denote the interior of the isotropy in $G$ and let $\JJ \coloneqq q^{-1}(\II)$, which is the interior of the isotropy in $\Sigma$. Let $x \in \Go$, let $\II_x$ and $c_x$ be as in \cref{lem:unit cocycle}, and suppose that the twisted group ring $R(\II_x,c_x)$ has only trivial units. Then for any normaliser $n$ of $A_R(\Go; q^{-1}(\Go))$ in $A_R(\II;\JJ)$, we have $\lvert \supp_{\II}(n)\cap \II_x \rvert \le 1$.
\end{prop}

\begin{proof}
Suppose that $n\restr{\JJ_x} \ne 0$. We must show that $n\restr{\JJ_x}$ is supported on $q^{-1}(\mu)$ for some $\mu \in \II$. Fix $\alpha \in \JJ_x$ with $n(\alpha) \ne 0$. \Cref{rmk:Steinberg4.5} gives $n^\dagger n = 1_{\s(\supp(n))} = 1_{\r(\supp(n))} = nn^\dagger$. So we obtain $(n^\dagger n)(\s(\alpha)) = 1 = (nn^\dagger)(\s(\alpha))$. Fix a section $\zeta\colon \II_x \to \JJ_x$ and let $\rho_\zeta\colon A_R(\II_x;\JJ_x) \to R(\II_x,c_x)$ be the isomorphism of \cref{lem:unit cocycle}. Since $\JJ= q^{-1}(\II)$ consists of isotropy, the convolution product $n^\dagger n$ evaluated at $x$ is the same as $n^\dagger\restr{\JJ_x} * n\restr{\JJ_x}$ at $x$ (and similarly for $nn^\dagger$), and we obtain $n^\dagger\restr{\JJ_x} * n\restr{\JJ_x} = 1_{A_R(\II_x;\JJ_x)} = n\restr{\JJ_x} * n^\dagger\restr{\JJ_x}$. It follows that $\rho_{\zeta}(n\restr{\JJ_x})$ is a unit, and therefore it has the form $t \delta_\mu$ for some $\mu \in \II_x$. So $n\restr{\JJ_x} = t \rho_\zeta^{-1}(\delta_\mu)$ is supported on $R^\times \cdot \zeta(\mu) = q^{-1}(\mu)$, as required.
\end{proof}

The following lemma can be viewed as a far-reaching extension of \cref{exam:no.non.triv.units}.

\begin{lemma} \label{lem:iso LBH}
Let $R$ be an indecomposable commutative ring and let $(\Sigma,i,q)$ be a discrete $R$-twist over an ample Hausdorff groupoid $G$. Let $\II$ denote the interior of the isotropy in $G$ and $\JJ \coloneqq q^{-1}(\II)$, which is the interior of the isotropy in $\Sigma$. For each $x \in \Go$, let $\II_x$ and $c_x$ be as in \cref{lem:unit cocycle}. Suppose that the set of units $x \in \Go$ for which $R(\II_x,c_x)^{\times} = \{ t\delta_\alpha : t \in R^\times \text{ and } \alpha \in \II_x \}$ is dense in $\Go$. Then $(\JJ,i,q\restr{\JJ})$ satisfies the local bisection hypothesis.
\end{lemma}

\begin{proof}
Suppose that $n$ is a normaliser of $A_R(\Go; q^{-1}(\Go))$ in $A_R(\II;\JJ)$. We must show that $\supp_{\II}(n)$ is a bisection. Suppose for contradiction that this is not the case. Then we can find $\alpha,\beta\in \supp_{\II}(n)$ with $\s(\alpha)=\s(\beta)$ but $\alpha\ne \beta$. Since $\II$ is Hausdorff and $\supp_{\II}(n)$ is open, we can find disjoint compact open bisections $U$ and $V$ of $\II$ with $\alpha\in U$, $\beta\in V$, and $U,V\subseteq \supp_{\II}(n)$. By assumption, $\s(U)\cap \s(V)\ne\varnothing$ as $\s(\alpha)=\s(\beta)$ is in the intersection. Since $\s(U) \cap \s(V)$ is an open subset of $\Go$, there exists $x\in \s(U)\cap \s(V)$ such that the twisted group ring $R(\II_x,c_x)$ has only trivial units. Take $\alpha'\in U$ and $\beta'\in V$ with $\s(\alpha')=x=\s(\beta')$. Then $\alpha'=\beta'$ by \cref{prop:trivial.units.good}, contradicting that $U$ and $V$ are disjoint.
\end{proof}

The next lemma shows that the local bisection hypothesis can be checked on the interior of the isotropy.

\begin{lemma} \label{lem:global LBH}
Let $R$ be an indecomposable commutative ring and let $(\Sigma,i,q)$ be a discrete $R$-twist over an ample Hausdorff groupoid $G$. Let $\II$ denote the interior of the isotropy in $G$ and let $\JJ \coloneqq q^{-1}(\II)$, which is the interior of the isotropy in $\Sigma$. Let $A \coloneqq A_R(G;\Sigma)$, let $A' \coloneqq A_R(\II;\JJ)$, and let $B \coloneqq A_R(\Go; q^{-1}(\Go))$. Then
\begin{enumerate}[label=(\alph*)]
\item \label{item1:global LBH} the normaliser $N_{A'}(B)$ of $B$ in $A'$ satisfies
\[
N_{A'}(B) = N_A(B) \cap A';\text{ and}
\]
\item \label{item2:global LBH} $(\Sigma,i,q)$ satisfies the local bisection hypothesis if and only if $(\JJ, i, q\restr{\JJ})$ does.
\end{enumerate}
\end{lemma}

\begin{proof}
The proof very closely follows the arguments of \cite[Proposition~4.7, Corollary~4.8, and Proposition~4.10]{Steinberg2019}, with just minor adjustments to incorporate the generality of twists.

Claim~1 (cf.~\cite[Proposition~4.7]{Steinberg2019}): if $n \in N_A(B)$ and $\sigma,\tau \in \supp(n)$, then $\s(\sigma) = \s(\tau)$ if and only if $\r(\sigma) = \r(\tau)$. We suppose that $x = \s(\sigma) = \s(\tau)$ and prove that $\r(\sigma) = \r(\tau)$; the converse follows from a symmetric argument. Suppose for contradiction that $y \coloneqq \r(\sigma) \ne \r(\tau) \eqqcolon z$. Fix disjoint compact open neighbourhoods $y \in U$ and $z \in V$ in $\Sigmao$, and let $p \coloneqq \tilde{1}_U$ and $q \coloneqq \tilde{1}_V$. Then $pn, qn \in N_A(B)$.

\Cref{rmk:Steinberg4.5} gives $(n^\dagger p n)(x) = 1$ because $(pn)(\sigma) \ne 0$ and $(pn)^\dagger (pn) = n^\dagger pn$. Similarly, $(n^\dagger q n)(x) = 1$ because $(qn)(\tau) \ne 0$. Since $p+q=\tilde{1}_{U\cup V}\in I(B)$ by the disjointness of $U$ and $V$, we have by the same reasoning that $(n^\dagger (p+q) n)(x)= 1$ because $((p+q)n)(\sigma) \ne 0$. Hence
\[
2 = (n^\dagger p n)(x) + (n^\dagger q n)(x) = (n^\dagger(p+q)n)(x) = 1,
\]
which is a contradiction.

Claim~2 (cf.~\cite[Corollary~4.8]{Steinberg2019}): if $n \in N_A(B)$, then $\supp(n)\supp(n)^{-1}$ and $\supp(n)^{-1}\supp(n)$ are contained in $\JJ$. For this, suppose that $\sigma \in \supp(n)$ and $\tau \in \supp(n)^{-1}$ and $\s(\sigma) = \r(\tau)$. Then $\tau^{-1} \in \supp(n)$ and $\s(\sigma) = \s(\tau^{-1})$, so Claim~1 shows that $\r(\sigma) = \r(\tau^{-1})$. So $\sigma\tau$ is in the isotropy of $\Sigma$. It follows that the open set $\supp(n)\supp(n)^{-1}$ is contained in the isotropy, and hence $\supp(n)\supp(n)^{-1} \subseteq \JJ$. A symmetric argument shows that $\supp(n)^{-1}\supp(n) \subseteq \JJ$.

Now for part~\cref{item1:global LBH}, note that $N_{A'}(B)$ is clearly contained in $N_A(B) \cap A'$. For the reverse containment, it suffices to show that if $n \in N_A(B) \cap A'$, then $\supp(n^\dagger) \subseteq \JJ$. So suppose that $\sigma \in \supp(n^\dagger)$. Then \cref{rmk:Steinberg4.5} shows that $(nn^\dagger)(\s(\sigma)) = 1$. So any section $\zeta\colon G_{\s(\sigma)} \to \Sigma_{\s(\sigma)}$ satisfies
\[
1 = (nn^\dagger)(\s(\sigma)) = \sum_{\alpha \in G_{\s(\sigma)}} n(\zeta(\alpha)^{-1}) \, n^\dagger(\zeta(\alpha)).
\]
Since the sum is nonzero, at least one term, say corresponding to $\alpha_0\in G_{\s(\sigma)}$, satisfies $\zeta(\alpha_0)^{-1}\in \supp(n)$ and $\zeta(\alpha_0)\in \supp(n^\dagger)$. Applying Claim~1 to $n^\dagger$, we see that $\r(\alpha_0) = \r(\sigma)$. Since $\supp(n) \subseteq \JJ$ by hypothesis on $n$, we have $\alpha_0^{-1} \in \II$, and so we deduce that $\s(\sigma) = \s(\alpha_0) = \r(\alpha_0) = \r(\sigma)$, as required.

For part~\cref{item2:global LBH}, the ``only if'' implication is clear, so we suppose that $(\JJ, i, q\restr{\JJ})$ satisfies the local bisection hypothesis, and establish that $(\Sigma,i,q)$ does too. Fix $n \in N_A(B)$. Suppose that $\sigma, \tau \in \supp(n)$ satisfy $\s(\sigma) = \s(\tau)$. Then $\r(\sigma) = \r(\tau)$ by Claim~1. We must show that $q(\sigma) = q(\tau)$. Since $n$ is locally constant, there is a compact open bisection $U$ containing $\sigma$ such that $n\restr{U}$ is constant. Claim~2 implies that $U^{-1}\supp(n) \subseteq \JJ$. Let $m \coloneqq \tilde{1}_{U^{-1}}$. Then $mn$ belongs to $N_A(B) \cap A'$, and so part~\cref{item1:global LBH} shows that $mn \in N_{A'}(B)$. Since $(\JJ, i, q\restr{\JJ})$ satisfies the local bisection hypothesis, we deduce that $\supp_G(mn) = \supp_\II(mn)$ is a bisection. Let $x = \s(\sigma)$. We claim that $x, \sigma^{-1}\tau \in \supp(mn)$. To see this, fix a section $S\colon \supp_G(m)=q(U^{-1})\to \supp(m)=R^\times\cdot U^{-1}$ with $S(q(\sigma))=\sigma$. Since $U$ is a bisection, we have
\[
(mn)(x) = \sum_{\gamma \in G^x} m(S(\gamma)) \, n(S(\gamma)^{-1}) = n(S(q(\sigma))) = n(\sigma) \ne 0,
\]
and
\[
(mn)(\sigma^{-1}\tau) = \sum_{\gamma \in G^x} m(S(\gamma)) \, n(S(\gamma)^{-1}\sigma^{-1}\tau) = n(\tau) \ne 0,
\]
as claimed. Since $q(\supp(mn))$ is a bisection and $\s(x) = x = \s(\sigma^{-1}\tau)$, we deduce that $q(\sigma^{-1}\tau) = q(x)$, and hence $q(\sigma)=q(\tau)$.
\end{proof}

We have seen that twists satisfying the local bisection hypothesis give rise to algebraic quasi-Cartan pairs. To finish the section we show that if the underlying groupoid $G$ is effective, then we obtain an algebraic Cartan pair, and if it is principal, then we obtain an
algebraic diagonal pair.

\begin{prop} \label{prop:effectiveACPprincipalADP}
Let $R$ be an indecomposable commutative ring and let $(\Sigma,i,q)$ be a discrete $R$-twist over an ample Hausdorff groupoid $G$. If $G$ is effective then $(A,B) \coloneqq (A_R(G;\Sigma), A_R(\Go; q^{-1}(\Go)))$ is an algebraic Cartan pair, and if $G$ is principal then $(A,B)$ is an algebraic diagonal pair.
\end{prop}

\begin{proof}
If $G$ is effective, then \cref{lem:global LBH}\cref{item2:global LBH} implies that the local bisection hypothesis holds. Then by \cref{lem:LBH<->AQP}, $(A,B)$ is an algebraic quasi-Cartan pair with conditional expectation $P\colon A \to B$ given by restriction of functions from $\Sigma$ to $q^{-1}(\Go)$. So, we just have to show that if $G$ is effective, then $(A,B)$ satisfies~\labelcref{item:ACP} and that if $G$ is principal then $(A,B)$ satisfies~\labelcref{item:ADP}.

First suppose that $G$ is effective. Then for $f \in A {\setminus} B$, we must show that $f$ does not commute with $B$. Since $f \notin B$, there exists $\sigma \in \Sigma \setminus q^{-1}(\Go)$ such that $f(\sigma) \ne 0$. Then $U \coloneqq f^{-1}(f(\sigma))$ is open. Since $G$ is effective, $q(U)$ is open, and $q(\sigma) \notin \Go$, \cite[Lemma~3.1(3)]{BCFS2014} implies that there exists $\tau \in U$ such that $\r(\tau) \ne \s(\tau)$. Fix a compact open set $V \subseteq \Go$ that contains $\s(\tau)$ but not $\r(\tau)$, and let $e_V \coloneqq \tilde{1}_V \in B$. Then $(fe_V)(\tau) = f(\tau)=f(\sigma) \ne 0 = (e_V f)(\tau)$. So $f e_V \ne e_V f$, and hence $f$ does not commute with $B$, as required.

Now suppose that $G$ is principal. Fix a compact open bisection $U \subseteq \Sigma$. By \cref{prop:f is a finite sum of lifts of indicators}, it suffices to show that $\tilde{1}_U$ is a linear combination of free normalisers. Let $U_0 \coloneqq U \cap q^{-1}(\Go)$ and $V \coloneqq U {\setminus} U_0$. Both $U_0$ and $V$ are compact open bisections, and we have $\tilde{1}_U = \tilde{1}_{U_0} + \tilde{1}_V$. Since $\tilde{1}_{U_0} \in B$ is a free normaliser, we just have to show that $\tilde{1}_V$ is a linear combination of free normalisers. Since $q(V) \cap \Go$ is empty and $G$ is principal, we have $\r(\sigma) \ne \s(\sigma)$ for all $\sigma \in V$. Thus, for each $\sigma\in V$, we can find a compact open bisection $V_\sigma\subseteq V$ containing $\sigma$ with $V_\sigma^2=\varnothing$ by choosing disjoint compact open neighbourhoods $U, U' \subseteq \Go$ of $\r(\sigma)$ and $\s(\sigma)$, respectively, and putting $V_\sigma \coloneqq UVU'\subseteq V$. By compactness, we can cover $V$ by finitely many compact open bisections $V_1,\dotsc, V_n$ such that $V_i^2 = \varnothing$ for $i \in \{1,\dotsc,n\}$. Then by putting $V_i' = V_i \setminus \bigcup_{j=1}^{i-1} V_j'$ for each $i \in \{1,\dotsc,n\}$, we obtain a refined cover of $V$ consisting of mutually disjoint sets. Now each $\tilde{1}_{V_i'}$ is a free normaliser, and $\tilde{1}_V = \sum_{i=1}^n \tilde{1}_{V_i'}$.
\end{proof}

\section{Building a twist from a pair of algebras} \label{sec:build twist}

Throughout this section we assume that $A$ is an $R$-algebra with $B$ a commutative subalgebra without torsion in the sense of \cref{cond:torsion free} (that is, for $e \in I(B)$ and $t\in R$, if $te=0$, then $t=0$ or $e=0$), and satisfying property~\cref{item:ACP.local units} of \cref{def:ACP} (that is, $I(B)$ is a set of local units for $A$).

\subsection{The groupoid \texorpdfstring{$\Sigma$}{of ultrafilters}} \label{sec:Sigma def}

In this section we use the partial order $\le$ on $N(B)$ defined in \cref{sec:invsemigroup} (on \cpageref{def:partial order}) to define the groupoid of ultrafilters on $N(B)$, which will play the same role as Renault's Weyl twist in \cite{Renault2008}.

If $U \subseteq N(B)$, then the \emph{upclosure} of $U$ is the set
\[
U^{\uparrow} \coloneqq \{m \in N(B) : \text{there exists } n \in U \text{ with } n \le m\}.
\]
A \emph{filter} of $N(B)$ is a subset $U \subseteq N(B) {\setminus} \{0\}$ such that $U = U^{\uparrow}$ and whenever $m,n \in U$ there exists $p \in U$ such that $p \le m,n$. An \emph{ultrafilter} is a maximal filter. The collection $\Sigma$ of all ultrafilters of $N(B)$ forms a groupoid with the following structure (see \cite[Proposition~9.2.1]{Lawson1998} and \cite[Proposition~2.13]{Lawson2010}). For each $U \in \Sigma$,
\[
U^{-1} \coloneqq \{ n^\dagger : n \in U\}.
\]
A pair $(U,V)$ of ultrafilters of $N(B)$ is composable if and only if $m^\dagger mnn^\dagger \ne 0$ (or, equivalently, if and only if $mn\ne 0$) for all $m \in U$ and $n \in V$, and then
\[
UV \coloneqq \{ mn : m \in U, n \in V\}^{\uparrow}.
\]
The definition of an ultrafilter ensures that a pair $(U,V)$ of ultrafilters is composable if and only if $\s(U) \coloneqq U^{-1}U$ is equal to $\r(V) \coloneqq VV^{-1}$. It follows that
\[
\Sigmao = \{U \in \Sigma : U \cap I(B) \ne \varnothing\}.
\]
For $n \in N(B)$, we write
\[
\VV_n\coloneqq \{U \in \Sigma : n \in U\}.
\]
The collection $\{\VV_n: n \in N(B)\}$ forms a basis of open bisections for a topology on $\Sigma$, making it an \'etale groupoid. We have $\VV_n \VV_m = \VV_{nm}$ and $\VV_n^{-1} = \VV_{n^\dagger}$ for all $m,n \in N(B)$. For more details, see \cite[Lemma~3.2]{ACaHJL2021}. By \cite[Lemma~2.22]{Lawson2012} (see also \cite[Propositions~2.2 and~4.4]{ACaHJL2021}), $\Sigmao$ is a Hausdorff subspace of $\Sigma$. In \cref{lem:scalarbasis}, we show how scalar multiplication in the algebra interacts with the groupoid structure.

\begin{remark} \label{rem:ultrafilters on I(B)}
Note that \cite[Proposition~2.2 and Proposition~4.4]{ACaHJL2021} say that $U \mapsto U \cap I(B)$ is a homeomorphism from the set of ultrafilters of $N(B)$ that contain an element of $I(B)$ to the set of ultrafilters of $I(B)$, and so we often identify elements of $\Sigmao$ with ultrafilters of $I(B)$ when convenient. In particular, since $I(B)$ is a Boolean algebra, Stone duality implies that $\VV_e$ is compact for each $e \in I(B)$.
\end{remark}

We will frequently use the following standard fact about ultrafilters. This fact, and much of what we say about the groupoid $\Sigma$ as an ample groupoid, would immediately follow from the results of~\cite{LawsonLenz2013} if one proved that $N(B)$ is a Boolean inverse semigroup (called a weakly Boolean inverse semigroup in~\cite{LawsonLenz2013}). However, to prove this would go too far afield.

\begin{lemma} \label{lem:ultrafilter partition}
Suppose that $A$ is an $R$-algebra, and that $B \subseteq A$ is a commutative subalgebra satisfying \cref{cond:torsion free} and \cref{item:ACP.local units} of \cref{def:ACP}. Let $\Sigma$ be the groupoid of ultrafilters of $N(B)$. Suppose that $n \in U \in \Sigma$ and that $e_1, \dotsc, e_k \in I(B)$ are mutually orthogonal idempotents satisfying $n = \sum_{i=1}^k ne_i$. Then there exists a unique $i \in \{1,\dotsc,k\}$ such that $ne_i \in U$. Dually, if $n = \sum_{i=1}^k f_in$ for mutually orthogonal idempotents $f_1,\dotsc, f_k\in I(B)$, then $f_in \in U$ for a unique $i \in \{1,\dotsc,k\}$.
\end{lemma}

\begin{proof}
Note that $ne_i=nn^\dagger ne_i$ for each $i$, and that the $n^\dagger ne_i\in I(B)$ are mutually orthogonal idempotents. So by replacing each $e_i$ with $n^\dagger n e_i$, we may assume that $n^\dagger ne_i=e_i$ for each $i \in \{1,\dotsc,k\}$, and hence $n^\dagger n=\sum_{i=1}^ke_i$. Since $n^\dagger n\in U^{-1}U\cap I(B)$, which is an ultrafilter of the Boolean algebra $I(B)$ by \cref{rem:ultrafilters on I(B)}, it follows that there is a unique $i \in \{1,\dotsc,k\}$ with $e_i \in U^{-1}U \cap I(B)$ (since the characteristic function of an ultrafilter on a Boolean algebra is a Boolean algebra homomorphism). Then $ne_i\in U(U^{-1}U)=U$. Moreover, if $ne_j\in U$ for some $j \in \{1,\dotsc,k\}$, then $e_j=(ne_j)^\dagger ne_j\in U^{-1}U \cap I(B)$, and so $j=i$. This proves the first statement.

For the dual result, note that $n = \sum_{i=1}^k f_i n = \sum_{i=1}^k n(n^\dagger f_i n)$, and that the $n^\dagger f_in$ are mutually orthogonal idempotents in $I(B)$. Therefore, $f_in=n(n^\dagger f_in)\in U$ for a unique $i \in \{1,\dotsc,k\}$ by the previous paragraph.
\end{proof}

\begin{lemma} \label{lem:scalarbasis}
Suppose that $A$ is an $R$-algebra, and that $B \subseteq A$ is a commutative subalgebra satisfying \cref{cond:torsion free} and satisfying property~\cref{item:ACP.local units} of \cref{def:ACP}. Let $\Sigma$ be the groupoid of ultrafilters of $N(B)$. For $t, s \in R^\times$ and $n \in N(B)$, we have the following.
\begin{enumerate}[label=(\alph*)]
\item \label{item1:scalarbasis}If $U \in \Sigma$, then $t U\coloneqq\{t m:m\in U\} \in \Sigma$.
\item \label{item2:scalarbasis} If $U \in \Sigma$, then $(t U)^{-1} = t^{-1}U^{-1}$.
\item \label{item3:scalarbasis} If $(U,W) \in \Sigmac$, then $(t U,s W) \in \Sigmac$ and $(t U)(s W) = (t s)(UW)$.
\item \label{item4:scalarbasis} If $U \in \Sigma$, then $\s(U)=\s(t U)$ and $\r(U) = \r(t U)$.
\item \label{item5:scalarbasis} $t \VV_n \coloneqq\{t U : U \in \VV_n\} =\VV_{tn}$.
\item \label{item6:scalarbasis} If $t \ne 1$, then $\VV_{t n} \cap \VV_n = \varnothing$.
\item \label{item7:scalarbasis} If $U,W \in \Sigmao$ with $tU = sW$, then $t=s$ and $U=W$.
\end{enumerate}
\end{lemma}

\begin{proof}
For part~\cref{item1:scalarbasis}, notice that $tU \subseteq N(B)$ by \cref{lem:scalarn}\cref{item1:scalarn}. We first show that $t U$ is a filter. Fix $t m, t n \in t U$. Since $U$ is a filter, there exists $k \in U$ such that $k \le m$ and $k \le n$. Then $tk \in tU$, and \cref{lem:scalarn}\cref{item2:scalarn} implies that $t k \le t m$ and $t k \le t n$. Next suppose that $p \in N(B)$ and that there exists $t m \in t U$ with $t m \le p$. Then $m \le t^{-1}p$ by \cref{lem:scalarn}\cref{item2:scalarn} again. Again using that $U$ is a filter, we have $t^{-1}p \in U$, and hence $p = t t^{-1}p \in t U$. Also, $0 \notin tU$, for otherwise $0=t^{-1}0\in U$. Hence $t U$ is a filter.

To see that $t U$ is an ultrafilter, suppose that $t U \subseteq W$, where $W$ is a filter. Then $U \subseteq t^{-1}W$. Since $U$ is an ultrafilter and $t^{-1}W$ is a filter by our previous argument, $U = t^{-1}W$, and hence $t U = W$.

For part~\cref{item2:scalarbasis}, \cref{lem:scalarn}\cref{item1:scalarn} implies that
\[
(t U)^{-1} = \{(t m)^\dagger : m \in U\} = \{t^{-1}m^\dagger : m \in U\} = t^{-1}U^{-1}.
\]

For part~\cref{item3:scalarbasis}, fix $(U,W) \in \Sigmac$. We begin by verifying that $tU$ and $sW$ are composable. Suppose that $m\in tU$ and $n\in sW$. Then $m=tm'$ with $m'\in U$ and $n=sn'$ with $n'\in W$. Since $(U,W)\in \Sigmac$, we have that $m'n'\ne 0$. Therefore, $mn=(ts)m'n'\ne 0$ since $ts \in R^\times$. Thus $(tU,sW) \in \Sigmac$. Since $(t s)(UW)$ is an ultrafilter by part~\cref{item1:scalarbasis}, it suffices to show that $(ts)(UW) \subseteq (t U)(s W)$. Fix $k \in (ts)(UW)$. Then there exists $m\in UW$ such that $k=(ts)m$. Therefore, there exist $n\in U$ and $p\in W$ such that $np\le m$, and so $(tn)(sp)=(ts)(np)\le (ts)m=k$. Hence $k\in (tU)(sW)$, and so $(t U)(s W) = (t s)(UW)$.

Item~\cref{item4:scalarbasis} follows easily from parts~\cref{item2:scalarbasis,item3:scalarbasis}.

For part~\cref{item5:scalarbasis}, the containment $t \VV_n \subseteq \VV_{t n}$ follows from part~\cref{item1:scalarbasis}. For the reverse containment, fix $W \in \VV_{tn}$. Then $t n \in W \in \Sigma$, and hence $n \in t ^{-1}W \in \Sigma$ by part~\cref{item1:scalarbasis}. So $t^{-1} W \in \VV_n$, and $W = t (t ^{-1}W) \in t \VV_n$.

For part~\cref{item6:scalarbasis}, we prove the contrapositive. For this, suppose that $U \in \VV_{t n} \cap \VV_n$. Then $n \ne 0$ and $tn, n \in U$, and so there exists $m \in U$ such that $0 \ne m \le t n, n$. Thus $0 \ne mm^\dagger\le tnn^\dagger$, and so $t=1$ by \cref{lem:properties of le}\cref{item2:properties of le}.

For part~\cref{item7:scalarbasis}, suppose that $U,W \in \Sigmao$ satisfy $tU = sW$. Then $U = (t^{-1}s)W$, and so it suffices to show that $t^{-1}s = 1$. For this, fix $e_1, e_2 \in I(B)$ with $e_1 \in U$ and $e_2 \in W$. Then $t^{-1}s e_2 \in U$, and so there exists $f \in U$ such that $0 \ne f \le e_1, t^{-1}s e_2$. Now by \cref{lem:properties of le}\cref{item1:properties of le}, we have $f \in I(B)$, and so \cref{lem:properties of le}\cref{item2:properties of le} implies that $t^{-1}s = 1$.
\end{proof}

\subsection{The groupoid \texorpdfstring{$G$}{G}} \label{sec:G def}

Let $A$ be an $R$-algebra, and let $B$ be a commutative subalgebra of $A$ satisfying \cref{cond:torsion free} and satisfying property \cref{item:ACP.local units} of \cref{def:ACP}. We define a relation $\simeq$ on $\Sigma$, the groupoid of ultrafilters of the inverse semigroup $N(B)$, by
\begin{equation} \label{eqn:scalar equivalence}
U \simeq W \iff \text{there exists } t \in R^\times \text{ such that } U = t W.
\end{equation}
Then $\simeq$ is an equivalence relation. We define $G$ to be the quotient $\Sigma/{\simeq}$ endowed with the quotient topology, and we denote the corresponding quotient map by $q\colon \Sigma \to G$. In the following lemma we show that $G$ is a groupoid that inherits its structure from $\Sigma$.

\begin{lemma} \label{lem:Ggroupoid}
Suppose that $A$ is an $R$-algebra, and that $B \subseteq A$ is a commutative subalgebra satisfying \cref{cond:torsion free} and satisfying property~\cref{item:ACP.local units} of \cref{def:ACP}. Let $\Sigma$ be the groupoid of ultrafilters of $N(B)$, and let $q\colon\Sigma \to G$ be the quotient map defined above.
\begin{enumerate}[label=(\alph*)]
\item \label{item1:Ggroupoid} The quotient $G = \{q(U) : U \in \Sigma\}$ is a groupoid with inversion given by $q(U)^{-1} \coloneqq q(U^{-1})$, composable pairs $\Gc \coloneqq \{(q(U),q(W)) : (U,W) \in \Sigmac\}$, and composition given by $q(U) q(W) \coloneqq q(UW)$ for all $(U,W) \in \Sigmac$. We have $\s(q(U)) = q(\s(U))$ and $\r(q(U)) = q(\r(U))$ for all $U \in \Sigma$, and so $\Go = q(\Sigmao)$.
\item \label{item2:Ggroupoid} The quotient map $q\colon \Sigma \to G$ is a groupoid homomorphism.
\item \label{item3:Ggroupoid} The quotient map restricts to a bijection $q\restr{\Sigmao}\colon \Sigmao \to \Go$.
\end{enumerate}
\end{lemma}

\begin{proof}
For part~\cref{item1:Ggroupoid}, note that \cref{lem:scalarbasis}\cref{item2:scalarbasis} implies that the inverse is well-defined. Composability and the product are well-defined by \cref{lem:scalarbasis}\cref{item3:scalarbasis}. The remainder of the groupoid structure comes from the groupoid structure of $\Sigma$, and part~\cref{item2:Ggroupoid} follows as well. For part~\cref{item3:Ggroupoid}, first note that $q(\Sigmao) = \Go$ by part~\cref{item1:Ggroupoid}. Injectivity of $q\restr{\Sigmao}$ follows from \cref{lem:scalarbasis}\cref{item7:scalarbasis} with $t=1$.
\end{proof}

\begin{lemma} \label{lem:qmap}
Suppose that $A$ is an $R$-algebra, and that $B \subseteq A$ is a commutative subalgebra satisfying \cref{cond:torsion free} and satisfying property \cref{item:ACP.local units} of \cref{def:ACP}. Let $\Sigma$ be the groupoid of ultrafilters of $N(B)$, and let $q\colon \Sigma \to G$ be the quotient map defined above.
\begin{enumerate}[label=(\alph*)]
\item \label{item1:qmap} The collection $\{\VV_n: n \in N(B)\}$ forms a basis of compact open bisections for the topology on $\Sigma$. In particular, $\Sigma$ is an ample groupoid.
\item \label{item2:qmap} The quotient map is open and restricts to a homeomorphism of unit spaces.
\item \label{item3:qmap} The collection $\{q(\VV_n): n \in N(B)\}$ forms a basis of compact open bisections for the quotient topology on $G$. In particular, $G$ is an ample groupoid.
\end{enumerate}
\end{lemma}

\begin{proof}
For part~\cref{item1:qmap}, to see that $\Sigma$ is ample, we show that $\VV_n$ is compact for each $n \in N(B)$. Fix $n \in N(B)$. Then $\VV_n$ is homeomorphic to $\r(\VV_n)$ because $\VV_n$ is a bisection. Now
\[
\r(\VV_n) = \VV_n\VV_{n^\dagger} = \VV_{nn^\dagger},
\]
which is compact by \cref{rem:ultrafilters on I(B)}.

For part~\cref{item2:qmap}, since $G$ is endowed with the quotient topology by definition, to see that $q$ is open, it suffices to show that $q^{-1}(q(\VV_n))$ is open in $\Sigma$ for any $n \in N(B)$. Using \cref{lem:scalarbasis}\cref{item5:scalarbasis} for the last equality, we calculate
\[
q^{-1}(q(\VV_n)) = \{t U : t \in R^\times \text{ and } U \in \VV_n\} = \bigcup_{t \in R^\times} t\VV_n = \bigcup_{t \in R^\times} \VV_{t n},
\]
which is open in $\Sigma$. That $q$ restricts to a homeomorphism of unit spaces follows from \cref{lem:Ggroupoid}\cref{item3:Ggroupoid}.

For part~\cref{item3:qmap}, compactness of each $q(\VV_n)$ follows from part~\cref{item1:qmap}, because $q$ is continuous. Since $\{\VV_n : n \in N(B)\}$ is a basis for $\Sigma$ and $q$ is surjective, the collection $\{q(\VV_n):n \in N(B)\}$ covers $G$. Part~\cref{item2:qmap} shows that each $q(\VV_n)$ is open in $G$. So to see that $\{q(\VV_n) : n \in N(B)\}$ is a basis for the quotient topology, it suffices to show that for each open subset $O$ of $G$ and each element $q(U) \in O$, there exists $n \in N(B)$ such that
\[
q(U) \in q(\VV_n) \subseteq O.
\]
Since $O$ is open in the quotient topology on $G$, $q^{-1}(O)$ is open in $\Sigma$. Since $U \in q^{-1}(O)$ and $\{\VV_n : n \in N(B)\}$ is a basis for $\Sigma$, there exists $n \in N(B)$ such that $q(U) \in q(\VV_n) \subseteq O$, and so $\{q(\VV_n) : n \in N(B)\}$ is a basis. To see that the sets $q(\VV_n)$ are open bisections, we need to show that the source and range maps are injective on each $q(\VV_n)$. But this follows from parts~\cref{item1:Ggroupoid} and \cref{item3:Ggroupoid} of \cref{lem:Ggroupoid}, since $\VV_n$ is an open bisection. Now \cite[Proposition~6.6]{BS2019} implies that $G$ is a topological groupoid, and hence $G$ is ample.
\end{proof}

\subsection{The twist of an algebraic quasi-Cartan pair} \label{sec:twist def}

The main theorem of this section is that if $(A,B)$ is an algebraic quasi-Cartan pair, then $(\Sigma,i,q)$ is a discrete $R$-twist over $G$, and $G$ is Hausdorff.

\begin{theorem} \label{thm:the twist}
Suppose that $A$ is an $R$-algebra, and that $B \subseteq A$ is a commutative subalgebra satisfying \cref{cond:torsion free} and satisfying properties \cref{item:ACP.local units,item:ACP.IB spans} of \cref{def:ACP}. Let $\Sigma$ be the groupoid of ultrafilters of $N(B)$, let $G$ be the quotient of $\Sigma$ by the equivalence relation given in~\cref{eqn:scalar equivalence}, and let $q\colon \Sigma \to G$ be the quotient map. Define $i\colon \Go \times R^\times \to \Sigma$ by $i(q(U),t) = t U$ for $U \in \Sigmao$ and $t\in R^\times$. Then the sequence
\[
\Go \times R^\times \overset{i} \hookrightarrow \Sigma \overset{q} \twoheadrightarrow G
\]
is a discrete $R$-twist over $G$. If $(A,B)$ is an algebraic quasi-Cartan pair, then $G$ is Hausdorff.
\end{theorem}

Since $B$ is without torsion as in \cref{cond:torsion free}, \cite[Th\'eor\`eme~1]{Keimel1970} shows that there is an isomorphism $\phi\colon B \to A_R(\Sigmao)$ that satisfies
\begin{equation} \label{eqn:phi description}
\phi(e) = 1_{\VV_e} \text{ for all } e \in I(B).
\end{equation}

In \cite[Th\'eor\`eme~1]{Keimel1970}, Keimel asks that $B$ and $R$ satisfy the stronger property that if $tb=0$ with $t\in R$ and $b \in B\setminus \{0\}$, then $t=0$. But the only time this property is used in the proof is in \cite[2.13]{Keimel1970}, when, in fact, $b \in I(B)$. Moreover, in the proof of~\cite[Corollaire~3]{Keimel1970}, Keimel implicitly uses that \cref{cond:torsion free} is strong enough to obtain the result (and the corollary is, in fact, formally equivalent to the theorem holding under the weaker hypothesis \labelcref{cond:torsion free}).

To start the proof of \cref{thm:the twist}, we show that if $(A,B)$ is an algebraic quasi-Cartan pair then $G$ is Hausdorff (see \cref{prop:Hausdorff}). For this, we need the following lemma.

\begin{lemma} \label{lem:diagonal normalisers}
Suppose that $(A,B)$ is an algebraic quasi-Cartan pair. If $n \in N(B) \cap B$, then $n^\dagger \in B$, and $\phi(n)(U) \in R^\times \cup\{0\}$ for all $U \in \Sigmao$.
\end{lemma}

\begin{proof}
Suppose that $n \in N(B) \cap B$. Then
\[
n^\dagger = n^\dagger nn^\dagger = P(n^\dagger n)n^\dagger =P(n^\dagger)nn^\dagger = P(n^\dagger nn^\dagger)=P(n^\dagger)\in B,
\]
proving the first statement. Since $n^\dagger n$ is an idempotent element of $B$ and is a right identity for $n$, and since $R$ is indecomposable, we have $\phi(n^\dagger n) = 1_W$ for some compact open set $W$ containing the support of $\phi(n)$. Since $\phi$ is multiplicative, we deduce that $\phi(n^\dagger)(U) \, \phi(n)(U) = 1$ for all $U \in \supp(\phi(n))$, and so each $\phi(n^\dagger)(U)$ is an inverse for $\phi(n)(U)$ (because $R$ is commutative).
\end{proof}

The following consequence of \cref{lem:diagonal normalisers} is used frequently, often without comment.

\begin{cor} \label{cor:units.in.front.diag.norm}
Suppose that $(A,B)$ is an algebraic quasi-Cartan pair. Let $n\in N(B) \cap B$, and suppose that $n=\sum_{i=1}^kt_ie_i$, where $e_1, \dotsc, e_k$ are mutually orthogonal idempotents and $t_1, \dotsc, t_k \in R {\setminus} \{0\}$. Then $t_i \in R^\times$ for each $i \in \{1,\dotsc,k\}$.
\end{cor}

\begin{proof}
Choose an ultrafilter $W\in \VV_{e_i}$. Then $e_j\in W$ if and only if $j=i$, and so $\phi(n)(W) = t_i \in R^\times$ by \cref{lem:diagonal normalisers}.
\end{proof}

\begin{remark} \label{rem:N(B) cap B characterisation}
If $(A,B)$ is an algebraic quasi-Cartan pair, then since $B = \vecspan{(I(B))}$, every $n \in B$ can be expressed in the form $n = \sum_{i=1}^k t_i e_i$, where $e_1, \dotsc, e_k \in I(B)$ are mutually orthogonal idempotents and $t_1, \dotsc, t_k \in R$. It follows (using \cref{cor:units.in.front.diag.norm}) that $n \in N(B) \cap B$ if and only if $t_1, \dotsc, t_k \in R^\times \cup \{0\}$. Thus
\[
N(B) \cap B = \big\{ b \in B : \phi(b)(U) \in R^\times \cup \{0\} \text{ for all } U \in \Sigmao \big\}.
\]
\end{remark}

\begin{prop} \label{prop:Hausdorff}
Suppose that $(A,B)$ is an algebraic quasi-Cartan pair. Then $G$ is Hausdorff.
\end{prop}

\begin{proof}
It suffices by \cite[Lemma~8.3.2]{Sims2020} to prove that $\Go$ is closed. Since $q$ is a quotient map, it suffices to show that $q^{-1}(G {\setminus} \Go)$ is open. This set consists of all ultrafilters $U$ of $N(B)$ such that $U$ contains no elements of the form $te$ with $e\in I(B)$ and $t\in R^\times$. Let $U$ be such an ultrafilter, and fix $n \in U$. By~\cref{item:AQP.expectation}, there exists $f \in I(B)$ such that $P(n) = nf = fn\in N(B)\cap B$. Take mutually orthogonal idempotents $e_1, \dotsc, e_k \in I(B)$ and coefficients $t_1, \dotsc, t_k \in R{\setminus}\{0\}$ such that $nf=P(n) = \sum_{i=1}^kt_ie_i$, and note that $t_i \in R^\times$ for each $i \in \{1, \dotsc, k\}$, by \cref{cor:units.in.front.diag.norm}. Also, $nfe_i=t_ie_i$ for all $i \in \{1, \dotsc, k\}$, and so $nf=\sum_{i=1}^k nfe_i$.

Note that $n-nf \in N(B)$ by \cref{lem:properties of le}\cref{item5:properties of le}. Thus, since $n = n(n^\dagger n - n^\dagger nf)+nf$, we deduce that either $nf\in U$ or $n(n^\dagger n-n^\dagger nf)=n-nf\in U$ by \cref{lem:ultrafilter partition}. But if $nf\in U$, then since $nf=\sum_{i=1}^k nfe_i$, we must have $t_ie_i=nfe_i\in U$ for some $i \in \{1,\dotsc,k\}$ by \cref{lem:ultrafilter partition}, which is a contradiction to our hypothesis on $U$. Thus we must have $n-nf\in U$, and so $U\in \VV_{n-nf}$.

We claim that if $V\in \VV_{n-nf}$, then $V \notin q^{-1}(\Go)$. To see this, suppose for contradiction that $tg \in V$ for some $t \in R^\times$ and $g \in I(B)$. Then $tg$ and $n-nf$ have a common lower bound $m$ in $V$, and so $t^{-1}m\le g$. Hence $t^{-1}m \in I(B)$ by \cref{lem:properties of le}\cref{item1:properties of le}, and so $m=th$ for some idempotent $h\in I(B)$. Since $th = m \le n-nf \le n$, we have that $th = tht^{-1}hn = hn$. So
\[
th = P(th) = P(hn) = hP(n) = hnf = thf.
\]
But then $m = th = thf \le (n-nf)f = 0$, and so $m = 0$, which contradicts $m \in V$. Thus $V \notin q^{-1}(\Go)$, as required. We conclude that $q^{-1}(G {\setminus} \Go)$ is open and hence $\Go$ is closed, whence $G$ is Hausdorff.
\end{proof}

\begin{proof}[Proof of \cref{thm:the twist}]
\Cref{lem:Ggroupoid}\cref{item2:Ggroupoid} shows that $q$ is a groupoid homomorphism. It is continuous and surjective by definition. By \cref{lem:qmap}\cref{item2:qmap}, $q$ restricts to a homeomorphism of unit spaces. In what follows, we identify $\Go$ with $\Sigmao$ for cleaner notation.

To see that $i$ is a groupoid homomorphism, fix a composable pair $((U,t),(U,s)) \in (\Go \times
R^\times)^{(2)}$, where $U \in \Sigmao$. Then
\[
i((U,t)(U,s)) = i(U,ts) = ts U = ts UU.
\]
An application of \cref{lem:scalarbasis}\cref{item3:scalarbasis} shows that this is equal to $t U s U = i(U,t)i(U,s)$, and hence $i$ is a groupoid homomorphism.

To see that $i$ is continuous, fix a basic open set $\VV_n \subseteq \Sigma$. Then
\begin{align*}
i^{-1}(\VV_n) &= \{(U,t) \in \Sigmao \times R^\times : t U \in \VV_n\} \\
&= \{(U,t) \in \Sigmao \times R^\times : U \in \VV_{t^{-1}n}\} \quad \text{by \cref{lem:scalarbasis}\cref{item5:scalarbasis} } \\
&=\bigcup_{t \in R^\times} (\VV_{t^{-1}n} \cap \Sigmao) \times \{t\}.
\end{align*}
Since $\VV_{t^{-1}n} \cap \Sigmao$ is open in $\Sigmao \cong \Go$, we deduce that $i^{-1}(\VV_n)$ is open in $\Go \times R^\times$.

\Cref{lem:scalarbasis}\cref{item7:scalarbasis} shows that $i$ is injective, and clearly $i(\Go\times \{1\})=\Sigmao$. For the exactness condition~\cref{item:exact}, fix $U \in \Go$, and note that
\[
i(\{U\} \times R^\times) = \{t U : t \in R^\times\}= q^{-1}(U).
\]

We next check condition~\cref{item:centrality}, which requires that the image of $i$ is central in $\Sigma$. Fix $U \in \Sigma$ and $t \in R^\times$. Then we have
\[
i(\r(U), t)U = t (UU^{-1})U = t U = U(t(U^{-1}U)) = U i(\s(U),t)
\]
by \cref{lem:scalarbasis}\cref{item3:scalarbasis}.

We use the implication \mbox{\cref{item4:check.locally.trivial}$\implies$\cref{item1:check.locally.trivial}} of \cref{prop:check.locally.trivial} to complete the proof that the sequence satisfies \cref{item:exact}--\cref{item:centrality} (and hence is a discrete $R$-twist). We have that $\Sigma$ is ample and that the map $q$ is open by parts~\cref{item1:qmap,item2:qmap} of \cref{lem:qmap}. We show that $i$ is open. Fix a basic open set $\VV_e\times \{t\}$ with $e \in I(B)$ and $t\in R^{\times}$. Then, using \cref{lem:scalarbasis}\cref{item5:scalarbasis} for the last equality, we see that
\[
i(\VV_e\times \{t\}) = t\VV_e=\VV_{te}.
\]
Thus $i$ is an open map, as required.

The final statement now follows by \cref{prop:Hausdorff}.
\end{proof}

\section{The isomorphism \texorpdfstring{$A \cong A_R(G;\Sigma)$}{of A with a twisted Steinberg algebra}} \label{sec:main iso}

Let $(A,B)$ be an algebraic quasi-Cartan pair and let $\Sigma$ be the groupoid of ultrafilters of $N(B)$. Let $C(\Sigma,R)$ denote the $R$-module of continuous (or equivalently, locally constant) functions from $\Sigma$ to $R$ with pointwise operations. In this section we build a map from $A$ to $C(\Sigma,R)$ using both the faithful conditional expectation $P\colon A \to B$ that is implemented by idempotents and the isomorphism $\phi\colon B \to A_R(\Sigmao)$ from \cref{eqn:phi description} that satisfies $\phi(e) = 1_{\VV_e}$ for all $e \in I(B)$. We prove that this map is in fact an isomorphism of $A$ onto the twisted Steinberg algebra $A_R(G;\Sigma)$.

\begin{prop}
\label{prop:ahat} Suppose that $(A,B)$ is an algebraic quasi-Cartan pair, and let $G$ and $\Sigma$ be the groupoids constructed in \cref{sec:build twist}. Let $\phi\colon B \to A_R(\Sigmao)$ be the isomorphism from~\cref{eqn:phi description} that satisfies $\phi(e) = 1_{\VV_e}$ for all $e \in I(B)$. For each $a \in A$, there is a function $\widehat{a}\colon \Sigma \to R$ such that for any ultrafilter $U \in \Sigma$ and any $n \in U$,
\[
\widehat{a}(U) = \phi(P(n^\dagger a))(\s(U)).
\]
Furthermore,
\begin{enumerate}[label=(\alph*)]
\item \label{item1:ahat} $\widehat{a}$ is continuous;
\item \label{item2:ahat} the map $a \mapsto \widehat{a}$ from $A$ to $C(\Sigma,R)$ is $R$-linear;
\item \label{item3:ahat} the map $a \mapsto \widehat{a}$ from $A$ to $C(\Sigma,R)$ is injective;
\item \label{item4:ahat} $\widehat{a}(t U) = t^{-1} \widehat{a}(U)$ for every $t \in R^\times$ and $U \in \Sigma$;
\item \label{item5:ahat} for $b \in B$, we have $\widehat{b}\restr{\Sigmao} = \phi(b)$, and $\supp(\widehat{b}) \subseteq i(\Go \times R^\times)$.
\end{enumerate}
\end{prop}

\begin{proof}
To see that there exists a (well-defined) function $\widehat{a}$ satisfying the given formula, we must show that $\phi(P(n^\dagger a))(\s(U))$ is independent of the choice of $n \in U$. So fix $a \in A$, $U \in \Sigma$, and $n, m \in U$. We must show that
\[
\phi(P(n^\dagger a))(\s(U)) - \phi(P(m^\dagger a))(\s(U)) = 0.
\]
Observe that if $m,n \in U$, then there exists $k \in U$ such that $k \le m,n$, whence $k^\dagger\le m^\dagger,n^\dagger$. So
\[
k^\dagger km^\dagger = k^\dagger = k^\dagger kn^\dagger,
\]
and moreover,
\[
\phi(k^\dagger k)(\s(U)) =
1_{\VV_{k^\dagger k}}(U^{-1}U) = 1.
\]
Now, using that $\phi$ is multiplicative and that $P$ is a conditional expectation, we compute
\begin{align*}
\phi(P(n^\dagger a))(\s(U)) - \phi(P(m^\dagger a))(\s(U)) &= \phi(k^\dagger k)(\s(U)) \, \phi(P(n^\dagger a - m^\dagger a))(\s(U)) \\
&= \phi(P(k^\dagger kn^\dagger a - k^\dagger km^\dagger a))(\s(U)) \\
&= \phi(P(k^\dagger a - k^\dagger a))(\s(U)) \\
&= 0.
\end{align*}

For part~\cref{item1:ahat}, fix $a \in A$. We show that $\widehat{a}$ is locally constant. Let $U\in \Sigma$ and choose $n\in U$. The function $\phi(P(n^\dagger a)) \in A_R(\Sigmao)$ is locally constant, and so there is an idempotent $e \in \s(U)$ such that for each ultrafilter $W \in \Sigmao$ containing $e$, we have $\phi(P(n^\dagger a))(W)=\Phi(P(n^\dagger a))(\s(U))$. Let $m \in U$ with $m^\dagger m\le e$. Then $m$ and $n$ have a common lower bound $k \in U$, and hence $k^\dagger k\le m^\dagger m \le e$. Thus $k=nk^\dagger k \le ne$. So $ne \in U$, and hence $\VV_{ne}$ is an open neighbourhood of $U$. If $V \in \VV_{ne}$, then $n \in V$, because $ne \le n$ (by \cref{lem:properties of le}\cref{item4:properties of le}), and $e \in \s(V)$ because $(ne)^\dagger ne = en^\dagger ne \le e$. Therefore,
\[
\widehat{a}(V) = \phi(P(n^\dagger a))(\s(V))=\phi(P(n^\dagger a))(\s(U)) = \widehat{a}(U),
\]
by the choice of $e$. Thus $\widehat{a}$ is locally constant, and hence is continuous.

Part~\cref{item2:ahat} follows from the $R$-linearity of $P$ and $\phi$.

For part~\cref{item3:ahat}, suppose that $\widehat{a}=0$. Then $\widehat{a}(U) = 0$ for all $U \in \Sigma$. We claim that $\phi(P(na)) = 0$ for all $n \in N(B)$. Fix $n \in N(B)$ and $U \in \Sigmao$. We show that $\phi(P(na))(U) = 0$ by considering two cases. First, suppose that $nn^\dagger \in U$. Then
\[
U \in \VV_{nn^\dagger} = \VV_n\VV_{n^\dagger} = \s(\VV_{n^\dagger}),
\]
and so we can find an ultrafilter $W \in \VV_{n^\dagger}$ such that $\s(W)=U$. We then have
\[
\phi(P(na))(U) = \phi(P(na))(\s(W)) = \widehat{a}(W) = 0.
\]
For the second case, suppose that $nn^\dagger \notin U$. Then $U \notin \VV_{nn^\dagger}$, and since $P$ is a conditional expectation, we have
\[
\phi(P(na))(U) = \phi(P(nn^\dagger na))(U) = \phi(nn^\dagger)(U) \, \phi(P(na))(U) = 1_{\VV_{nn^\dagger}}(U) \, \phi(P(na))(U) = 0,
\]
as claimed. Since $\phi$ is injective, we deduce that $P(na)=0$ for all $n\in N(B)$. Now, since $P$ is faithful (property~\cref{item:ACP.FCE} of \cref{def:ACP}), we deduce that $a = 0$.

For part~\cref{item4:ahat}, fix $U \in \Sigma$ and $t \in R^\times$. Then $\s(U) = \s(tU)$ by \cref{lem:scalarbasis}\cref{item4:scalarbasis}. Let $n\in U$. Then $tn\in tU$ and $(tn)^\dagger =t^{-1}n^\dagger$ by \cref{lem:scalarn}\cref{item1:scalarn}. Therefore,
\[
\widehat{a}(tU) = \phi(P((tn)^\dagger a))(\s(tU)) = \phi(P(t^{-1}n^\dagger a))(\s(U)) = t^{-1}\phi(P(n^\dagger a))(\s(U)) = t^{-1}\widehat a(U),
\]
as required.

For part~\cref{item5:ahat}, fix $b \in B$ and $U \in \Sigmao$ with $e \in U \cap I(B)$, and note that $e=e^\dagger$. Then
\[
\widehat{b}(U) = \phi(P(eb))(\s(U)) = \phi(eb)(U)=\phi(e)(U)\phi(b)(U) = 1_{\VV_e}(U) \, \phi(b)(U) = \phi(b)(U).
\]
Finally, we show that the support of $\widehat{b}$ is contained in $i(\Go \times R^\times)$. Since $B$ is spanned by $I(B)$ and the map $a \mapsto \widehat{a}$ is $R$-linear by part~\cref{item2:ahat}, it suffices to consider the case where $b = e \in I(B)$. Fix $U \in \Sigma$ such that $\widehat{e}(U) \ne 0$. We claim that $U \in R^\times \cdot \Sigmao$. To see this, fix $n \in U$. Then
\[
\phi(P(n^\dagger e))(\s(U)) = \widehat{e}(U) \ne 0.
\]
Since
$(A,B)$ is an algebraic quasi-Cartan pair, there exists $f \in I(B)$ such that $fn^\dagger = n^\dagger f = P(n^\dagger)$, and so
\[
fen^\dagger =efn^\dagger = eP(n^\dagger) = P(n^\dagger) e=n^\dagger fe.
\]
Also, $P(n^\dagger e) = P(n^\dagger)e=n^\dagger fe$. We then have
\begin{align*}
0 \ne \phi(P(n^\dagger e))(\s(U)) &= \phi(n^\dagger fe)(\s(U)) \\
&= \phi(n^\dagger fefe)(\s(U)) \\
&= \phi(n^\dagger fe)(\s(U)) \, \phi(fe)(\s(U)) \\
&= \phi(n^\dagger fe)(\s(U)) \, 1_{\VV_{fe}}(\s(U)).
\end{align*}
Thus $fe \in \s(U)$, and so $nfe \in U\s(U) = U$. Therefore, $fen^\dagger\in U^{-1}$. Since $P(n^\dagger e) = n^\dagger fe \in N(B) \cap B$, \cref{rem:N(B) cap B characterisation} implies that $fen^\dagger=n^\dagger fe = P(n^\dagger e)$ can be expressed as a finite sum $\sum_{i=1}^k t_i e_i$, where $e_1,\dotsc,e_k \in I(B)$ are mutually orthogonal idempotents and $t_1,\dotsc,t_k \in R^\times$. Notice that for each $i \in \{1, \dotsc, k\}$, we have $t_ie_i = fen^\dagger e_i$, and so
\[
\sum_{i=1}^k fen^\dagger e_i = fen^\dagger \in U^{-1}.
\]
Since $U^{-1}$ is an ultrafilter, it follows from \cref{lem:ultrafilter partition} that there exists (a unique) $i \in \{1, \dotsc, k\}$ such that $t_i e_i = fen^\dagger e_i \in U^{-1}$. Hence $t_i^{-1}e_i=(t_ie_i)^\dagger \in U$ by \cref{lem:scalarn}\cref{item1:scalarn}, which forces $V \coloneqq t_i U \in \Sigmao$. Therefore, $U = t_i^{-1} V \in q^{-1}(\Go) = i(\Go \times R^\times)$.
\end{proof}

We now prove that if $(A,B)$ is an algebraic quasi-Cartan pair, and $G$ and $\Sigma$ are as constructed in \cref{sec:build twist}, then there is an isomorphism of $A$ onto $A_R(G;\Sigma)$ that carries $B$ to the canonical subalgebra isomorphic to $A_R(\Sigmao)$. We need the following technical lemma.

In what follows, we write $\widehat{P}(a) \coloneqq \widehat{P(a)}$ for each $a\in A$, where $a \mapsto \widehat{a} \colon A \to C(\Sigma,R)$ is the map from \cref{prop:ahat}.

\begin{lemma} \label{lem:cond.expectation.zero}
Let $R$ be an indecomposable commutative ring. Suppose that $(A,B)$ is an algebraic quasi-Cartan pair with faithful conditional expectation $P\colon A \to B$ implemented by idempotents. Let $G$ and $\Sigma$ be the groupoids constructed in \cref{sec:build twist}. Let $m,n\in N(B)$ and $U \in \Sigmao$. If $\widehat{P}(m n)(U) \ne 0$, then there exists $f \in U\cap I(B)$ such that $f m n = m n f = \widehat{P}(m n)(U) f \in B$.
\end{lemma}

\begin{proof}
Suppose that $\widehat{P}(mn)(U) \ne 0$. Then $P(mn) \ne 0$ by \cref{prop:ahat}\cref{item2:ahat}, and since $P$ is implemented by idempotents, there exists $e_0 \in I(B)$ such that $e_0 mn = mn e_0 = P(mn)$. Since $P(mn) \in B$, \cref{prop:ahat}\cref{item5:ahat} implies that
\begin{align*}
0 &\ne \widehat{P}(mn)(U) = \phi(P(mn))(U) = \phi(mm^\dagger mne_0^2)(U) \\
&= \phi(mm^\dagger)(U) \, \phi(mne_0)(U) \, \phi(e_0)(U) = 1_{\VV_{mm^\dagger}}(U) \, \phi(mne_0)(U) \, 1_{\VV_{e_0}}(U),
\end{align*}
and so $mm^\dagger,e_0 \in U$. Hence $e \coloneqq e_0 mm^\dagger \in U \cap I(B)$ satisfies
\begin{equation} \label{eqn:emn=mne}
e mn = e e_0 mn = e P(mn) = P(mn)e = mn e_0 e = mn e.
\end{equation}
Since $e P(mn) \in B$, there is a finite set $F \subseteq I(B)$ of mutually orthogonal idempotents and nonzero coefficients $r_f \in R$ such that $e P(mn) = \sum_{f \in F} r_f f$; without loss of generality we may assume that $f \le e$ for each $f \in F$. Thus, using \cref{prop:ahat}\cref{item5:ahat} and \cref{eqn:phi description}, we see that
\[
0 \ne \widehat{P}(mn)(U) = \widehat{e}(U) \widehat{P}(mn)(U) = \sum_{f \in
F} r_f \widehat{f}(U) = \sum_{f \in F} r_f 1_{\VV_f}(U).
\]
Hence we deduce that $U \in \bigcup_{f\in F}\VV_f$. But the mutual orthogonality of the idempotents in $F$ implies that the $\VV_f$ are pairwise disjoint, and so there is a unique $f \in F \cap U$. Therefore, $r_f = \widehat{P}(mn)(U)$. Since $f \le e$, a similar argument to the one used in \cref{eqn:emn=mne} shows that $f$ commutes with $mn$, and $fmn = femn = feP(mn) = r_f f = \widehat{P}(mn)(U) f \in B$.
\end{proof}

\begin{lemma} \label{lem:nhat}
Let $(A,B)$ be an algebraic quasi-Cartan pair, and let $G$ and $\Sigma$ be the groupoids constructed in \cref{sec:build twist}. Then for $n \in N(B)$ and $U \in \Sigma$, we have
\[
\widehat{n}(U) = \begin{cases}
t^{-1} & \text{if } U \in \VV_{t n} \text{ for some } t \in R^\times \\
0 & \text{otherwise}.
\end{cases}
\]
In particular, $\widehat{n}$ is equal to the function $\tilde{1}_{\VV_n} \in A_R(G;\Sigma)$ of \cref{lem:unique extensions}.
\end{lemma}

\begin{proof}
For the first case, suppose that $U \in \VV_{t n}$ for some $t \in R^\times$, and let $\phi\colon B \to A_R(\Sigmao)$ be the isomorphism of \cref{eqn:phi description} that satisfies $\phi(e) = 1_{\VV_e}$ for all $e \in I(B)$. Then $(tn)^\dagger = t^{-1}n^\dagger$ by \cref{lem:scalarn}\cref{item1:scalarn}, and so $n^\dagger n= (tn)^\dagger tn\in \s(U)$. Therefore, we have
\[
\widehat{n}(U) = \phi(P(t^{-1}n^\dagger n))(\s(U)) = \phi(t^{-1}n^\dagger n)(\s(U)) = t^{-1} \phi(n^\dagger n)(\s(U)) = t^{-1}.
\]

For the second case, we prove the contrapositive. Suppose that $\widehat{n}(U) \ne 0$, and fix $m \in U$. Then by \cref{prop:ahat}\cref{item5:ahat} and the definition of $\widehat{n}$, we have
\[
\widehat{P}(m^\dagger n)(\s(U)) = \phi(P(m^\dagger n))(\s(U)) = \widehat{n}(U) \ne 0.
\]
Hence \cref{lem:cond.expectation.zero} shows that there exists an idempotent $f \in \s(U)$ such that $f m^\dagger n = m^\dagger nf = \widehat{n}(U) f$. Since $\widehat{n}(U) f = fm^\dagger n \in N(B) \cap B$, \cref{lem:diagonal normalisers} implies that $t \coloneqq \widehat{n}(U)$ belongs to $R^\times$. Since $f \in \s(U)$ and $m\in U$, we have $mf\in U$. Note that $t(mf) = m(tf) = mm^\dagger nf \le n$ by \cref{lem:properties of le}\cref{item4:properties of le}. Thus $mf \le t^{-1}n$ by \cref{lem:scalarn}\cref{item2:scalarn}, and so $t^{-1}n \in U$. That is, $U\in \VV_{t^{-1}n}$, as required.
\end{proof}

For the surjectivity in the main theorem (\cref{thm:main}), we need to know that each element of $A_R(G;\Sigma)$ can be written as a finite sum of elements of the form $t \widehat{n}$ (where $t \in R$ and $n \in N(B)$); we do this in the following two results, the first of which is standard.

\begin{lemma} \label{lem:disjointify}
Let $H$ be an ample groupoid and let $D$ be a compact open bisection of $H$. Let $\BB$ be an inverse semigroup of compact open bisections that form a basis for $H$ and whose idempotents are closed under relative complement and disjoint union. Then $D$ can be expressed as a finite disjoint union of elements of $\BB$.
\end{lemma}

\begin{proof}
Since $\BB$ is a basis and $D$ is compact, we can certainly write $D=\bigcup_{i=1}^N B_i$ with $B_i\in \BB$. Put $B'_1=B_1$. Assume inductively that we have found $B_1',\dotsc, B_j'\in \BB$ pairwise disjoint with $\bigcup_{i=1}^j B_i=\bigsqcup_{i=1}^j B_i'$ for $1\le j<n$. Then put $B_{j+1}' = B_{j+1}\big(\s(B_{j+1}) \setminus \cup_{i=1}^j \s(B_i')\big)$. By assumption on $\BB$, we have that $B_{j+1}'\in \BB$ as $B_i ',B_k'\subseteq D$ disjoint implies that $\s(B_i')$ and $\s(B_k')$ are disjoint because $D$ is a bisection. Note that $B_{j+1}'$ is disjoint from $B_1',\dotsc, B_j'$ by construction.

Since $B_{j+1}'\subseteq B_{j+1}$, by the inductive assumption it suffices to show that $B_{j+1}\subseteq \bigsqcup_{i=1}^{j+1} B_i'$. Fix $\gamma \in B_{j+1}$. If $\s(\gamma) \notin \s(B_1') \cup \dotsb \cup \s(B_j')$, then trivially $\gamma\in B_{j+1}'$. If $\s(\gamma) \in \s(B_i')$ for some $i \in \{1,\dotsc,j\}$, then since $B_i',B_{j+1}\subseteq D$ and $D$ is a bisection, we must have that $\gamma\in B_i'$. This completes the proof.
\end{proof}

\begin{prop} \label{prop:hat ns span}
Let $(A,B)$ be an algebraic quasi-Cartan pair, and let $G$ and $\Sigma$ be the groupoids constructed in \cref{sec:build twist}. For any $f \in A_R(G;\Sigma)$, there exist $n_1,\dotsc,n_M \in N(B)$ and $t_1,\dotsc,t_M \in R$ such that
\[
f = \sum_{j=1}^M t_j \widehat{n_j}.
\]
\end{prop}

\begin{proof}
By \cref{prop:f is a finite sum of lifts of indicators}, it suffices to prove the claim for $f = \tilde{1}_D$, where $D$ is a compact open bisection of $\Sigma$. The collection $\BB = \{\VV_n:n\in N(B)\}$ is an inverse semigroup of compact open bisections forming a basis for the topology on $\Sigma$. Moreover, if $e,f \in I(B)$, then $\VV_e\setminus \VV_f = \VV_{e-ef}$, and $\VV_e\cap \VV_f=\varnothing$ if and only if $ef=0$, in which case, $\VV_e \sqcup \VV_f = \VV_{e+f}$, by Stone duality applied to $I(B)$. Thus, by \cref{lem:disjointify}, we can express $D$ as a finite disjoint union of basic compact open sets, say, $D = \bigsqcup_{j=1}^M \VV_{n_j}$. Then $g\coloneqq\sum_{j=1}^M \widehat{n_j}$ is an element of $A_R(G;\Sigma)$ satisfying $\supp(g)\subseteq R^\times \cdot D$ and $g\restr{D} \equiv 1$, by \cref{lem:scalarbasis}\cref{item5:scalarbasis} and \cref{lem:nhat}. Therefore, \cref{lem:unique extensions}
gives $g = \tilde{1}_D$.
\end{proof}

\begin{theorem} \label{thm:main}
Suppose that $(A,B)$ is an algebraic quasi-Cartan pair. Let $G$ and $\Sigma$ be the groupoids constructed in \cref{sec:build twist}. Then the map $\varphi\colon a \mapsto \widehat{a}$ from $A$ to $C(\Sigma,R)$ defined in \cref{prop:ahat} is an isomorphism of $A$ onto $A_R(G;\Sigma)$ that takes $B$ to $A_R(\Go; q^{-1}(\Go))$, which is isomorphic to $A_R(\Sigmao)$, and hence also to $A_R(\Go)$.
\end{theorem}

\begin{proof}
We begin by showing that $\widehat{a} \in A_R(G;\Sigma)$ for each $a \in A$. If $n \in N(B)$, then $\widehat{n} \in A_R(G;\Sigma)$ by \cref{lem:nhat}. Since each $a \in A$ can be expressed as an $R$-linear combination of elements of $N(B)$ (by property~\cref{item:ACP.NB spans} of \cref{def:ACP}), it follows that $\widehat{a} \in A_R(G;\Sigma)$, because $\varphi$ is $R$-linear by \cref{prop:ahat}\cref{item2:ahat}.

\Cref{prop:ahat} implies that $\varphi$ is an injective $R$-linear map, and \cref{prop:hat ns span} implies that $\varphi$ is surjective.

To complete the proof that $\varphi$ is an isomorphism, we must show that $\widehat{aa'} = \widehat{a} * \widehat{a'}$ for all $a, a' \in A$. Since $A$ is the $R$-linear span of $N(B)$, it suffices to prove this for $a = n, a' = m \in N(B)$. But this follows from \cref{lem:nhat,cor:normalisermult}, since $\VV_n \VV_m = \VV_{nm}$.

Finally, we check the statement about $B$. If $e\in I(B)$, then $\varphi(e) = \widehat e = \tilde{1}_{\VV_e}$ by \cref{lem:nhat}. Since $I(B)$ spans $B$ and $\{\tilde{1}_{\VV_e} : e \in I(B)\}$ spans $A_R(\Go; q^{-1}(\Go))$ (by \cref{prop:Steinberg diagonal}), it follows that $\varphi$ carries $B$ isomorphically to $A_R(\Go; q^{-1}(\Go))$, completing the proof.
\end{proof}

\begin{cor}
Suppose that $(A,B)$ is an algebraic quasi-Cartan pair. Then the discrete $R$-twist $(\Sigma,i,q)$ of \cref{thm:the twist} satisfies the local bisection hypothesis.
\end{cor}

\begin{proof}
This follows immediately from \cref{thm:main,lem:LBH<->AQP}.
\end{proof}

\section{Algebraic information from the isotropy structure of \texorpdfstring{$G$}{G}} \label{sec:info from iso}

In this section we describe the properties of the groupoid $G$ (from the previous section) that identify the algebraic diagonal pairs and the algebraic Cartan pairs amongst all algebraic quasi-Cartan pairs.

\begin{prop}
\label{prop:effective} Suppose that $(A,B)$ is an algebraic quasi-Cartan pair. Then
\begin{enumerate}[label=(\alph*)]
\item \label{item:masa<->effective} $(A,B)$ is an algebraic Cartan pair if and only if $G$ is effective, and
\item \label{item:diagonal<->principal} $(A,B)$ is an algebraic diagonal pair if and only if $G$ is principal.
\end{enumerate}
\end{prop}

The proof will follow easily once we establish the following two lemmas.

\begin{lemma}
\label{lem:inB}
Let $(A,B)$ be an algebraic Cartan pair, and let $n \in N(B)$. Suppose that $\VV_n \subseteq \Iso(\Sigma)$. Then $n \in B$.
\end{lemma}

\begin{proof}
Let $P\colon A \to B$ be the faithful conditional expectation that is implemented by idempotents, and let $\phi\colon B \to A_R(\Sigmao)$ be the isomorphism from~\cref{eqn:phi description} that satisfies $\phi(e) = 1_{\VV_e}$ for all $e \in I(B)$. We claim that $nb=bn$ for every $b \in B$. Since $B$ is spanned by $I(B)$ (property~\cref{item:ACP.IB spans} of \cref{def:ACP}), it suffices to fix $e \in I(B)$ and show that $en = ne$. We show that $\widehat{ne} = \widehat{en}$; the claim then follows by the injectivity of $a \mapsto \widehat{a}$ (\cref{prop:ahat}\cref{item3:ahat}). Fix $U \in \Sigma$. If $\widehat{n}(U) = 0$, then by multiplicativity of $a \mapsto \widehat{a}$ (\cref{thm:main}), we have
\[
\widehat{en}(U) = \widehat{e}(U) \, \widehat{n}(U) = 0 = \widehat{n}(U) \, \widehat{e}(U) = \widehat{ne}(U).
\]
So it suffices to consider $U \in \supp(\widehat{n}) = R^\times\cdot \VV_n$ by \cref{lem:nhat}; so, in particular, $\r(U) = \s(U)$ since $\VV_n \in \Iso(\Sigma)$. Now fix $m \in U$. Then
\begin{align*}
\widehat{ne}(U) &= \phi(P(m^\dagger ne))(\s(U)) \\
&= \phi(P(m^\dagger n))(\s(U)) \, \phi(e)(\s(U)) \\
&= \begin{cases} \phi(P(m^\dagger n))(\s(U)) & \text{ if } e \in \s(U) \\
0 & \text{otherwise}
\end{cases} \\
&= \begin{cases} \widehat{n}(U) & \text{ if } e \in \s(U) \\
0 & \text{otherwise}.
\end{cases}
\end{align*}

In the first case, we have $e \in \s(U)=\r(U)$. Since $\r(U)U = U$, we have $em \in U$, and hence
\[
\widehat{en}(U) = \phi(P(m^\dagger en))(\s(U)) = \phi(P((em)^\dagger n))(\s(U)) = \widehat{n}(U) = \widehat{ne}(U).
\]

For the second case, suppose that $e \notin \s(U)=\r(U)$. Let $e_1 \coloneqq mm^\dagger - emm^\dagger$ and let $e_2 \coloneqq emm^\dagger$. These are orthogonal idempotents in $I(B)$, and $m=e_1m+e_2m$. Therefore, either $e_1m=m-em$ or $e_2m=em$ belongs to $U$, by \cref{lem:ultrafilter partition}. But if $em \in U$, then since $em(em)^\dagger = em m^\dagger e \le e$, we have $e \in \r(U)$, which is a contradiction. So we must have $m - em \in U$. Let $k \coloneqq (m-em)^\dagger$. Then $k = m^\dagger - m^\dagger e$ by \cref{lem:properties of le}\cref{item5:properties of le}, and so $ke=0$. Therefore, $\widehat{en}(U) = \phi(P(ken))(\s(U))=0$, as required. This completes the proof of the claim.

By the claim, $n$ commutes with all of $B$. Since $B$ is a maximal commutative subalgebra, this forces $n \in B$.
\end{proof}

\begin{lemma} \label{lem:Bunits}
Suppose that $(A,B)$ is an algebraic quasi-Cartan pair. If $n \in N(B) \cap B$, then $\VV_n \subseteq q^{-1}(\Go)$.
\end{lemma}

\begin{proof}
The result is trivial if $n = 0$, and so we assume that $n \ne 0$. By \cref{rem:N(B) cap B characterisation}, there exist mutually orthogonal idempotents $e_1, \dotsc, e_k \in I(B)$ and coefficients $t_1, \dotsc, t_k \in R^\times$ such that $n = \sum_{i=1}^k t_i e_i$. Note that $n e_i = t e_i$ for each $i \in \{1,\dotsc,k\}$, and so $n = \sum_{i=1}^k n e_i$. Fix $V \in \VV_n$. Then \cref{lem:ultrafilter partition} implies that $n e_i \in V$ for some (unique) $i \in \{1,\dotsc,k\}$. Thus $t_i e_i = n e_i \in V$, and so $e_i \in t_i^{-1}V$. Therefore, $t_i^{-1}V \in \Sigmao$, and so $V = t_i(t_i^{-1}V) \in R^\times\cdot \Sigmao=q^{-1}(\Go)$, as required.
\end{proof}

\begin{proof}[Proof of \cref{prop:effective}]
For part~\cref{item:masa<->effective}, first suppose that $(A,B)$ is an algebraic Cartan pair. Fix an open set $O$ contained in the interior of the isotropy of $G$. Without loss of generality, we may assume that $O=q(\VV_n)$ for some $n \in N(B)$. Then $\VV_n$ is contained in the isotropy of $\Sigma$. \Cref{lem:inB} implies that $n \in B$, and \cref{lem:Bunits} gives
$q(\VV_n) \subseteq \Go$.

Now suppose that $G$ is effective. Then the algebraic quasi-Cartan pair $(A,B)$ is isomorphic to $(A_R(G;\Sigma), A_R(\Go, q^{-1}(\Go)))$ (by \cref{thm:main}), which is an algebraic Cartan pair by \cref{prop:effectiveACPprincipalADP}.

For part~\cref{item:diagonal<->principal}, first suppose that $(A,B)$ is an algebraic diagonal pair. Fix $U \in \Sigma$ such that $\r(U) = \s(U)$. We must show that $q(U) \in \Go$. For this, fix $n \in U$. Since $A$ is spanned by the free normalisers of $B$, we can write $n = n_0 + \sum_{i=1}^k n_i$, where $n_0 \in B$ and each $n_i$ is a normaliser satisfying $n_i^2 = 0$. Since $\s(U)$ is an ultrafilter, we can find an idempotent $e \in \s(U)=\r(U)$ such that for every $i \in \{1, \dotsc, k\}$, we have $e \le n_i^\dagger n_i$ whenever $n_i^\dagger n_i \in \s(U)$, and $e n_i^\dagger n_i = 0$ whenever $n_i^\dagger n_i \notin \s(U)$. Since $e \in \s(U) = \r(U)$, we have $\sum_{i=0}^k en_i e = ene \in U$. If $i \in \{1, \dotsc, k\}$ satisfies $n_i^\dagger n_i \notin \s(U)$, then $e n_i e = e n_i n_i^\dagger n_i e = 0$; and if $i\in \{1, \dotsc, k\}$ satisfies $n_i^\dagger n_i \in \s(U)$, then $e \le n_i^\dagger n_i$, and so $e n_i e = e(n_i^\dagger n_i) (n_i e) = 0$ as well. So we obtain $e n e = e n_0 e \in U \cap B \cap N(B)$. Hence $U \in \VV_{ene} \subseteq q^{-1}(\Go)$ by \cref{lem:Bunits}.

Now suppose that $G$ is principal. Then $(A,B) \cong (A_R(G;\Sigma), A_R(\Go, q^{-1}(\Go)))$ is the algebra of a twist over a principal groupoid (by \cref{thm:main}), and hence is an algebraic diagonal pair by \cref{prop:effectiveACPprincipalADP}.
\end{proof}

\section{Recovering a twist from its quasi-Cartan pair} \label{sec:recover twist}

In this section we show that if $(\Sigma,i,q)$ is a discrete $R$-twist over an ample Hausdorff groupoid $G$, then we can construct a natural embedding of $\Sigma$ into the groupoid of ultrafilters $\Sigma'$ obtained from \cref{thm:the twist} applied to the pair $(A_R(G;\Sigma), A_R(\Go; q^{-1}(\Go)))$. We then show that this map is an isomorphism if and only if $(\Sigma,i,q)$ satisfies the local bisection hypothesis. We start with some technical results. Throughout, we will identify $\Sigma'^{(0)}$ with ultrafilters of $I(B)$, where $B \coloneqq A_R(\Go; q^{-1}(\Go))$, via the homeomorphism described in \cref{rem:ultrafilters on I(B)}.

\begin{lemma} \label{lem:germ-ultrafilter}
Let $(\Sigma,i,q)$ be a discrete $R$-twist over an ample Hausdorff groupoid $G$, let $A \coloneqq A_R(G;\Sigma)$, and let $B \coloneqq A_R(\Go; q^{-1}(\Go))$. If $U$ is an ultrafilter of $I(B)$ and $n\in N(B)$ satisfies $n^\dagger n\in U$, then the upclosure $V \coloneqq (nU)^{\uparrow}$ of $nU \subseteq N(B) {\setminus} \{0\}$ is an ultrafilter of $N(B)$ containing $n$.
\end{lemma}

\begin{proof}
Regarding $U$ as an ultrafilter in $\Sigma'^{(0)}$, we have $U \in \VV_{n^\dagger n} = \s(\VV_n)$, and so there exists an ultrafilter $V \in \VV_n$ with $\s(V) = U$. Then \cite[Lemma~3.1(a)]{ACaHJL2021} implies that $V = (nU)^{\uparrow}$. Since $n^\dagger n \in U$, we have $n = n n^\dagger n \in nU \subseteq V \subseteq N(B) {\setminus} \{0\}$.
\end{proof}

\begin{lemma} \label{lem:norm.vs.indicator}
Let $(\Sigma,i,q)$ be a discrete $R$-twist over an ample Hausdorff groupoid $G$, let $A \coloneqq A_R(G;\Sigma)$, and let $B \coloneqq A_R(\Go; q^{-1}(\Go))$. Let $X \subseteq \Sigma$ be a compact open bisection, and let $n \in N(B)$. Then
\begin{enumerate}[label=(\alph*)]
\item \label{lem:nvi.i.leq.n} $\tilde{1}_X\le n$ if and only if $n\restr{X} \equiv 1$ and $q(X) = \s^{-1}(\s(q(X)))\cap \supp_G(n)$;
\item \label{item:nvi.n.leq.i} if $n\le\tilde{1}_X$, then $n=\tilde{1}_Y$ for some compact open bisection $Y\subseteq X$; and
\item \label{item:nvi.Y.sub.X} given a compact open bisection $Y \subseteq \Sigma$, we have $\tilde{1}_Y\le\tilde{1}_X$ if and only if $Y\subseteq X$.
\end{enumerate}
\end{lemma}

\begin{proof}
For part~\cref{lem:nvi.i.leq.n}, suppose first that $\tilde{1}_X \le n$; that is, that $\tilde{1}_X = n\tilde{1}_X^\dagger \tilde{1}_X = n\tilde{1}_{\s(X)}$. \Cref{prop:Steinberg diagonal} implies that for all $\sigma \in X$,
\[
1 = \tilde{1}_X(\sigma) = (n\tilde{1}_{\s(X)})(\sigma) = n(\sigma) \, 1_{\s(X)}(\s(\sigma)) = n(\sigma).
\]
The containment $q(X) \subseteq \s^{-1}(\s(q(X))) \cap \supp_G(n)$ follows from the above equality. For the reverse containment, suppose that $\sigma \in \Sigma$ satisfies $q(\sigma)\in \s^{-1}(\s(q(X)))\cap\supp_G(n)$. Then $\s(q(\sigma))\in \s(q(X))$, and so there exists $\gamma\in X$ such that $q(\s(\gamma))=\s(q(\gamma))=\s(q(\sigma))=q(\s(\sigma))$. This implies that there exists a unique $t \in R^\times$ such that $\s(\sigma) = t \cdot \s(\gamma)$. By the equality $\tilde{1}_X=n\tilde{1}_{\s(X)}$, we obtain
\[
\tilde{1}_X(\sigma) = (n\tilde{1}_{\s(X)})(\sigma) = n(\sigma) \, \tilde{1}_{\s(X)}(\s(\sigma)) = t^{-1}n(\sigma)\ne 0.
\]
This implies that $\sigma = u \cdot \sigma'$ for some $\sigma'\in X$ and $u\in R^\times$. Hence $q(\sigma)=q(\sigma')\in q(X)$, and the containment $\s^{-1}(\s(q(X)))\cap\supp_G(n)\subseteq q(X)$ follows.

For the converse, notice that for $\gamma \in \Sigma$, applying \cref{prop:Steinberg diagonal} at the first step gives
\begin{align*}
(n\tilde{1}_{\s(X)})(\gamma) \ne 0
&\iff n(\gamma) \, \tilde{1}_{\s(X)}(\s(\gamma)) \ne 0 \\
&\iff \gamma \in \supp (n) \text{ and } \s(q(\gamma)) = q(\s(\gamma)) \in q(\s(X)) = \s(q(X)) \\
&\iff q(\gamma)\in \s^{-1}(\s(q(X)))\cap \supp_G(n) = q(X) \\
&\iff \gamma\in R^\times\cdot X.
\end{align*}
Since $(n\tilde{1}_{\s(X)})(\sigma) = 1$ for all $\sigma \in X$, we deduce that $n\tilde{1}_{\s(X)}=\tilde{1}_X$ by \cref{lem:unique extensions}, and hence $\tilde{1}_X\le n$.

For part~\cref{item:nvi.n.leq.i}, let $Y \coloneqq X \cap \s^{-1}(\s(\supp(n)))$. Note that $\s(\supp(n))=\s(\supp_G(n))$ (under the usual identification of $\Go$ and $\Sigmao$) is compact and open, and hence is clopen since $\Go$ is Hausdorff. Thus $Y$ is a clopen subset of the compact open bisection $X$, and so $Y$ is a compact open bisection. Since $n \le \tilde{1}_X$ by hypothesis, \cref{rmk:Steinberg4.5} implies that
\[
n = \tilde{1}_X n^\dagger n = \tilde{1}_X \tilde{1}_{\s(\supp(n))} = \tilde{1}_Y.
\]

For part~\cref{item:nvi.Y.sub.X}, let $Y$ be a compact open bisection of $\Sigma$. If $\tilde{1}_Y \le \tilde{1}_X$, then since $\tilde{1}_Y \in N(B)$, part~\cref{item:nvi.n.leq.i} implies that $\tilde{1}_Y = \tilde{1}_Z$ for some compact open bisection $Z\subseteq X$. This forces $Y=Z$, so that $Y\subseteq X$. On the other hand, if $Y\subseteq X$, then $X\s(Y)=Y$, and so $\tilde{1}_Y = \tilde {1}_X\tilde{1}_{\s(Y)} \le \tilde{1}_X$ by \cref{lem:properties of le}\cref{item4:properties of le} and \cref{cor:normalisermult} (since $\tilde{1}_{\s(Y)} \in I(B)$ and $\tilde{1}_X \in N(B)$).
\end{proof}

\begin{lemma} \label{lem:ind.choice.bisection}
Let $(\Sigma,i,q)$ be a discrete $R$-twist over an ample Hausdorff groupoid $G$, let $A \coloneqq A_R(G;\Sigma)$, and let $B \coloneqq A_R(\Go; q^{-1}(\Go))$. Fix $\sigma\in\Sigma$. The set $U_{\s(\sigma)} \coloneqq \{e \in I(B) : e(\s(\sigma)) = 1\}$ is an ultrafilter of $I(B)$.
\begin{enumerate}[label=(\alph*)]
\item \label{item1:ind.choice.bisection} If $X$ and $Y$ are compact open bisections such that $\sigma\in X\cap Y$, then $(\tilde{1}_X U_{\s(\sigma)})^{\uparrow}=(\tilde{1}_Y U_{\s(\sigma)})^{\uparrow}$.
\item \label{item2:ind.choice.bisection} If $X$ and $Y$ are compact open bisections such that $\sigma\in X$ and $\tilde{1}_Y \in (\tilde{1}_X U_{\s(\sigma)})^{\uparrow}$, then $\sigma\in Y$.
\item \label{item3:ind.choice.bisection} If $X$ is a compact open bisection such that $\sigma\in X$ and $n\in (\tilde{1}_X U_{\s(\sigma)})^{\uparrow}$, then there exists a compact open set $W \subseteq \s(X)$ such that $\sigma \in XW$ and $\tilde{1}_{XW} \le n$.
\end{enumerate}
\end{lemma}

\begin{proof}
A routine argument shows that $U_{\s(\sigma)}$ is an ultrafilter of $I(B)$.

For part~\cref{item1:ind.choice.bisection}, it suffices by symmetry to show that $(\tilde{1}_X U_{\s(\sigma)})^{\uparrow} \subseteq (\tilde{1}_Y U_{\s(\sigma)})^{\uparrow}$. Since the right-hand side is closed upwards, it suffices to show that it contains $\tilde{1}_X U_{\s(\sigma)}$. But the right-hand side is also closed under right multiplication by $U_{\s(\sigma)}$. Thus it suffices to show that $\tilde{1}_X\in (\tilde{1}_Y U_{\s(\sigma)})^{\uparrow}$. Since $X \cap Y$ is open, there is a compact open set $W\subseteq X\cap Y$ with $\sigma\in W$, and clearly $W$ is a compact open bisection. So $\tilde{1}_W\le \tilde{1}_X,\tilde{1}_Y$ by \cref{lem:norm.vs.indicator}\cref{item:nvi.Y.sub.X}, and $\tilde{1}_{\s(W)} \in U_{\s(\sigma)}$ since $\sigma \in W$. Note that $\tilde{1}_W = \tilde{1}_Y\tilde{1}_{\s(W)}$ by \cref{cor:normalisermult}. Therefore, $\tilde{1}_Y \tilde{1}_{\s(W)} = \tilde{1}_W \le \tilde{1}_X$, and so $\tilde{1}_X\in (\tilde{1}_Y U_{\s(\sigma)})^{\uparrow}$, as required.

For part~\cref{item2:ind.choice.bisection}, note that since $\tilde{1}_Y \in (\tilde{1}_X U_{\s(\sigma)})^{\uparrow}$, there exists a compact open set $W \subseteq \Sigmao$ such that $\tilde{1}_W \in U_{\s(\sigma)}$ and $\tilde{1}_{XW} = \tilde{1}_X \tilde{1}_W \le \tilde{1}_Y$ by \cref{cor:normalisermult}. Thus $\sigma=\sigma\s(\sigma)\in XW$, and so $\tilde{1}_Y(\sigma)=1$ by \cref{lem:norm.vs.indicator}\cref{lem:nvi.i.leq.n}. Hence $\sigma \in Y$.

For part~\cref{item3:ind.choice.bisection}, note that since $n \in (\tilde{1}_X U_{\s(\sigma)})^{\uparrow}$, there exists a compact open set $W \subseteq \Sigmao$ such that $\tilde{1}_W \in U_{\s(\sigma)}$ and $\tilde{1}_X \tilde{1}_W \le n$. By replacing $W$ with $W \cap \s(X)$, we may assume that $W \subseteq \s(X)$. Hence $\tilde{1}_{XW} = \tilde{1}_X \tilde{1}_W \le n$. Since $W \in U_{\s(\sigma)}$, we have $\s(\sigma) \in W$, and so $\sigma = \sigma \s(\sigma) \in XW$.
\end{proof}

We are now in a position to show that for any discrete $R$-twist $(\Sigma,i,q)$ over $G$, there is a continuous open groupoid homomorphism $\Phi$ of $\Sigma$ into the twist obtained from the pair $(A_R(G;\Sigma), A_R(\Go; q^{-1}(\Go)))$, and that this induces a morphism of twists. We first establish the existence of $\Phi$.

\begin{prop} \label{prop:embedding}
Let $(\Sigma,i,q)$ be a discrete $R$-twist over an ample Hausdorff groupoid $G$, let $A \coloneqq A_R(G;\Sigma)$, and let $B \coloneqq A_R(\Go; q^{-1}(\Go))$. Let $\Sigma'$ be the groupoid of ultrafilters of $N(B)$ defined in \cref{sec:Sigma def}, and let $G'$ be the corresponding quotient of $\Sigma'$ by the action of $R^\times$. For each $x \in \Sigmao$, let $U_x$ be the ultrafilter $\{e \in I(B) : e(x) = 1\}$ of $I(B)$ corresponding to $x$. For $\sigma \in \Sigma$, let $X$ be a compact open bisection containing $\sigma$, and let
\[
\Phi(\sigma) \coloneqq (\tilde{1}_X U_{\s(\sigma)})^{\uparrow} \in \Sigma'.
\]
Then $\Phi(\sigma)$ does not depend on the choice of $X$. The map $\Phi\colon \Sigma \to \Sigma'$ is an $R^\times$-equivariant continuous open embedding of topological groupoids such that $\Phi(\Sigmao)=\Sigma'^{(0)}$.
\end{prop}

\begin{proof}
\Cref{lem:ind.choice.bisection}\cref{item1:ind.choice.bisection} shows that $\Phi(\sigma)$ does not depend on the choice of $X$. To see that $\Phi$ is injective, suppose that $\sigma$ and $\tau$ are distinct elements of $\Sigma$. Since $\Sigma$ is Hausdorff, there is a compact open bisection $X \subseteq \Sigma$ such that $\sigma \in X$ and $\tau \notin X$. In particular, $\tilde{1}_X \tilde{1}_{\s(X)} = \tilde{1}_X \in \Phi(\sigma) \setminus \Phi(\tau)$ by \cref{cor:normalisermult} and \cref{lem:ind.choice.bisection}\cref{item2:ind.choice.bisection}, and so $\Phi(\sigma) \ne \Phi(\tau)$.

Fix $x \in \Sigmao$. By taking any compact open neighbourhood $X \subseteq \Sigmao$ of $x$, we see that $\Phi(x) = (\tilde{1}_X U_x)^{\uparrow} = (U_x)^{\uparrow}$ is the unique ultrafilter of $N(B)$ containing $U_x$. By Stone duality, it follows that $\Phi$ restricts to a homeomorphism from $\Sigmao$ to $\Sigma'^{(0)}$.

Fix $\sigma \in \Sigma$ and let $X \subseteq \Sigma$ be a compact open bisection containing $\sigma$. Put $V\coloneqq (\tilde{1}_XU_{\s(\sigma)})^\uparrow=\Phi(\sigma)$. Then $\s(\Phi(\sigma)) = V^{-1}V$. If $e \in U_{\s(\sigma)}$, then $e \ge (\tilde{1}_X e)^\dagger (\tilde{1}_X e) \in V^{-1}V$, and so $e \in V^{-1}V$. Therefore, $\Phi(\s(\sigma)) = (U_{\s(\sigma)})^{\uparrow}\subseteq V^{-1}V$, and so by the maximality property of ultrafilters, we have $\Phi(\s(\sigma)) = V^{-1}V = \s(\Phi(\sigma))$. Suppose that $e \in U_{\r(\sigma)} \subseteq \Phi(\r(\sigma))$. Then $(\tilde{1}_X^\dagger e \tilde{1}_X)(\s(\sigma)) = 1$ by \cref{prop:Steinberg diagonal}, and so $\tilde{1}_X^\dagger e \tilde{1}_X \in U_{\s(\sigma)} = V^{-1}V$. Therefore, $e \ge \tilde{1}_X(\tilde{1}_X^\dagger e \tilde{1}_X) \tilde{1}_X^\dagger \in V(V^{-1}V)V^{-1} = VV^{-1}$, and so $e \in VV^{-1}$. Thus $\Phi(\r(\sigma)) = (U_{\r(\sigma)})^{\uparrow} \subseteq VV^{-1} = \r(\Phi(\sigma))$, and so $\Phi(\r(\sigma)) = \r(\Phi(\sigma))$ since these are ultrafilters. It follows that $\Phi$ carries composable pairs to composable pairs.

To see that $\Phi$ is multiplicative, fix $(\sigma,\tau) \in \Sigmac$, and choose a compact open bisection $W \subseteq \Sigma$ containing $\sigma\tau$. By the continuity of multiplication, there exist compact open bisections $X$ and $Y$ of $\Sigma$ such that $\sigma \in X$, $\tau \in Y$, and $XY \subseteq W$. Then $\tilde{1}_X \tilde{1}_Y = \tilde{1}_{XY} \le \tilde{1}_W$ by \cref{cor:normalisermult} and \cref{lem:norm.vs.indicator}\cref{item:nvi.Y.sub.X}. It follows that $\tilde{1}_W \in \Phi(\sigma) \Phi(\tau)$. Since $\s(\tau) = \s(\sigma\tau)$, it follows that $\tilde{1}_W U_{\s(\sigma\tau)} \subseteq \Phi(\sigma) \Phi(\tau) U_{\s(\tau)} \subseteq \Phi(\sigma) \Phi(\tau)$, and so $\Phi(\sigma\tau) \subseteq \Phi(\sigma) \Phi(\tau)$. Hence $\Phi(\sigma\tau) = \Phi(\sigma) \Phi(\tau)$ by the maximality property of ultrafilters.

To see that $\Phi$ is continuous, fix $n \in N(B)$, and let $\sigma \in \Phi^{-1}(\VV_n)$. Choose a compact open bisection $X \subseteq \Sigma$ containing $\sigma$. Then $(\tilde{1}_X U_{\s(\sigma)})^{\uparrow} = \Phi(\sigma) \in \VV_n$, and so by the definition of $\VV_n$, we have $n \in (\tilde{1}_X U_{\s(\sigma)})^{\uparrow}$. It follows from \cref{lem:ind.choice.bisection}\cref{item3:ind.choice.bisection} that there is a compact open set $W \subseteq \s(X)$ such that $\sigma \in XW$ and $\tilde{1}_{XW} \tilde{1}_{\s(X)} = \tilde{1}_{XW} \le n$. Hence $\Phi(\tau) \in \VV_n$ for each $\tau \in XW$, and so $XW$ is an open neighbourhood of $\sigma$ contained in $\Phi^{-1}(\VV_n)$.

To see that $\Phi$ is open, it suffices to fix a compact open bisection $X$ of $\Sigma$, and show that $\Phi(X) = \VV_{\tilde{1}_X}$. If $\sigma \in X$, then certainly $\tilde{1}_X \in \Phi(\sigma)$, and so $\Phi(\sigma) \in \VV_{\tilde{1}_X}$. This gives $\Phi(X) \subseteq \VV_{\tilde{1}_X}$. For the reverse containment, suppose that $U \in \VV_{\tilde{1}_X}$. Then $\tilde{1}_X \in U$, and so $\tilde{1}_{\s(X)} = \tilde{1}_X^\dagger \tilde{1}_X \in \s(U)$. We already know that $\s(U) = (U_x)^{\uparrow}$ for some $x \in \Sigmao$, and hence $(\tilde{1}_{\s(X)})(x) = 1$, which implies that $x \in \s(X)$. Let $\sigma$ be the unique element of $X$ such that $\s(\sigma) = x$. We claim that $U = \Phi(\sigma)$. Indeed, since $\s(U) = (U_x)^{\uparrow} = (U_{\s(\sigma)})^{\uparrow}$ and $\tilde{1}_X \in U$, we have $\tilde{1}_X U_{\s(\sigma)} \subseteq U$, and hence $\Phi(\sigma)\subseteq U$. Thus $U = \Phi(\sigma) \subseteq \Phi(X)$ by the maximality property of ultrafilters, and so $\Phi$ is an open map.

It remains to prove $R^\times$-equivariance. Fix $t \in R^\times$ and $\sigma \in \Sigma$. Fix a compact open bisection $X \subseteq \Sigma$ containing $\sigma$. Then $t \cdot X$ is a compact open bisection containing $t \cdot \sigma$, and $\tilde{1}_{t\cdot X} = t \tilde{1}_X$. Since $\s(t\cdot \sigma)=\s(\sigma)$, we have $\Phi(t \cdot \sigma) = (\tilde{1}_{t \cdot X} U_{\s(\sigma)})^{\uparrow}$. Now
\[
\tilde{1}_{t\cdot X} U_{\s(\sigma)} = (t \tilde{1}_X) U_{\s(\sigma)} \subseteq t \cdot \Phi(\sigma),
\]
and thus $\Phi(t \cdot \sigma) \subseteq t \cdot \Phi(\sigma)$. Hence $\Phi(t \cdot \sigma) = t \cdot \Phi(\sigma)$ since they are ultrafilters.
\end{proof}

It is now relatively straightforward to show that $\Phi$ induces a morphism of discrete
$R$-twists.

\begin{cor} \label{cor:comm.diagram.twist}
Let $(\Sigma,i,q)$ be a discrete $R$-twist over an ample Hausdorff groupoid $G$, let $A \coloneqq A_R(G;\Sigma)$, and let $B \coloneqq A_R(\Go; q^{-1}(\Go))$. Let
\[
G'^{(0)} \times R^\times \overset{i'} \hookrightarrow \Sigma' \overset{q'} \twoheadrightarrow G'
\]
be the sequence obtained from $(A,B)$ via \cref{thm:the twist}. Let $\Phi\colon \Sigma \to \Sigma'$ be the map of \cref{prop:embedding}. Then there is a (well-defined) map $\Phi_G\colon G \to G'$ given by $\Phi_G(q(\sigma))=q'(\Phi(\sigma))$, and the following diagram commutes.
\[
\begin{tikzcd}
\Go\times R^\times \ar[r,hook,"i"]\ar[d,"\Phi_G \times \id"] & \Sigma \ar[r,two heads, "q"]\ar[d,"\Phi"] & G \ar[d,"\Phi_G"] \\
G'^{(0)}\times R^\times \ar[r,hook,"i'"] & \Sigma' \ar[r,two heads, "q'"] & G'
\end{tikzcd}
\]
\end{cor}

\begin{proof}
\Cref{prop:embedding} shows that $\Phi$ is $R^\times$-equivariant. Thus $\Phi_G$ is well-defined, and the right-hand square of the diagram commutes. To see that the left-hand square commutes, observe that for all $x \in \Sigmao$, we have
\[
\Phi(i(q(x),t)) = \Phi(t \cdot x) = t \cdot \Phi(x) = i'(q'(\Phi(x)),t) = i'(\Phi_G(q(x)),t). \qedhere
\]
\end{proof}

To proceed, we need the following observation about normalisers $n \in N(B)$ with $\supp_G(n)$ a bisection.

\begin{prop} \label{prop:local.bisection.is.indicator}
Fix $n \in N(B)$. Then $\supp_G(n)$ is a bisection if and only if $n=\tilde{1}_X$ for a (necessarily unique) compact open bisection $X$ of $\Sigma$.
\end{prop}

\begin{proof}
Suppose that $X \subseteq \Sigma$ is a compact open bisection of $\Sigma$. Then $\supp_G(\tilde{1}_X) = q(X)$ is a compact open bisection of $G$, since $q$ is a continuous open map that restricts to a homeomorphism of unit spaces. Moreover, $\tilde{1}_X^{-1}(1) = X$, and so uniqueness is clear.

For the converse, suppose that $\supp_G(n)$ is a compact open bisection of $G$. Note that $n^\dagger n = \tilde{1}_{\s(\supp(n))}$ by \cref{rmk:Steinberg4.5}. Fix $\sigma \in \supp(n)$. Then $1 = (n^\dagger n)(\s(\sigma)) = n^\dagger(\sigma^{-1})n(\sigma)$ since $\supp_G(n)$ is a bisection, and so $n(\sigma)\in R^\times$. It follows that $n$ takes values in $R^\times$, and so by the $R^\times$-contravariance of $n$, we have $\supp(n) = R^\times \cdot X$, where $X\coloneqq n^{-1}(1)$. Since $n$ is continuous, $X$ is open. Moreover, $q(X) = \supp_G(n)$ by the above observation. We claim that $q\restr{X}$ is injective. To see this, suppose that $q(\sigma) = q(\tau)$ for $\sigma, \tau \in X$. Then $\tau = t \cdot \sigma$ for some $t \in R^\times$, and so $n(\sigma) = 1 = n(\tau) = t^{-1}$. Hence $t = 1$, and so $\sigma = \tau$. Therefore, $X$ is homeomorphic to $q(X) = \supp_G(n)$, and hence $X$ is compact. To see that $X$ is a bisection, note that if $\s(\sigma) = \s(\tau)$ for $\sigma, \tau \in X$, then $q(\sigma) = q(\tau)$ since $\supp_G(n)$ is a bisection, and so $\sigma = \tau$ by the injectivity of $q\restr{X}$. Similarly, $\r\restr{X}$ is injective. Since $\tilde{1}_X$ and $n$ are supported on $R^\times \cdot X$ and agree on $X$, we deduce that $n = \tilde{1}_X$ by \cref{lem:unique extensions}.
\end{proof}

It follows from \cref{prop:local.bisection.is.indicator} that $X \mapsto \tilde{1}_X$ is an isomorphism of the inverse semigroup of compact open bisections of $\Sigma$ onto $N(B)$ if the twist $\Sigma \to G$ satisfies the local bisection hypothesis (using \cref{cor:normalisermult}). Therefore, we deduce by so-called noncommutative Stone duality~\cite{LawsonLenz2013} that $\Sigma \cong \Sigma'$. Moreover, the map $\Phi$ is the isomorphism that noncommutative Stone duality provides. The next result gives a more precise statement; namely, that the map $\Phi$ of \cref{prop:embedding} is surjective precisely when the pair of algebras associated to $G$ is an algebraic quasi-Cartan pair.

\begin{theorem} \label{thm:reconstructing.the.twist}
Let $(\Sigma,i,q)$ be a discrete $R$-twist over an ample Hausdorff groupoid $G$, let $A \coloneqq A_R(G;\Sigma)$, and let $B \coloneqq A_R(\Go; q^{-1}(\Go))$. Let $\Sigma'$ be the groupoid of ultrafilters of $N(B)$ defined in \cref{sec:Sigma def}. Let $\Phi\colon \Sigma \to \Sigma'$ be the map from \cref{prop:embedding}. The following are equivalent.
\begin{enumerate}[label=(\arabic*)]
\item \label{item:iso.twist.qcp} The pair $(A,B)$ is an algebraic quasi-Cartan pair.
\item \label{item:iso.twist.lbh} The twist $(\Sigma,i,q)$ satisfies the local bisection hypothesis.
\item \label{item:iso.twist.Phi.surjective} The map $\Phi$ is surjective and, in particular, it is an isomorphism of topological groupoids.
\end{enumerate}
\end{theorem}

\begin{proof}
\Cref{lem:LBH<->AQP} gives \cref{item:iso.twist.qcp}$\iff$\cref{item:iso.twist.lbh}.

To see that \cref{item:iso.twist.lbh}$\implies$\cref{item:iso.twist.Phi.surjective}, let $U \in \Sigma'$ and $n \in U$. Then $\supp_G(n)$ is a bisection of $G$, and so by \cref{prop:local.bisection.is.indicator}, there is a compact open bisection $X$ of $\Sigma$ such that $n = \tilde{1}_X$. For each $x \in \Sigmao$, let $U_x$ be the ultrafilter $\{e \in I(B) : e(x) = 1\}$ of $I(B)$. By \cref{prop:embedding}, we have $\s(U) = (U_x)^\uparrow$ for some $x \in \Sigmao$, and since $\tilde{1}_{\s(X)} = n^\dagger n \in \s(U)$, there exists $e \in U_x$ such that $e \le \tilde{1}_{\s(X)}$. Since $e(x) = 1$, \cref{lem:norm.vs.indicator}\cref{item:nvi.Y.sub.X} implies that $x \in \s(X)$. Since $X$ is a bisection, there is a unique $\sigma \in X$ with $\s(\sigma) = x$. Note that $\tilde{1}_X U_x = n U_x \subseteq U\s(U) = U$, and hence $\Phi(\sigma) = (\tilde{1}_X U_{\s(\sigma)})^{\uparrow}\subseteq U$. Thus $U = \Phi(\sigma)$ since these are ultrafilters. Therefore, $\Phi$ is surjective and hence is an isomorphism by \cref{prop:embedding}.

To see that \cref{item:iso.twist.Phi.surjective}$\implies$\cref{item:iso.twist.lbh}, fix $n \in N(B) {\setminus} \{0\}$ and $x \in \s(\supp_G(n))$. Note that $n^\dagger n = \tilde{1}_{\s(\supp(n))}$ by \cref{rmk:Steinberg4.5}, and hence $\s(\supp_G(n)) = \supp_G(n^\dagger n)$. Therefore, $n^\dagger n \in U_x$, and so \cref{lem:germ-ultrafilter} implies that $(nU_x)^{\uparrow}$ is an ultrafilter of $N(B)$ containing $n$, with source $U_x$. Since $\Phi$ is surjective and respects the source map, there exists $\tau \in \Sigma$ with $\s(\tau) = x$ such that $\Phi(\tau) = (nU_x)^{\uparrow}$. So $n \in \Phi(\tau)$. Let $X$ be a compact open bisection of $\Sigma$ containing $\tau$, so that $\Phi(\tau) = (\tilde{1}_X U_{\s(\tau)})^{\uparrow}$. By \cref{lem:ind.choice.bisection}\cref{item3:ind.choice.bisection}, there exists a compact open set $W \subseteq \s(X)$ such that $\tau\in XW$ and $\tilde{1}_{XW}\le n$. Since $XW \subseteq X$ is a bisection, \cref{lem:norm.vs.indicator}\cref{lem:nvi.i.leq.n} implies that $q(\tau)$ is the unique element of $\supp_G(n)$ with $\s(q(\tau)) = x$. Since $x \in \s(\supp_G(n))$ was arbitrary, it follows that $\s$ is injective on $\supp_G(n)$. A symmetric argument shows that $\r$ is injective on $\supp_G(n)$, and so $\supp_G(n)$ is a bisection.
\end{proof}

To finish off, we deduce that the local bisection hypothesis is encoded algebraically in the sense that if the algebraic pairs corresponding to two twists are isomorphic, then either neither of the two twists satisfies the local bisection hypothesis, or they both do.

\begin{definition}
Let $(\Sigma_1,i_1,q_1)$ and $(\Sigma_2,i_2,q_2)$ be discrete $R$-twists over ample Hausdorff groupoids $G_1$ and $G_2$, respectively. We say that an isomorphism $\Psi\colon A_R(G_1;\Sigma_1) \to A_R(G_2;\Sigma_2)$ is \emph{diagonal-preserving} if $\Psi\big(A_R(G^{(0)}_1; q_1^{-1}(G^{(0)}_1))\big) = A_R(G^{(0)}_2; q_2^{-1}(G^{(0)}_2))$.
\end{definition}

\begin{remark}
Let $(\Sigma,i,q)$ be a discrete $R$-twist over $G$, let $A \coloneqq A_R(G;\Sigma)$, and let $B \coloneqq A_R(\Go; q^{-1}(\Go))$. As in the untwisted case, we call $B$ the \emph{diagonal subalgebra} of $A$, even though \cref{prop:effective} says that $(A,B)$ is an algebraic diagonal pair if and only if $G$ is principal.
\end{remark}

\begin{cor} \label{cor:equiv.twist}
Let $(\Sigma_1,i_1,q_1)$ and $(\Sigma_2,i_2,q_2)$ be discrete $R$-twists over ample Hausdorff groupoids $G_1$ and $G_2$, respectively. Suppose that $\Sigma_1 \to G_1$ satisfies the local bisection hypothesis. The following are equivalent.
\begin{enumerate}[label=(\arabic*)]
\item \label{item:cor.equiv.twist} The twists $(\Sigma_1,i_1,q_1)$ and $(\Sigma_2,i_2,q_2)$ are isomorphic.
\item \label{item:cor.R-alg.iso} There exists a diagonal-preserving isomorphism of $R$-algebras $\Psi\colon A_R(G_1;\Sigma_1)\to A_R(G_2;\Sigma_2)$.
\end{enumerate}
\end{cor}

\begin{proof}
The implication \cref{item:cor.equiv.twist}$\implies$\cref{item:cor.R-alg.iso} is immediate. We prove the implication \cref{item:cor.R-alg.iso}$\implies$\cref{item:cor.equiv.twist}. Write $(A_1,B_1)$ and $(A_2,B_2)$ for the algebraic pairs associated to $\Sigma_1 \to G_1$ and $\Sigma_2 \to G_2$, respectively. We first show that $(A_2,B_2)$ is an algebraic quasi-Cartan pair. Define $\tilde{P}\colon A_2 \to B_2$ by $\tilde{P}(a_2) = \Psi(P_1(\Psi^{-1}(a_2)))$, where $P_1$ is the unique faithful conditional expectation that is implemented by idempotents for the pair $(A_1,B_1)$. Since $\Psi$ is diagonal-preserving, it is straightforward to prove that $\tilde{P}\restr{B_2} = \id_{B_2}$, that $\tilde{P}(b_2a_2b'_2) = b_2 \tilde{P}(a_2)b'_2$ for all $a_2 \in A_2$ and $b_2,b'_2 \in B_2$, and that $\Psi(N(B_1)) = N(B_2)$. The latter implies that $\tilde{P}$ is faithful and that it is implemented by idempotents. Since $\Psi$ and $P_1$ are $R$-linear, $\tilde{P}$ is $R$-linear. Hence the pair $(A_2,B_2)$ with $\tilde{P}$ is an algebraic quasi-Cartan pair. By \cref{lem:LBH<->AQP}, $(\Sigma_2,i_2,q_2)$ satisfies the local bisection hypothesis, and by \cref{prop:canonical CE,lem:LBH<->AQP}, $\tilde{P}$ coincides with the conditional expectation of $(A_2,B_2)$ given by restriction of
functions.

We now build an isomorphism between the twists. By \cref{thm:reconstructing.the.twist}\cref{item:iso.twist.Phi.surjective} and \cref{cor:comm.diagram.twist}, we can identify each $\Sigma_i$ with the groupoid of ultrafilters in $N(B_i)$, and we can identify each $G_i$ with the corresponding quotient. Since the partial order on each $N(B_i)$ is defined in terms of the multiplication, it follows that if $U$ is a filter of $N(B_1)$, then $\Psi(U)$ is a filter of $N(B_2)$, and if $V$ is a filter of $N(B_2)$, then $\Psi^{-1}(V)$ is a filter of $N(B_1)$. Hence the map $\psi\colon \Sigma_1 \to \Sigma_2$ given by $\psi(U) \coloneqq \Psi(U)$ is a well-defined bijection. Moreover, since the multiplication and topology on each $\Sigma_i$ depend only on the multiplicative structure of $N(B_i)$, it is straightforward to show that $\psi$ is a topological groupoid isomorphism. Since $\Psi$ is $R$-linear, we have $\psi(tU) = t \psi(U)$ for all $t \in R^\times$ and $U \in \Sigma_1$, and so $\psi$ is $R^\times$-equivariant. Hence there is a (well-defined) topological groupoid isomorphism $\psi_G\colon G_1 \to G_2$ given by $\psi_G(q_1(U)) \coloneqq q_2(\psi(U))$. It follows that the diagram
\[
\begin{tikzcd}
G_1^{(0)}\times R^\times \ar[r,hook,"i_1"]\ar[d,"\psi_G \times \id"] & \Sigma_1 \ar[r,two heads, "q_1"]\ar[d,"\psi"] & G_1 \ar[d,"\psi_G"] \\
G_2^{(0)}\times R^\times \ar[r,hook,"i_2"] & \Sigma_2 \ar[r,two heads, "q_2"] & G_2
\end{tikzcd}
\]
gives an isomorphism between of the twists $(\Sigma_1,i_1,q_1)$ and $(\Sigma_2,i_2,q_2)$.
\end{proof}

If at least one of the two twists is effective, then we can strengthen the previous result by asking only for an isomorphism of twisted Steinberg algebras that restricts to an inclusion of diagonal subalgebras; this implies that in fact it restricts to an isomorphism of diagonal subalgebras.

\begin{cor} \label{cor:equiv.twist.effective}
Let $(\Sigma_1,i_1,q_1)$ and $(\Sigma_2,i_2,q_2)$ be discrete $R$-twists over ample Hausdorff groupoids $G_1$ and $G_2$, respectively. Suppose that $G_1$ is effective. The following are equivalent.
\begin{enumerate}[label=(\arabic*)]
\item \label{item:cor.eff.equi.twist} The twists $(\Sigma_1,i_1,q_1)$ and $(\Sigma_2,i_2,q_2)$ are isomorphic.
\item \label{item:cor.eff.R-alg.iso} There exists a diagonal-preserving isomorphism of $R$-algebras $\Psi\colon A_R(G_1;\Sigma_1)\to A_R(G_2;\Sigma_2)$.
\item \label{item:cor.eff.R-alg.iso.ma} There exists an isomorphism of $R$-algebras $\Psi\colon A_R(G_1;\Sigma_1)\to A_R(G_2;\Sigma_2)$ such that $\Psi\big(A_R(G^{(0)}_1; q_1^{-1}(G^{(0)}_1))\big) \subseteq A_R(G^{(0)}_2; q_2^{-1}(G^{(0)}_2))$.
\end{enumerate}
\end{cor}

\begin{proof}
Since $G_1$ is effective, \cref{lem:LBH<->AQP,lem:C=>Q,prop:effectiveACPprincipalADP} together imply that $(\Sigma_1,i_1,q_1)$ satisfies the local bisection hypothesis. The equivalence \cref{item:cor.eff.equi.twist}$\iff$\cref{item:cor.eff.R-alg.iso} follows from \cref{cor:equiv.twist}. The implication \cref{item:cor.eff.R-alg.iso}$\implies$\cref{item:cor.eff.R-alg.iso.ma} is clear. We prove that \cref{item:cor.eff.R-alg.iso.ma}$\implies$\cref{item:cor.eff.R-alg.iso}. Let $\Psi$ be an isomorphism as in \cref{item:cor.eff.R-alg.iso.ma}. Since $G_1$ is effective, \cref{prop:effectiveACPprincipalADP} implies that $B_1 \coloneqq A_R(G^{(0)}_1; q^{-1}(G^{(0)}_1))$ is a maximal commutative subalgebra of $A_R(G_1;\Sigma_1)$, which implies that $\Psi(B_1)$ is a maximal commutative subalgebra of $A_R(G_2;\Sigma_2)$. Since $A_R(G^{(0)}_2; q^{-1}(G^{(0)}_2))$ is a commutative subalgebra containing $\Psi(B_1)$, they are equal.
\end{proof}

\section{Examples satisfying the local bisection hypothesis} \label{sec:examples}

In this short section we show that our theorems apply to a relatively large class of groupoids and twists. In particular, all twists over Deaconu--Renault groupoids with totally disconnected unit spaces satisfy the local bisection hypothesis.

The starting point is to show that there are plenty of groups whose twisted group algebras over large classes of indecomposable rings have no nontrivial units. Most of what we write here is well known.

The question as to which group rings have no nontrivial units goes back to the thesis of Graham Higman from the early 1940s. The famous Kaplansky unit conjecture asserts that if $R$ is an integral domain and $G$ is a torsion-free group, then $RG$ has no nontrivial units. Not long after this paper was first posted to ArXiv, Gardam \cite{Gardam2021} produced the first known counterexample to the Kaplansky unit conjecture; the group is torsion-free virtually abelian (in fact, crystallographic) and Gardam showed that it has nontrivial units over fields of characteristic $2$. Murray recently extended Gardam's result by showing that the same group has nontrivial units over any field of positive characteristic \cite{Murray2021}. The question of whether there are counterexamples in characteristic zero is still open. In any event, it is now clear that the class of groups whose group rings have no nontrivial units is less wide than previously hoped. So let us turn to those classes of groups for which the unit conjecture is known to hold; twisted group rings of groups in these classes also have no nontrivial units, which is the relevance for us.

Recall that a group $G$ is \emph{right orderable} if there is a total order $<$ on $G$ such that $x<y$ implies that $xz<yz$ for all $x,y,z\in G$. If, in addition, the order satisfies $zx<zy$ for all $x,y,z\in G$ such that $x < y$, then $G$ is said to be \emph{orderable}. Note that a group $G$ is (right) orderable if and only if all of its finitely generated subgroups are too \cite[Lemma~13.2.2]{PassmanBook}.

A group $G$ is said to have the \emph{unique product property} if, given two finite nonempty subsets $A$ and $B$ of $G$, there is an element of $AB$ that can be uniquely written as a product $ab$ with $a \in A$ and $b \in B$. This property is satisfied by any right-orderable group \cite[Theorem~13.1.7]{PassmanBook} and so, in particular, by any torsion-free abelian group, torsion-free nilpotent group \cite[Lemma~13.1.6]{PassmanBook}, and free group. Note that locally indicable groups are right orderable. Here, a group $G$ is \emph{locally indicable} if every finitely generated subgroup admits a nontrivial homomorphism to an infinite cyclic group.

It was shown by Strojnowski~\cite{Strojnowski} that if $G$ has the unique product property and $A$ and $B$ are finite subsets of $G$ with $\lvert A \rvert + \lvert B \rvert > 2$, then there are two distinct elements $g,h\in AB$ which have unique representations $g=ab$ and $h=a'b'$ with $a,a'\in A$ and $b,b'\in B$. Note that every group satisfying the unique product property is torsion-free, and that there are groups satisfying the unique product property that are not right orderable.

If $R$ is an integral domain and $G$ has the unique product property, then it is well known that the group ring $RG$ has no nontrivial units. The same argument applies \emph{mutatis mutandis} to twisted group rings. A more general result for group graded rings can be found in~\cite[Theorem~3.4]{Oinert}. We include the elementary proof for completeness.

\begin{lemma} \label{lem:unique.prod}
Let $G$ be a group satisfying the unique product property, let $R$ be a unital integral domain, and let $c$ be a normalised $R^\times$-valued $2$-cocycle on $G$. Then the only units of the twisted group ring $R(G,c)$ are the elements of the set $\{t \delta_g : t \in R^\times, g \in G\}$.
\end{lemma}

\begin{proof}
Let $e$ denote the identity of $G$. Suppose that $a,b \in R(G,c)$ satisfy $ab = ba = 1_{R(G,c)} = \delta_e$. Write $a = \sum_{g \in G} a_g \delta_g$ and $b = \sum_{h \in G} b_h \delta_h$. Let $A \coloneqq \{g : a_g \ne 0\}$ and $B \coloneqq \{h : b_h \ne 0\}$. These are finite sets. If $a$ and $b$ are not both trivial units, then $\lvert A \rvert + \lvert B \rvert > 2$. So there exist distinct $g, h \in AB$ which have unique factorisations $g=xy$ and $h=x'y'$ with $x,x'\in A$ and $y,y'\in B$. It follows that the coefficients of $\delta_g$ and $\delta_h$ in $ab$ are $a_xb_yc(x,y)$ and $a_{x'}b_{y'}c(x',y')$, respectively, and these are both nonzero since $R$ is an integral domain. This contradicts $ab=\delta_e$. Therefore, $\lvert A \rvert + \lvert B \rvert \le 2$, and so $a=a_g\delta_g$ and $b=b_h\delta_h$ with $g,h\in G$. We then have $a_gb_hc(g,h) = 1$, and so $a_g, b_h\in R^\times$, as required.
\end{proof}

We call $G$ a \emph{trivial-units-only group} if the twisted group ring $k(G,c)$ has no nontrivial units for every field $k$ and every normalised $k^\times$-valued $2$-cocycle $c$ on $G$. For example, a group satisfying the unique product property is a trivial-units-only group by \cref{lem:unique.prod}.

Recall that a commutative unital ring $R$ is \emph{reduced} if it has no nonzero nilpotent elements. This is equivalent to the intersection of all prime ideals of $R$ being trivial. Every integral domain is reduced and indecomposable. However, the class of reduced and indecomposable commutative rings is much larger and includes the coordinate rings of connected (but not necessarily irreducible) affine varieties over algebraically closed fields. For instance, $\C[x,y]/(xy)$, which is the coordinate ring of the union of the coordinate axes in $\C^2$, is reduced and indecomposable but not an integral domain.

The following proposition is a variation of \cite[Proposition~2.1]{Steinberg2019}, which is based on \cite[Theorem~3]{Neher}. It generalises \cref{lem:unique.prod}.

\begin{prop} \label{prop:Neher}
Let $G$ be a nontrivial trivial-units-only group as defined above (for example, a group satisfying the unique product property). Let $R$ be a commutative unital ring and let $c$ be an $R^\times$-valued normalised $2$-cocycle on $G$. Then $R(G,c)$ has no nontrivial units if and only if $R$ is reduced and indecomposable.
\end{prop}

\begin{proof}
Let $e$ denote the identity of $G$.
We begin with necessity; this argument applies to all nontrivial groups $G$. Fix $g\in G{\setminus} \{e\}$. Suppose first that $R$ is not indecomposable, and let $f \in R {\setminus} \{0,1\}$ be an idempotent. Then $a = f\delta_e+(1-f)\delta_g$ is a nontrivial unit with inverse $b = f\delta_e+(1-f)c(g,g^{-1})^{-1}\delta_{g^{-1}}$. Indeed, using that $c$ is normalised, we have that $ab = \big(f+(1-f)c(g,g^{-1})c(g,g^{-1})^{-1}\big)\delta_e = \delta_e$, and similarly, $ba=\delta_e$, using that $c(g,g^{-1})=c(g^{-1},g)$, since $c$ is a normalised $2$-cocycle.

Next assume that $R$ is not reduced. If $n \in R {\setminus} \{0\}$ is nilpotent, then $n\delta_g$ is nilpotent, and so $\delta_e - n\delta_g$ is a nontrivial unit with inverse $\delta_e + n\delta_g$.

For sufficiency, we use some algebraic geometry, following~\cite{Neher}. Since $R$ is indecomposable, its prime (Zariski) spectrum $\Spec(R)$ is connected. If $\fp \in \Spec(R)$, then the residue field at $\fp$ is the field of fractions $\kappa(\fp)$ of $R/\fp$. For each $r \in R$, we write $r(\fp)$ for the image of $r$ in $\kappa(\fp)$ under the composition $R \to R/\fp \to \kappa(\fp)$. We have an induced $\kappa(\fp)^\times$-valued normalised $2$-cocycle $c_{\fp}$ on $G$ obtained by putting $c_{\fp}(g,h) = c(g,h)(\fp)$. Moreover, there is a canonical surjective homomorphism $\pi_{\fp}\colon R(G,c)\to \kappa(\fp)(G,c_{\fp})$ sending $\sum_{g\in G}a_g\delta_g$ to $a(\fp)\coloneqq \sum_{g\in G}a_g(\fp)\delta_g$.

Let $a \in R(G,c)$ be a unit, say $a = \sum_{g\in G} a_g\delta_g$. Then $a(\fp)$ is a unit of $\kappa(\fp)(G,c_{\fp})$ and hence is a trivial unit by the hypothesis that $G$ is a trivial-units-only group. It follows that there is a unique element $h \in G$ with $a_h(\fp) \ne 0$. Denote this element by $g(\fp)$. Then $g\colon \fp \mapsto g(\fp)$ is a map from $\Spec(R)$ to $G$. We claim that $g$ is locally constant. Indeed, if $g(\fp) = h$, then $a_h\notin \fp$. Let $D(a_h)$ denote the basic open set of all prime ideals not containing $a_h$. For $\fq \in D(a_h)$, we have $a_h(\fq)\ne 0$, and so $g(\fq) = h$. Therefore, $g$ is constant on $D(a_h)$, and so $g$ is locally constant, as claimed. Since $\Spec(R)$ is connected, we deduce that $g$ is constant. Hence there exists $h \in G$ such that $a_h(\fp) \ne 0$ for all $\fp \in \Spec(R)$; and for all $h' \ne h$, we have $a_{h'}(\fp) = 0$ for all $\fp \in \Spec(R)$. Since $R$ is reduced, the intersection of all of its prime ideals is zero, and so $a_{h'} = 0$ for $h'\ne h$. And, since $a_h$ belongs to no prime ideal, and hence to no maximal ideal, it is a unit. Thus $a = a_h\delta_h$ is a trivial unit.
\end{proof}

\begin{cor} \label{cor:application}
Let $R$ be a reduced and indecomposable commutative ring, and let $(\Sigma,i,q)$ be a discrete $R$-twist over an ample Hausdorff groupoid $G$. Let $\II$ denote the interior of the isotropy of $G$, and suppose that there is a dense set $X\subseteq \Go$ such that for each $x \in X$, the group $\II_x$ satisfies the unique product property (for example, it might be a right-orderable group). Let $A \coloneqq A_R(G;\Sigma)$ and $B \coloneqq A_R(\Go; q^{-1}(\Go)) \subseteq A$. Then $(A,B)$ is an algebraic quasi-Cartan pair, and the map $\Phi$ of \cref{prop:embedding} from $\Sigma$ to the twist $\Sigma'$ is an isomorphism of twists.
\end{cor}

\begin{proof}
\Cref{lem:unique.prod,prop:Neher} show that twisted group rings of groups satisfying the unique product property have no nontrivial units, and so the hypothesis of \cref{lem:iso LBH} is satisfied. Therefore, \cref{lem:iso LBH} and \cref{lem:global LBH}\cref{item2:global LBH} together show that $(\Sigma,i,q)$ satisfies the local bisection hypothesis. The result follows from \cref{cor:comm.diagram.twist,thm:reconstructing.the.twist}.
\end{proof}

Since free abelian groups are orderable and hence have the unique product property, we have the following corollary of \cref{cor:application}.

\begin{cor} \label{cor:DR-groupoids}
Let $X$ be a totally disconnected locally compact Hausdorff space, and let $T\colon n \mapsto T_n$ be an action of $\N^k$ on $X$ by local homeomorphisms. Let $G$ be the associated Deaconu--Renault groupoid, let $R$ be a reduced and indecomposable commutative ring, and let $(\Sigma,i,q)$ be a discrete $R$-twist over $G$. Let $(A,B) \coloneqq (A_R(G;\Sigma), A_R(\Go; q^{-1}(\Go)))$ be the associated algebraic pair. Then $(A,B)$ is an algebraic quasi-Cartan pair, and \cref{thm:reconstructing.the.twist} recovers the groupoid $G$ and the discrete twist $\Sigma$ from $(A,B)$.
\end{cor}

\begin{example} \label{eg:KP-algs}
Let $\Lambda$ be a countable row-finite higher-rank graph with no sources as in \cite{KP2000NYJM}, let $R$ be a reduced and indecomposable commutative ring, and let $\omega$ be normalised $R^\times$-valued $2$-cocycle on $\Lambda$ as in \cite{KPS2015TAMS}. The arguments of \cite[Section~6]{KPS2015TAMS} show how to construct a normalised locally constant $R^\times$-valued $2$-cocycle $c_\omega$ on the $k$-graph groupoid $G_\Lambda$. Since this groupoid is a Deaconu--Renault groupoid, our results show that the twist $G_\Lambda \times_{c_\omega} R^\times$ over $G_\Lambda$ obtained from $c_\omega$ can be recovered from the twisted Steinberg algebra $A_R(G_\Lambda; G_\Lambda \times_{c_\omega} R^\times)$ and its canonical commutative subalgebra (isomorphic to) $C_c(\Lambda^\infty,R)$. Moreover, $A_R(G_\Lambda;G_\Lambda \times_{c_\omega} R^\times) \cong A_R(G_\Lambda,c_\omega)$ as in \cite[Theorem~4.23]{ACCLMR2022}, and one can recover the cohomology class of ${c_\omega}$ from the twist $G_\Lambda \times_{c_\omega} R^\times$ over $G_\Lambda$ as in \cite[Corollary~4.16]{ACCLMR2022}. So the groupoid $G_\Lambda$ and the cohomology class of $c_\omega$ are invariants of the pair $(A_R(G_\Lambda,c_\omega), C_c(\Lambda^\infty,R))$.
\end{example}

\vspace{2ex}
\bibliographystyle{amsplain}
\makeatletter\renewcommand\@biblabel[1]{[#1]}\makeatother
\bibliography{references}
\newpage
\end{document}